\documentclass[12pt, leqno]{article}

\usepackage{fancyhdr}
\usepackage{enumitem} 
\newlist{condenum}{enumerate}{1}
\usepackage{hyperref} 

\usepackage{amsmath}
\allowdisplaybreaks
\usepackage{amsthm}
\usepackage{amssymb}
\usepackage{mathabx}
\usepackage{bbm} 
\usepackage{shuffle} 
\usepackage{stmaryrd} 

\usepackage{tikz}
\usetikzlibrary{patterns, arrows, matrix, positioning}

\usepackage[
	backend = biber,
	style=numeric,
	isbn=false, 
	url = false,
	isbn=false, 
	doi=false
]{biblatex}
\renewbibmacro{in:}{\ifentrytype{article}{}{\printtext{\bibstring{in}\intitlepunct}}}
\DeclareFieldFormat[article, book, incollection]{pages}{#1}
\newcounter{intro}

\newtheorem{thm}[equation]{Theorem}

\newtheorem{prop}[equation]{Proposition}
\newtheorem{lem}[equation]{Lemma}
\newtheorem{cor}[equation]{Corollary}
\theoremstyle{definition}
\newtheorem{ex}[equation]{Example}
\theoremstyle{remark}
\newtheorem*{rem}{Remark}

\numberwithin{equation}{section}

\newcommand{\FF}{\mathbb{F}}

\newcommand{\CC}{\mathbb{C}}

\def\mydefb#1{\expandafter\def\csname #1#1#1\endcsname{\mathcal{#1}}}
\def\mydefallb#1{\ifx#1\mydefallb\else\mydefb#1\expandafter\mydefallb\fi}
\mydefallb ABCDEFGHIJKLMNOPQRSTUVWXYZ\mydefallb

\newcommand{\calK}{\mathcal{K}}

\newcommand{\UT}{\mathrm{UT}}
\newcommand{\GL}{\mathrm{GL}}
\newcommand{\Mat}{\mathrm{Mat}}

\renewcommand{\epsilon}{\varepsilon}
\newcommand{\spanning}{\operatorname{-span}}
\renewcommand{\hom}{\operatorname{Hom}}
\renewcommand{\setminus}{-}

\let\shortto\to
\renewcommand{\to}{\longrightarrow}
\let\shortmapsto\mapsto
\renewcommand{\mapsto}{\longmapsto}

\renewcommand{\mapsfrom}{\longmapsfrom}

\newcommand{\cf}{\mathsf{cf}}
\newcommand{\scf}{\mathsf{scf}}
\newcommand{\cfUT}{\mathbf{cf}}
\newcommand{\ind}{\operatorname{Ind}}
\newcommand{\res}{\operatorname{Res}}
\newcommand{\infl}{\operatorname{Inf}}
\newcommand{\defl}{\operatorname{Def}}

\newcommand{\resf}{\operatorname{Resf}}

\newcommand{\asc}{\operatorname{asc}}

\newcommand{\st}{\operatorname{st}}
\newcommand{\cano}{\operatorname{cano}}
\newcommand{\w}[1]{w_{#1}}
\newcommand{\ordinalsum}{\mathchoice{
\tikz[baseline = -0.2em, inner sep = 2pt]{\node at (0, 0) (+) {$+$}; \node[above] at (+) {$\rightarrow$};}}{
\tikz[baseline = -0.2em, inner sep = 2pt]{\node at (0, 0) (+) {$+$}; \node[above] at (+) {$\rightarrow$};}}{
\tikz[baseline = -0.4*0.2em, inner sep = 1.2pt, scale = 0.3]{\node at (0, 0) (+) {$\scriptstyle +$}; \node[above] at (+) {$\scriptstyle\rightarrow$};}}{
\tikz[baseline = -0.25*0.2em, inner sep = 0.7pt]{\node at (0, 0) (+) {$\scriptscriptstyle +$}; \node[above] at (+) {$\scriptscriptstyle\rightarrow$};}}}
\newcommand{\ordinalsumsh}[1]{\mathchoice{
\tikz[baseline = -0.2em, inner sep = 2pt]{\node at (0, 0) (+) {$+$}; \node[above] at (+) {$\rightarrow$}; \node[below right] at (+) {$\scriptstyle #1$};}}{
\tikz[baseline = -0.2em, inner sep = 2pt]{\node at (0, 0) (+) {$+$}; \node[above] at (+) {$\rightarrow$}; \node[below right] at (+) {$\scriptstyle #1$};}}{
\tikz[baseline = -0.4*0.2em, inner sep = 1.2pt, scale = 0.3]{\node at (0, 0) (+) {$\scriptstyle +$}; \node[above] at (+) {$\scriptstyle\rightarrow$}; \node[below right] at (+) {$\scriptscriptstyle #1$};}}{
\tikz[baseline = -0.25*0.2em, inner sep = 0.7pt]{\node at (0, 0) (+) {$\scriptscriptstyle +$}; \node[above] at (+) {$\scriptscriptstyle\rightarrow$}; \draw (a) node[below right, scale = 0.7] {$\scriptscriptstyle #1$};}}}
\newcommand{\daggeraut}{\mathfrak{d}}
\newcommand{\titsprod}{\wedge}
\newcommand{\longw}{\tilde{w}}

\newcommand{\ASC}{\operatorname{Asc}}
\newcommand{\EQ}{\operatorname{Eq}}
\newcommand{\INV}{\operatorname{Inv}}
\newcommand{\natorder}{\mathcal{N}\mathcal{O}}

\renewcommand{\S}[1]{S_{#1}}

\newcommand{\UL}{\operatorname{UL}}
\newcommand{\UP}{\operatorname{UP}}
\newcommand{\UR}{\operatorname{UR}}
\renewcommand{\L}[2][n]{L_{#2}}
\renewcommand{\P}[2][n]{P_{#2}}
\newcommand{\R}[2][n]{R_{#2}}

\newcommand{\permchar}{\overline{\chi}}
\newcommand{\permind}{\overline{\delta}}

\begin{document}

\title{A $\GL(\FF_{q})$-compatible Hopf algebra of unitriangular class functions}
\author{Lucas Gagnon}
\date{\today}
\maketitle
\begin{abstract}
This paper constructs a novel Hopf algebra $\mathsf{cf}(\mathrm{UT}_{\bullet})$ on the class functions of the unipotent upper triangular groups $\mathrm{UT}_{n}(\mathbb{F}_{q})$ over a finite field.  
This construction is representation theoretic in nature and uses the machinery of Hopf monoids in the category of vector species.  
In contrast with a similar known construction, this Hopf algebra has the property that induction to the finite general linear group induces a homomorphism to Zelevinsky's Hopf algebra of $\mathrm{GL}_{n}(\mathbb{F}_{q})$ class functions.  
Furthermore, $\mathsf{cf}(\mathrm{UT}_{\bullet})$ contains a Hopf subalgebra which is isomorphic to a known combiantorial Hopf algebra, previously used to prove a conjecture about chromatic quasisymmetric functions. 
Some additional Hopf algebraic properties are also established.
\end{abstract}
{\small \textbf{Keywords} Hopf, unipotent, general linear, group, character, induction}

\section{Introduction}
\label{sec:intro}

A key structural result for finite groups of Lie type is the Bruhat decomposition
\[
G = \bigsqcup_{w \in W} UTwU,
\]
which builds $G$ from elements from the Weyl group $W$, a torus $T$, and a maximal unipotent subgroup $U$.  
This remains a useful heuristic when considering complex representations of $G$.  
Deligne--Lusztig theory uses the character theory of tori and Weyl groups to give a taxonomy for the irreducible characters of $G$ and key information like degree~\cite{lusztig}, but is not overly constructive and omits some character values.  
A complementary method is to induce representations directly from the maximal unipotent subgroup $U$.  
The latter approach is not generally well understood, but particularly nice examples like the Gelfand--Graev representation~\cite{GelGra} and generalizations~\cite{Kawanaka, Zel} have led to advances in the representation theory of $G$; relevant examples include the constructions in~\cite{An2, Ja} of the irreducible unipotent representations for the finite general linear groups.

This paper is motivated by the desire to better understand induction from $U$ to $G$, and makes significant progress for the case of the finite general linear groups.    
The character theory of $\GL_{n} = \GL_{n}(\FF_{q})$ can be understood as a Hopf algebra structure on the tower of class functions
 \[
\cf(\GL_{\bullet}) = \bigoplus_{n \ge 0} \cf(\GL_{n}),
\]
for the general linear groups as described by Zelevinsky in~\cite{Zel}; see Section~\ref{sec:GLalg}.  
In this context, the Deligne--Lusztig framework can be seen as a collection of homomorphisms between a Hopf algebra $\cf(S_{\bullet})$ of class functions for the symmetric groups and $\cf(\GL_{\bullet})$.  
The structure of $\cf(S_{\bullet})$ is exceptionally well understood via the theory of symmetric functions~\cite{Mac}, and with the aforementioned homomorphisms the combinatorics of symmetric functions also pervade the structure of $\cf(\GL_{\bullet})$.

The main results of this paper are given in Sections~\ref{sec:Hopfalgebra} and~\ref{sec:GL} and concern an analogous connection between $\GL_{n}$ and its maximal subgroup $\UT_{n}$ of unipotent upper triangular matrices.  
Section~\ref{sec:cfalgebra} defines a Hopf algebra structure on $\cf(\UT_{\bullet})$, the tower of class functions for the unipotent upper triangular groups.  
The operations in this Hopf algebra are given by representation theoretic maps which are reminiscent of the Harish-Chandra induction and restriction maps used for $\cf(\GL_{\bullet})$.    
Theorem~\ref{thm:inductionhomomorphism} demonstrates the utility of this Hopf structure by showing that the map
\[
\ind^{\GL}_{\UT} = \bigoplus_{n \ge 0} \ind^{\GL_{n}}_{\UT_{n}}: \cf(\UT_{\bullet}) \to \cf(\GL_{\bullet})
\]
is a Hopf algebra homomorphism.  
Taken together, these results parallel the Hopf algebraic description of Deligne--Lusztig theory for $\GL_{n}$ above, giving the `$U$' part of a kind of Bruhat decomposition for $\cf(\GL_{\bullet})$.  

It should be noted that the Hopf algebra in Theorem~\ref{thm:HopfAlgebra} is similar to another Hopf algebra structure on the space $\cf(\UT_{\bullet})$ implicit in the work of~\cite{AgBerTh} (see also~\cite{AgEtAl}).
The two are similar in many respects: both Hopf algebras come from Hopf monoids (see Sections~\ref{sec:Hopfprelims} and~\ref{sec:cfmonoid}) and the two are isomorphic as algebras.  
However, the coalgebra structures are distinct, so the two are not isomorphic: the Hopf algebra in~\cite{AgBerTh} is cocommutative but my Hopf algebra is not; see Theorem~\ref{thm:noncocomm}.  

Section~\ref{sec:subHopf} gives a few more results which incorporate the combinatorial representation theory of $\UT_{n}$ into the Hopf algebra $\cf(\UT_{\bullet})$.  
The past 15 years have seen a rapid development in the representation theory of unipotent groups, with focus shifting from the often inscrutable irreducible characters to supercharacters: families of well behaved but not necessarily irreducible characters.  
For $\UT_{n}$ and other maximal unipotent subgroups, these developments have taken a combinatorial turn in which families of supercharacters are indexed by set partitions and related objects~\cite{AlTh20, Andre, AnNe, An, Benedetti, Yan}.  
This paper will consider one such combinatorial supercharacter theory for $\UT_{n}$ from~\cite{AlTh20} with supercharacters are indexed by a family of Catalan objects known as natural unit interval orders.

Writing $\scf(\UT_{n}) \subseteq \cf(\UT_{n})$ for the subspace spanned by aforementioned Catalan supercharacters, Corollary~\ref{cor:scfsubhopf} shows that 
\[
\scf(\UT_{\bullet}) = \bigoplus_{n \ge 0} \scf(\UT_{n}),
\]
is a sub-Hopf algebra of $\cf(\UT_{\bullet})$, and computes explicit structure constants for $\scf(\UT_{\bullet})$.  
Even describing a basis of $\cf(\UT_{\bullet})$ is extremely difficult~\cite{GudEtAl}, so this provides a much-needed layer of computability to earlier results: for example, these structure constants are essential to the previously mentioned noncocommutativity result (Theorem~\ref{thm:noncocomm}).  

Remarkably, the sub-Hopf algebra $\scf(\UT_{\bullet})$ is already known: Theorem~\ref{thm:Dhom} shows that it is isomorphic to a Hopf algebra defined in~\cite{GP16} to study the chromatic quasisymmetric function of~\cite{ShWa}.  
This connection suggests that the combinatorics of the Hopf algebra in~\cite{GP16} should manifest in the representation theory of the finite general linear groups, which is confirmed in a companion work.  
The paper~\cite{Gagb} builds directly on the results stated here to give a combinatorial method for computing the map $\ind^{\GL}_{\UT}$.  
Applying this method to the characters in $\scf(\UT_{\bullet})$ gives rise to a family of $\GL_{n}$-modules which categorify the chromatic quasisymmetric function, and a second family of unipotent $\GL_{n}$-modules corresponding to the vertical strip LLT polynomial of~\cite{HHL}.  
These results are quite unexpected and demonstrate the potential of the Hopf algebraic connection constructed in this paper;~\cite{Gagb} will be the first of several papers using this framework to investigate the behavior of characters of $\UT_{n}$ under induction to $\GL_{n}$.  

The results in paper also set the stage for similar approaches for other finite groups of Lie type.
In particular, the class functions of the finite orthogonal and symplectic groups have a Hopf module-like structure over $\cf(\GL_{\bullet})$ which has been formalized in the ``2-compatible Hopf modules'' of~\cite{SHELLEY-ABRAHAMSON} and the ``twisted PSH-modules'' of~\cite{vanLeeuwen}.  
It should be possible to construct analogous module-like structures over $\cf(\UT_{\bullet})$ for the class functions of the unipotent orthogonal and symplectic groups in such a way that induction gives a structure preserving map.  
For other groups, more work is require to emulate the results of this paper, but these may nevertheless serve as a blueprint.  
Moreover,~\cite{Gagb} demonstrates that this process is likely worthwhile, especially in cases where unipotent representations are not fully understood.  

The remaining sections of the paper are organized as follows.  Sections~\ref{sec:Hopfprelims} and~\ref{sec:repprelims} cover the relevant preliminary material.  Sections~\ref{sec:cfmonoid} and~\ref{sec:Hopfalgebra} construct the Hopf algebra $\cf(\UT_{\bullet})$, along with a related Hopf monoid in the category of vector species, $\cfUT$.  Section~\ref{sec:GL} concerns the Hopf algebra $\cf(\GL_{\bullet})$ and the induction map $\ind^{\GL}_{\UT}$.  Finally, Section~\ref{sec:subHopf} presents the sub-Hopf algebra $\scf(\UT_{\bullet})$ and related results.

\paragraph{Acknowledgments}

Along with~\cite{Gagb}, this paper is part of a Ph.D.~thesis undertaken at the University of Colorado Boulder, and I am extremely grateful for the support and insights of my advisor Nathaniel Thiem throughout this process.  Thanks are also due to Richard Green and Nantel Bergeron for their helpful comments on previous drafts.

\section{Hopf monids}
\label{sec:Hopfprelims}
\newcommand{\set}{\mathrm{set}^{\times}}
\newcommand{\id}{\operatorname{id}}

A Hopf monoid is a kind of categorical generalization of a Hopf algebra.  
As it turns out, the structure maps for the Hopf algebra $\cf(\UT_{\bullet})$ are best described as part of a Hopf monoid; see  Secion~\ref{sec:cfmonoid} and Section~\ref{sec:Fockfunctors} for specifics.  
This section gives the key definitions and examples of Hopf monoids used in later results, following~\cite{AgMahshort} and~\cite{AgMahlong}.  

\subsection{Set compositions}
\label{sec:compositions}

Let $I$ be a finite set.  A \emph{set composition} of $I$ is a sequence of pairwise disjoint nonempty sets
\[
A = (A_{1}, \ldots, A_{\ell}) \qquad \text{with} \qquad I = \bigsqcup_{r = 1}^{\ell} A_{r}.
\]
Call each $A_{r}$ a \emph{part} of $A$, and write $A \vDash I$ and $\ell(A) = \ell$.  The empty set has one set composition with zero parts.

For disjoint sets $I$ and $J$, the \emph{concatenation} of set compositions $A \vDash I$ and $B \vDash J$ is
\[
A \cdot B = (A_{1}, \ldots, A_{\ell(A)}, B_{1}, \ldots, B_{\ell(B)}) \vDash I \sqcup J.
\]
On the other hand, the \emph{restriction} of a set composition $C \vDash I \sqcup J$ to $I$ is the set composition
\[
C|_{I} = (C_{1} \cap I, \ldots, C_{\ell(C)} \cap I)^{\sharp} \vDash I,
\]
where $\sharp$ denotes the removal of $\emptyset$ from a sequence; define $C|_{J}$ in the same way.

The set compositions of $I$ are ordered by refinement: $B$ \emph{refines} $A$ ($A$ \emph{coarsens} $B$) if 
\[
B = B|_{A_{1}} \cdot B|_{A_{2}} \cdots B|_{A_{\ell(A)}},
\]
so that each part of $A$ is the union of consecutive parts of $B$.  The \emph{Tits product} of set compositions $A, B \vDash I$ is the coarsest set partition refining $A$ with each part a subset of some part of $B$:
\[
A \titsprod B = B|_{A_{1}} \cdot B|_{A_{2}} \cdots B|_{A_{\ell(A)}}.
\]
This composition may not refine $B$, but $B \titsprod A$ has the same parts in a possibly different order so as to refine $B$.  If $B$ refines $A$, then $A \titsprod B = B =  B \titsprod A$. 

\subsection{Vector Species and Hopf Monoids}
\label{sec:hopfmonoiddiagrams}

This section defines connected Hopf monoids.  The connectedness assumption leads to a simplified but equivalent variant of the usual definition, as found in~\cite[Proposition 55]{AgMahshort}.

Let $\set$ denote the category of finite sets with morphisms given by bijections, and $\mathrm{vect}_{\CC}$ the space of finite-dimensional $\CC$-vector spaces with morphisms given by linear transformations.  A $\CC$-\emph{vector species} is a functor 
\[
\mathbf{a}: \mathrm{set}^{\times} \to \mathrm{vect}_{\CC}.
\]
Say that $\mathbf{a}$ is \emph{connected} if $\mathbf{a}[\emptyset] = \CC$.  For each set composition $A = (A_{1}, A_{2}, \ldots, A_{\ell})$ of a finite set $I$, write
\[
\mathbf{a}(A) = \mathbf{a}[A_{1}] \otimes \mathbf{a}[A_{2}] \otimes \cdots \otimes \mathbf{a}[A_{\ell}].
\]

A \emph{connected Hopf monoid} is a connected vector species $\mathbf{h}$ equipped with two families of linear maps for each $A \vDash I$ with $I \in \set$,
\[
\mu_{A}: \mathbf{h}(A) \to \mathbf{h}[I]
\qquad \text{and} \qquad
\Delta_{A}:  \mathbf{h}[I] \to \mathbf{h}(A),
\]
with which a number of diagrams must commute; they are as follows.

\paragraph{Associativity and coassociativity:}

For set compositions $A$ and $B$ of a finite set $I$ with $B$ refining $A$, the diagrams
\[
\begin{tikzpicture}[baseline = -1.1cm]
\draw (0, 0) node[] (B) {$\mathbf{h}(B)$};
\draw (5, 0) node[] (A)  {$\mathbf{h}(A)$};
\draw (5, -2) node[] (I) {$\mathbf{h}[I]$};
\draw[thick, ->] (B) -- node[below left] {$\mu_{B}$} (I);
\draw[thick, ->] (B) -- node[above] {$\mu_{B|_{A_{1}}} \otimes \cdots \otimes \mu_{B|_{A_{\ell}}}$} (A);
\draw[thick, ->] (A) -- node[right] {$\mu_{A}$} (I);
\end{tikzpicture}
\quad \text{and} \quad
\begin{tikzpicture}[baseline = -1.1cm]
\draw (5, -2) node[] (B) {$\mathbf{h}(B)$};
\draw (0, -2) node[] (A)  {$\mathbf{h}(A)$};
\draw (0, 0) node[] (I) {$\mathbf{h}[I]$};
\draw[thick, <-] (B) -- node[above] {$\Delta_{B}$} (I);
\draw[thick, <-] (B) -- node[below] {$\Delta_{B|_{A_{1}}} \otimes \cdots \otimes \Delta_{B|_{A_{\ell}}}$} (A);
\draw[thick, <-] (A) -- node[left] {$\Delta_{A}$} (I);
\end{tikzpicture}
\]
must commute.

\paragraph{Compatibility:}

For each finite set $I$ and set compositions $A, B \vDash I$, the diagram
\[
\begin{tikzpicture}
\draw (0, -2) node[] (AB) {$\mathbf{h}(A \titsprod B)$};
\draw (6, -2) node[] (BA) {$\mathbf{h}(B\titsprod  A)$};
\draw (0, 0) node[] (A) {$\mathbf{h}(A)$};
\draw (3, 0) node[] (I) {$\mathbf{h}[I]$};
\draw (6, 0) node[] (B) {$\mathbf{h}(B)$};
\draw[thick, ->] (AB)  -- node[above] {$\cong$} (BA);
\draw[thick, ->] (A)  -- node[left] {$\Delta_{B|_{A_{1}}} \otimes \cdots \otimes \Delta_{B|_{A_{\ell}}}$} (AB);
\draw[thick, <-] (B)  -- node[right] {$\mu_{B|_{A_{1}}} \otimes \cdots \otimes \mu_{B|_{A_{\ell}}}$} (BA);
\draw[thick, ->] (A)  -- node[above] {$\mu_{A}$} (I);
\draw[thick, ->] (I)  -- node[above] {$\Delta_{B}$} (B);
\end{tikzpicture}
\]
must commute, where $\cong$ denotes the isomorphism given by rearranging tensor factors.

\paragraph{Naturality:} For each bijection $\sigma: I \shortto J$ and set composition $A = (A_{1}, \ldots, A_{\ell})$ of $I$, let ${}^{\sigma}A = (\sigma(A_{1}), \ldots, \sigma(A_{\ell}))$; with this notation the diagrams
\[
\begin{tikzpicture}[baseline = -1.1cm]
\draw (0, 0) node[] (A) {$\mathbf{h}(A)$};
\draw (0, -2) node[] (I) {$\mathbf{h}[I]$};
\draw (5.5, 0) node[] (sA) {$\mathbf{h}({}^{\sigma}A)$};
\draw (5.5, -2) node[] (sI) {$\mathbf{h}[J]$};
\draw[thick, ->] (A) -- node[left] {$\mu_{A}$} (I);
\draw[thick, ->] (A) -- node[above] {$\mathbf{h}[\sigma|_{A_{1}}] \otimes \cdots \otimes \mathbf{h}[\sigma|_{A_{\ell}}]$} (sA);
\draw[thick, ->] (I) -- node[above] {$\mathbf{h}[\sigma]$} (sI);
\draw[thick, ->] (sA) -- node[right] {$\mu_{{}^{\sigma}A}$} (sI);
\end{tikzpicture}
\quad \text{and}\quad
\begin{tikzpicture}[baseline = -1.1cm]
\draw (0, -2) node[] (A) {$\mathbf{h}(A)$};
\draw (0, 0) node[] (I) {$\mathbf{h}[I]$};
\draw (5.5, -2) node[] (sA) {$\mathbf{h}({}^{\sigma}A)$};
\draw (5.5, 0) node[] (sI) {$\mathbf{h}[J]$};
\draw[thick, <-] (A) -- node[left] {$\Delta_{A}$} (I);
\draw[thick, ->] (A) -- node[above] {$\mathbf{h}[\sigma|_{A_{1}}] \otimes \cdots \otimes \mathbf{h}[\sigma|_{A_{\ell}}]$} (sA);
\draw[thick, ->] (I) -- node[above] {$\mathbf{h}[\sigma]$} (sI);
\draw[thick, <-] (sA) -- node[right] {$\Delta_{{}^{\sigma}A}$} (sI);
\end{tikzpicture}
\]
must commute.

\subsection{Example: partial and total orders}
\label{sec:monoidexamples}

This section presents two standard Hopf monoids from~\cite{AgMahshort}, giving concrete examples of the structures describe in Section~\ref{sec:hopfmonoiddiagrams}.  Much of the content will be used in later sections.

For a finite set $I$, let $\PPP\OOO[I]$ denote the set of partial orders of $I$, so that $\pi \in \PPP\OOO[I]$ is a reflexive, antisymmetric, and transitive relation $\pi \subseteq I \times I$.  
The \emph{species of partial orders} is the connected vector species $\mathbf{po}$ defined by
\[
\mathbf{po}[I] = \CC\spanning(\PPP\OOO[I])
\qquad\text{and}\qquad
\begin{array}{rcl}
\mathbf{po}[\sigma]: \mathbf{po}[I] & \to & \mathbf{po}[J] \\[0.5em]
\pi & \mapsto & {}^{\sigma} \pi 
\end{array}
\]
for each finite set $I$ and bijection $\sigma: I \shortto J$, where
\[
{}^{\sigma} \pi = \big\{ \big( \sigma(i), \sigma(j) \big) \;\big|\; (i, j) \in \pi \big\}.
\]

Some additional notation is required to define a Hopf monoid structure on $\mathbf{po}$.  For a finite set $I$, a partial order $\pi \in \PPP\OOO[I]$, and a subset $J \subseteq I$, define the \emph{restriction} of $\pi$ to be
\[
\pi|_{J} = \big\{(i, j) \in \pi \;\big|\; i, j \in J \big\} \in \PPP\OOO[J].
\]
Also, for disjoint finite sets $I_{1}$ and $I_{2}$, define the \emph{ordinal sum} of a partial orders $\pi_{1} \in \PPP\OOO[I_{1}]$, $\pi_{2} \in \PPP\OOO[I_{2}]$ to be
\[
\pi_{1} \ordinalsum \pi_{2} = \pi_{1} \sqcup \pi_{2} \sqcup I_{1} \times I_{2},
\]
which is an element of $\PPP\OOO[I_{1} \sqcup I_{2}]$; this operation is associative.  The ordinal sum operation is sometimes written as $\oplus$, but this symbol has another meaning in this paper. 

\begin{ex}
Let $I_{1} = \{a, b, c\}$, $I_{2} = \{d\}$.  Then we have
\[
\begin{tikzpicture}[scale = 0.65, baseline = 0.65*0.8cm]
\draw (0, 0.2) node[inner sep = 0.05cm] (1) {$\scriptstyle a$};
\draw (1, 0.2) node[inner sep = 0.05cm] (2) {$\scriptstyle b$};
\draw (1, 1) node[inner sep = 0.05cm] (3) {$\scriptstyle c$};
\draw (0.5, 1.8) node[inner sep = 0.05cm] (4) {$\scriptstyle d$};
\draw (1) -- (4);
\draw (2) -- (3);
\draw (3) -- (4);
\end{tikzpicture} \;\bigg|_{I_{1}} = \begin{tikzpicture}[scale = 0.65, baseline = 0.65*0.4cm]
\draw (0, 0.6) node[inner sep = 0.05cm] (1) {$\scriptstyle a$};
\draw (0.75, 0.2) node[inner sep = 0.05cm] (2) {$\scriptstyle b$};
\draw (0.75, 1) node[inner sep = 0.05cm] (3) {$\scriptstyle c$};
\draw (2) -- (3);
\end{tikzpicture} 
\qquad \text{and}\qquad
\left(\begin{tikzpicture}[scale = 0.65, baseline = 0.65*-0.15cm]
\draw (0, 0) node (1) {$\scriptstyle a$};
\draw (0.5, 0.05) node (2) {$\scriptstyle b$};
\draw (1, 0) node (3) {$\scriptstyle c$};
\end{tikzpicture}\right)
\ordinalsum
\left(\begin{tikzpicture}[scale = 0.65, baseline = -0.65*0.2cm]
\draw (0, 0) node[inner sep = 0.05cm] (4) {$\scriptstyle d$};
\end{tikzpicture}\right)
=
\begin{tikzpicture}[scale = 0.65, baseline = 0.65*0.3cm]
\draw (0, 0) node[inner sep = 0.05cm] (1) {$\scriptstyle a$};
\draw (0.5, 0.05) node[inner sep = 0.05cm] (2) {$\scriptstyle b$};
\draw (1, 0) node[inner sep = 0.05cm] (3) {$\scriptstyle c$};
\draw (0.5, 1) node[inner sep = 0.05cm] (4) {$\scriptstyle d$};
\draw (1) -- (4);
\draw (2) -- (4);
\draw (3) -- (4);
\end{tikzpicture}.
\]
\end{ex}

\begin{prop}[{\cite[Proposition 3]{An}}]
\label{prop:posetmonoid}
The species $\mathbf{po}$ is a connected Hopf monoid; for each set composition $A = (A_{1}, \ldots, A_{\ell})$ of each finite set $I$ the product map is given by
\[
\mu_{A}(\pi_{1} \otimes \cdots \otimes \pi_{\ell} ) =  \pi_{1} \ordinalsum \cdots \ordinalsum \pi_{\ell} 
\]
for $\pi_{i} \in \PPP\OOO[A_{i}]$, $1 \le i \le \ell$, and for $\pi \in \PPP\OOO[I]$, the coproduct map is given by
\[
\Delta_{A}(\pi) = \pi|_{A_{1}} \otimes \cdots \otimes \pi|_{A_{\ell}}.
\]
\end{prop}

Another Hopf monoid can be defined using subspaces of $\mathbf{po}$ (technically speaking, this is a \emph{sub-Hopf monoid}).    A \emph{total order} of the finite set $I$ is a partial order $\tau$ of $I$ which is not properly contained in any other partial order of $I$; equivalently these are the partial orders with $\binom{|I|+1}{2}$ relations.  Let
\[
\TTT[I] = \{\text{total orders of $I$}\}.
\]
I will write elements of $\TTT[I]$ as their maximal chain, so that 
\[
a_{1} < a_{2} < \cdots < a_{|I|}
\qquad\text{refers to}\qquad
\{(a_{i}, a_{j}) \;|\; 1 \le i \le j \le |I|\} \in \TTT[I].
\]  
Define a species $\mathbf{t}$ by
\[
\mathbf{t}[I] = \CC\spanning(\TTT[I])
\qquad\text{and}\qquad
\mathbf{t}[\sigma] = \mathbf{po}[\sigma]\big|_{\mathbf{t}[I]}
\]
for each finite set $I$ and bijection $\sigma: I \shortto J$.

\begin{prop}[{\cite[Example 8.16]{AgMahlong}}]
With the same product and coproduct maps as $\mathbf{po}$, the species $\mathbf{t}$ is a Hopf monoid.
\end{prop}

\begin{ex}
Let $I_{1} = \{a, b, d\}$, $I_{2} = \{c\}$.  Then
\[
\big( a < b < c < d \big)|_{I_{1}} = a < b < d
\qquad \text{and}\qquad
\big( a < b < d\big)
\ordinalsum
\big(c \big)
=
a < b < d < c.
\]
\end{ex}

\section{Linear algebraic groups and class functions}
\label{sec:repprelims}

The results in this paper concern certain linear algebraic groups, their class functions, and maps between spaces of class functions; each of these objects are defined in this section.
The finite general linear groups and some important unipotent subgroups are defined in Section~\ref{sec:matrixgroups}, after which Section~\ref{sec:grouphomomorphisms} gives a number of homomorphisms between these groups.  
Finally, Section~\ref{sec:classfunctions} introduces class functions and describes the functoriality of this construction with respect to the aforementioned homomorphisms.

\subsection{Linear algebraic groups}
\label{sec:matrixgroups}

For each finite set $I$, let $\Mat(I)$ denote the set of $I$-indexed matrices over $\FF_{q}$, so that
\[
\Mat(I) = \{(x_{i, j})_{i, j \in I} \;|\; x_{i, j} \in \FF_{q}\}.
\]
When depicting elements $x \in \Mat(I)$, the convention will be to fix a total order of $I$ and to write $x$ as an $|I| \times |I|$ array with row and column indices appearing in this over.

\begin{ex}
Let $I = \{a, b, c, d\}$.  Ordering this set alphabetically, 
\vspace{-0.5em}
\[
\Mat(\{a, b, c, d\}) = 
\begin{tikzpicture}[scale = 0.45, baseline = -0.45*2.2 cm]
\draw (1 - 0.5, 0.65) node {$\scriptstyle a$};
\draw (2 - 0.5, 0.65) node {$\scriptstyle b$};
\draw (3 - 0.5, 0.65) node {$\scriptstyle c$};
\draw (4 - 0.5, 0.65) node {$\scriptstyle d$};
\draw (-0.65, -1 + 0.5) node {$\scriptstyle a$};
\draw (-0.65, -2 + 0.5) node {$\scriptstyle b$};
\draw (-0.65, -3 + 0.5) node {$\scriptstyle c$};
\draw (-0.65, -4 + 0.5) node {$\scriptstyle d$};
\foreach \x in {1, 2, 3, 4}{\foreach \y in {1, 2, 3, 4}{\draw (\x - 0.5, -\y + 0.5) node {$\ast$};}}
\draw (0.2, 0.085) -- (-0.15, 0.085);
\draw[very thick] (-0.1, 0.1) -- (-0.1, -4.1);
\draw (-0.15, -4.085) -- (0.2, -4.085);
\draw (3.8, 0.085) -- (4.15, 0.085);
\draw[very thick] (4.1, 0.1) -- (4.1, -4.1);
\draw (4.15, -4.085) -- (3.8, -4.085);
\end{tikzpicture}.
\]
\end{ex}

The general linear group in $\Mat(I)$ is the multiplicative group of full rank elements
\[
\GL(I) = \{x \in \Mat(I) \;|\; \operatorname{rank}(x) = |I|\},
\]
with identity $1_{I}$, the $I$-indexed identity matrix.  
Say that an element $g \in \GL(I)$ is \emph{unipotent} if $g - 1_{I}$ is a nilpotent matrix.  Each partial order $\pi \in \PPP\OOO[I]$ determines a subgroup of $\GL(I)$ comprising entirely unipotent elements; the \emph{pattern subgroup} indexed by $\pi$ is
\[
\UT(\pi) = \{X \in \GL_{I} \;|\; \text{$(X - 1_{I})_{i, j} \neq 0$ only if $(i, j) \in \tau$}\},
\]
and has order $|\UT(\pi)| = q^{|\pi| - |I|}$.  

\begin{ex}
\label{ex:unipotentgroups}
Let $I = \{a, b, c, d\}$, $\tau = a < b < c < d$, and $\pi =
 \begin{tikzpicture}[scale = 0.65, baseline = 0.65*0.3cm]
\draw (0, 0) node[inner sep = 0.05cm] (a) {$\scriptstyle a$};
\draw (1, 0) node[inner sep = 0.05cm] (b) {$\scriptstyle b$};
\draw (1.5, 0.5) node[inner sep = 0.05cm] (c) {$\scriptstyle c$};
\draw (0.5, 1) node[inner sep = 0.05cm] (d) {$\scriptstyle d$};
\draw (a) -- (d);
\draw (b) -- (d);
\end{tikzpicture}$.  Then 
\vspace{-1em}
\[
\UT(\tau) = \begin{tikzpicture}[scale = 0.45, baseline = -0.45*2.2 cm]
\draw (1 - 0.5, 0.65) node {$\scriptstyle a$};
\draw (2 - 0.5, 0.65) node {$\scriptstyle b$};
\draw (3 - 0.5, 0.65) node {$\scriptstyle c$};
\draw (4 - 0.5, 0.65) node {$\scriptstyle d$};
\draw (-0.65, -1 + 0.5) node {$\scriptstyle a$};
\draw (-0.65, -2 + 0.5) node {$\scriptstyle b$};
\draw (-0.65, -3 + 0.5) node {$\scriptstyle c$};
\draw (-0.65, -4 + 0.5) node {$\scriptstyle d$};
\foreach \x in {1, 2, 3, 4}{\draw (\x - 0.5, -\x + 0.5) node {$1$};}
\foreach \x in {1}{\foreach \y in {2, 3, 4}{\draw (\x - 0.5, -\y + 0.5) node {$0$};}}
\foreach \x in {2}{\foreach \y in {3, 4}{\draw (\x - 0.5, -\y + 0.5) node {$0$};}}
\foreach \x in {3}{\foreach \y in {4}{\draw (\x - 0.5, -\y + 0.5) node {$0$};}}
\foreach \x in {2, 3, 4}{\foreach \y in {1}{\draw (\x - 0.5, -\y + 0.5) node {$\ast$};}}
\foreach \x in {3, 4}{\foreach \y in {2}{\draw (\x - 0.5, -\y + 0.5) node {$\ast$};}}
\foreach \x in {4}{\foreach \y in {3}{\draw (\x - 0.5, -\y + 0.5) node {$\ast$};}}
\draw (0.2, 0.085) -- (-0.15, 0.085);
\draw[very thick] (-0.1, 0.1) -- (-0.1, -4.1);
\draw (-0.15, -4.085) -- (0.2, -4.085);
\draw (3.8, 0.085) -- (4.15, 0.085);
\draw[very thick] (4.1, 0.1) -- (4.1, -4.1);
\draw (4.15, -4.085) -- (3.8, -4.085);
\end{tikzpicture}
\qquad\text{and}\qquad
\UT(\pi) = \begin{tikzpicture}[scale = 0.45, baseline = -0.45*2.2 cm]
\draw (1 - 0.5, 0.65) node {$\scriptstyle a$};
\draw (2 - 0.5, 0.65) node {$\scriptstyle b$};
\draw (3 - 0.5, 0.65) node {$\scriptstyle c$};
\draw (4 - 0.5, 0.65) node {$\scriptstyle d$};
\draw (-0.65, -1 + 0.5) node {$\scriptstyle a$};
\draw (-0.65, -2 + 0.5) node {$\scriptstyle b$};
\draw (-0.65, -3 + 0.5) node {$\scriptstyle c$};
\draw (-0.65, -4 + 0.5) node {$\scriptstyle d$};
\foreach \x in {1, 2, 3, 4}{\draw (\x - 0.5, -\x + 0.5) node {$1$};}
\foreach \x in {1}{\foreach \y in {2, 3, 4}{\draw (\x - 0.5, -\y + 0.5) node {$0$};}}
\foreach \x in {2}{\foreach \y in {3, 4}{\draw (\x - 0.5, -\y + 0.5) node {$0$};}}
\foreach \x in {3}{\foreach \y in {4}{\draw (\x - 0.5, -\y + 0.5) node {$0$};}}
\foreach \x in {2, 3}{\foreach \y in {1}{\draw (\x - 0.5, -\y + 0.5) node {$0$};}}
\foreach \x in {3}{\foreach \y in {2}{\draw (\x - 0.5, -\y + 0.5) node {$0$};}}
\foreach \x in {4}{\foreach \y in {3}{\draw (\x - 0.5, -\y + 0.5) node {$0$};}}
\foreach \x in {4}{\foreach \y in {1}{\draw (\x - 0.5, -\y + 0.5) node {$\ast$};}}
\foreach \x in {4}{\foreach \y in {2}{\draw (\x - 0.5, -\y + 0.5) node {$\ast$};}}
\draw (0.2, 0.085) -- (-0.15, 0.085);
\draw[very thick] (-0.1, 0.1) -- (-0.1, -4.1);
\draw (-0.15, -4.085) -- (0.2, -4.085);
\draw (3.8, 0.085) -- (4.15, 0.085);
\draw[very thick] (4.1, 0.1) -- (4.1, -4.1);
\draw (4.15, -4.085) -- (3.8, -4.085);
\end{tikzpicture}.
\]
\end{ex}

The maximal pattern subgroups are indexed by total orders; for $\tau \in \TTT[I]$, the \emph{unipotent $\tau$-upper triangular group} is $\UT(\tau)$.  When $\tau$ agrees with the row and column ordering, $\UT(\tau)$ is actually the the set of unipotent upper triangular matrices.  
For $\tau \in \TTT[I]$, the unipotent $\tau$-upper triangular group $\UT(\tau)$ contains each pattern subgroup $\UT(\pi)$ for which $\pi \subseteq \tau$, so that the order $\pi$ extends to the total order $\tau$.

\begin{lem}[\cite{Mar} Lemma 4.1]
\label{lem:patterngroupnormalcy}
Let $\tau$ be a total order of a finite set $I$ and $\pi$ a partial order of $I$ which is extended by $\tau$.  Then $\UT(\pi) \trianglelefteq \UT(\tau)$ if and only if 
\begin{equation}
\label{eq:naturalunitintervalorder}
\text{for each $(j, k) \in \pi$ with $j \neq k$}: \qquad \{(i, l) \;|\;\text{$(i, j)$ and $(k, l) \in \tau$}\} \subseteq \pi.
\end{equation}
\end{lem}

Under the assumption that $\tau$ is the canonical order on a set of the form $\{1, \ldots, n\}$, the meaning of~\eqref{eq:naturalunitintervalorder} is considered in greater depth in Section~\ref{sec:scf}.  

\subsection{Homomorphisms}
\label{sec:grouphomomorphisms}

First, take $I$ to be a finite set.  Each bijection $\sigma: I \shortto J$ induces an isomorphism,
\[
\begin{array}{rccc}
\sigma\colon & \GL(I) & \to & \GL(J) \\
& X & \mapsto & {}^{\sigma}X
\end{array}
\qquad\text{where ${}^{\sigma}X = (X_{\sigma(i), \sigma(j)})_{i, j \in I}$}.
\]
For each partial order $\pi \in \PPP\OOO[I]$, this map restricts to an isomorphism 
\[
\sigma\colon \UT(\pi) \cong \UT({}^{\sigma}\pi).
\]

\begin{rem}
Taking $J = I$ above, there is a natural identification of $\sigma$ with a permutation matrix in $\GL(I)$.  In this context, the isomorphisms $\sigma: \GL(I) \shortto \GL(I)$ is equivalent to the inner automorphism of $\GL(I)$ given by conjugation by $\sigma$.
\end{rem}

Now consider disjoint finite sets $I_{1}, \ldots, I_{\ell}$.  There is an injection
\[
\begin{array}{ccc}
\Mat(I_{1}) \times \cdots \times \Mat(I_{\ell}) & \to & \Mat(I_{1}  \sqcup \cdots \sqcup I_{\ell}) \\[0.25em]
\big(X^{(1)}, \ldots, X^{(\ell)}\big) & \mapsto & X^{(1)} \oplus \cdots \oplus X^{(\ell)},
\end{array}
\]
where
\[
\big(X^{(1)} \oplus \cdots \oplus X^{(\ell)}\big)_{i, j} = \begin{cases}  X^{(r)}_{i, j} & \text{if $i, j \in I_{r}$, $1 \le r \le \ell$,} \\ 0 & \text{otherwise.} \end{cases}
\]
This map is additive with respect to rank and preserves matrix multiplication, so
\[
\GL(I_{1})  \times \cdots \times \GL(I_{\ell}) \cong \GL(I_{1}) \oplus \cdots \oplus \GL(I_{\ell}) \subseteq \GL\big(I_{1} \sqcup I_{2} \sqcup \cdots \sqcup I_{\ell}\big).
\]

The relations above also restrict to pattern subgroups.  For partial orders $\pi_{1} \in \PPP\OOO[I_{1}], \ldots$, $\pi_{\ell} \in \PPP\OOO[I_{\ell}]$,
\begin{equation}
\label{eq:dirsumiso}
\UT(\pi_{1}) \times \cdots \times \UT(\pi_{\ell}) \cong \UT(\pi_{1}) \oplus \cdots \oplus \UT(\pi_{\ell}), 
\end{equation}
which is also equal to $\UT(\pi_{1} \sqcup \cdots \sqcup \pi_{\ell})$.  As such the direct sum $\UT(\pi_{1}) \oplus \cdots \oplus \UT(\pi_{\ell})$ is contained in $\UT(\phi)$ for any $\phi \in \PPP\OOO[I_{1} \sqcup \cdots \sqcup I_{\ell}]$ with $\pi_{i} \subseteq \phi|_{I_{r}}$ for each $1 \le r \le \ell$.  The ordinal sum $\phi = \pi_{1} \ordinalsum \cdots \ordinalsum \pi_{\ell}$ satisfies this condition, so
\[
\UT(\pi_{1}) \oplus \cdots \oplus \UT(\pi_{\ell}) \subseteq \UT(\pi_{1} \ordinalsum \cdots \ordinalsum \pi_{\ell}).
\]

\subsection{Class functions and representation theoretic maps}
\label{sec:classfunctions}

For a finite group $G$, the \emph{space of class functions} of $G$ is
\[
\cf(G) = \CC\spanning\{ \psi: G \shortto \CC \;|\; \text{$\psi(g) = \psi(hgh^{-1})$ for all $g, h \in G$}\},
\]
equipped with the standard inner product $\langle \cdot, \cdot \rangle: \cf(G) \otimes \cf(G) \shortto \CC$.

Each homomorphism of groups $f: K \shortto G$ induces a $\CC$-linear map
\[
\begin{array}{rccc}
f^{\ast}\colon & \cf(G) &\to& \cf(K) \\
& \psi & \mapsto & \psi \circ f.
\end{array}
\]
As an example, take $f: \UT(\pi) \times \UT(\rho) \shortto \UT(\pi) \oplus \UT(\rho)$ to be the isomorphism in Equation~\eqref{eq:dirsumiso}; then $f^{\ast}$ gives an isomorphism
\[
\cf\big(\UT(\pi) \oplus \UT(\rho)\big) \cong \cf(\UT(\pi)) \otimes \cf(\UT(\rho)),
\]
where $\cf(\UT(\pi)) \otimes \cf(\UT(\rho))$ and $\cf(\UT(\pi) \times \UT(\rho))$ are canonically identified in accordance with standard practice.

Two more special cases are extremely important in subsequent sections.  If $K \subseteq G$ and $f\colon K \shortto G$ is the inclusion map, then $f^{\ast}$ is the \emph{restriction} map
\[
\res^{G}_{K} \colon \cf(G) \to \cf(K).
\]
With respect to $\langle \cdot, \cdot \rangle$, the adjoint of $\res^{K}_{G}$ is \emph{induction},
\[
\begin{array}{rccc}
\ind^{G}_{K} \colon & \cf(K) & \to & \cf(G) \\
& \psi & \mapsto & \Bigg( g \shortmapsto \displaystyle\frac{1}{|G|} \sum_{\substack{x \in G \\ xgx^{-1} \in K }} \psi(xgx^{-1}) \Bigg).
\end{array}
\]

If $G = K/H$ and $f \colon K \shortto G$ is the canonical projection map, then $f^{\ast}$ is \emph{inflation}
\[
\infl^{K}_{G}\colon \cf(G) \to \cf(K),
\]
and the adjoint of $\infl^{K}_{G}$ is \emph{deflation},
\[
\begin{array}{rccc}
\defl^{K}_{G} \colon & \cf(K) & \to & \cf(G) \\
& \psi & \mapsto & \bigg( gH \shortmapsto \displaystyle\frac{1}{|H|} \sum_{h \in H} \psi(gh) \bigg)
\end{array}.
\]

\section{The Hopf monoid $\cfUT$}
\label{sec:cfmonoid}

This section will define a Hopf monoid $\cfUT$ which lies above the titular Hopf algebra $\cf(\UT_{\bullet})$.  
The underlying structure maps are easier to work with in the context of Hopf monoids, so this significantly simplifies many proofs in later sections.   
However, the details of $\cfUT$ are still quite complicated, so this section will begin with an overview of the construction. 

Recall the content of Sections~\ref{sec:Hopfprelims} and~\ref{sec:repprelims}.  The underlying vector species for $\cfUT$ was first constructed in~\cite{AgBerTh}, and has
\[
\cfUT[I] = \bigoplus_{\tau \in \TTT[I]} \cf(\UT(\tau)) \qquad \text{and} \qquad \cfUT[\sigma] = \bigoplus_{\tau \in \TTT[I]} \bigg( \cf(\UT(\tau)) \xrightarrow{\: (\sigma^{-1})^{\ast} \;} \cf(\UT({}^{\sigma}\tau)) \bigg)
\]
for each finite set $I$ and bijection $\sigma: I \shortto J$, where $\TTT[I]$ denotes the set of total orders as in Section~\ref{sec:monoidexamples}.
For any finite set $I$ and set composition $A = (A_{1}, \ldots, A_{\ell})$ of $I$, Equation~\ref{eq:dirsumiso} shows that the space $\cfUT(A) = \cfUT[A_{1}] \otimes \cdots \otimes \cfUT[A_{\ell}]$ can be canonically identified as
\begin{equation}
\label{eq:cfA}
\cfUT(A) = \bigoplus_{\substack{\tau_{i} \in \TTT[A_{i}] \\ 1 \le i \le \ell }} \cf\big(\UT(\tau_{1}) \oplus \cdots \oplus \UT(\tau_{\ell})\big).
\end{equation}
This identification makes it possible to construct functorial maps between $\cfUT[I]$ and $\cfUT(A)$ as follows.  Each $\UT(\tau_{1}) \oplus \cdots \oplus \UT(\tau_{\ell})$ is a subgroup of $\UT(\phi)$ for each total order $\phi \in \TTT[I]$ with $\phi|_{A_{i}} = \tau_{i}$, and Section~\ref{sec:CFalgprelims} shows that this subgroup has a normal complement in a larger ``parabolic'' subgroup $\UP(\phi, A)$ of $\UT(\phi)$.  Then, Section~\ref{sec:downmaps} defines a coproduct for $\cfUT$ by
\[
\Delta_{A} = \bigoplus_{\phi \in \TTT[I]} \bigg( \cf(\UT(\phi)) \xrightarrow{\;\defl \circ \res \;} \cf\big(\UT(\phi|_{A_{1}}) \oplus \cdots \oplus \UT(\phi|_{A_{\ell}})\big) \bigg),
\]
with the deflation--restriction composition factoring through the parabolic subgroup.  Following this, Section~\ref{sec:upmaps} considers the special case in which $\phi = \tau_{1} \ordinalsum \cdots \ordinalsum \tau_{\ell}$ and shows that here the parabolic subgroup is the full unipotent upper triangular group, $\UP(\phi, A) = \UT(\sigma)$.  The product for $\cfUT$ is defined by
\[
\mu_{A} = \bigoplus_{\substack{\tau_{i} \in \TTT[A_{i}] \\ 1 \le i \le \ell}} \bigg( \cf\big(\UT(\tau_{1}) \oplus \cdots \oplus \UT(\tau_{\ell})\big) \xrightarrow{\;\infl\;} \cf(\UT(\tau_{1} \ordinalsum \cdots \ordinalsum \tau_{\ell})) \bigg),
\]
which in this case is the adjoint of the deflation--restriction map used in the coproduct.

\begin{thm}
\label{thm:CFhopfmonoid}
With the maps $\mu$ and $\Delta$ as above, $\cfUT$ is a connected Hopf monoid.
\end{thm}
\begin{proof}
Taking the Hopf monoid $\mathbf{t}$ defined in Section~\ref{sec:monoidexamples} into consideration, it is sufficient to show that each Hopf monoid axiom from Section~\ref{sec:hopfmonoiddiagrams} holds the for individual summands in $\mu_{A}$ and $\Delta_{A}$.  Coassociativity and associativity are established in this manner by Propositions~\ref{prop:downtransitivity} and~\ref{prop:upwardtransitivity}, respectively.  In the same way, compatibility is established by Proposition~\ref{prop:compatibility}, and naturality by Proposition~\ref{prop:naturality}.
\end{proof}

This section will fill in the remaining details in the story above.  
Sections~\ref{sec:CFcombprelims} and~\ref{sec:CFalgprelims} set up the machinery necessary for later sections.  
Then, Section~\ref{sec:downmaps} defines the maps used in the coproduct and establishes coassociativity, and Section~\ref{sec:upmaps} defines the map used in the product and establishes associativity. Finally, Section~\ref{sec:compatibility} establishes the remaining axioms.  

\paragraph{Convention:} Outside of numbered results and examples, $I$ and $\tau \in \TTT[I]$ refer to an arbitrary fixed finite set and total order.

\subsection{Combinatorial underpinnings}
\label{sec:CFcombprelims}

The maps used in the product and coproduct of $\cfUT$ involve certain pattern groups constructed from set compositions.  This section gives the combinatorics of this construction.

Each set composition $A \vDash I$ partitions the set $I \times I$ into three disjoint subsets:
\begin{align*}
\textbf{$A$-ascents:} &\qquad \ASC(A) = \bigsqcup_{1 \le r < s \le \ell(A)} A_{r} \times A_{s}, \\[0.75em]
\textbf{$A$-equalities:} &\qquad \EQ(A) = \bigsqcup_{1 \le r \le \ell(A)} A_{r} \times A_{r},\;\text{and}\\[0.75em]
\textbf{$A$-inversions:} &\qquad \INV(A) = \bigsqcup_{1 \le r < s \le \ell(A)} A_{s} \times A_{r}.
\end{align*}

Each of there sets can be visualized as a collection of cells in an $|I| \times |I|$ array with row and column labels in $I$, ordered by $\tau$: $\ASC(A)$ is the intersection of $A_{r}$-labelled rows with $A_{s}$-labelled columns for $r < s$, and $\EQ(A)$ is the intersection of the rows and columns with labels that both belong to the same part of $A$, and $\INV(A)$ are the remaining cells.  

\begin{ex}
\label{ex:asceqinv}
Let $I = \{a, b, c, d\}$ and $\tau = a < b < c < d$.  Two examples of set compositions of $I$ and their corresponding ascent, equality, and inversion sets are shown below:
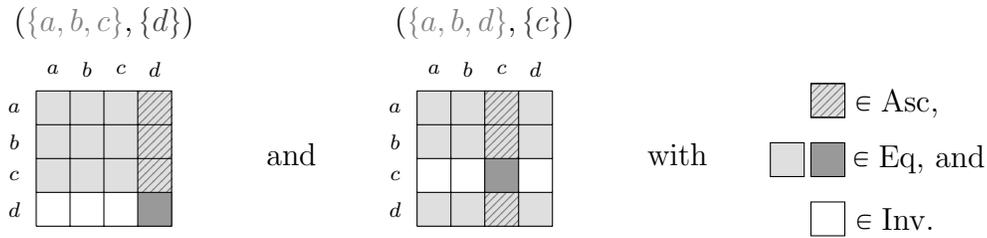
\[
\begin{tikzpicture}[scale = 0.45, baseline = 0.45*-2.2cm]
\node at (2, 2) {$({\color{gray} \{ a, b, c \}}, {\color{gray!150}\{d\}} )$};
\draw (1 - 0.5, 0.65) node {$\scriptstyle a$};
\draw (2 - 0.5, 0.65) node {$\scriptstyle b$};
\draw (3 - 0.5, 0.65) node {$\scriptstyle c$};
\draw (4 - 0.5, 0.65) node {$\scriptstyle d$};
\draw (-0.65, -1 + 0.5) node {$\scriptstyle a$};
\draw (-0.65, -2 + 0.5) node {$\scriptstyle b$};
\draw (-0.65, -3 + 0.5) node {$\scriptstyle c$};
\draw (-0.65, -4 + 0.5) node {$\scriptstyle d$};
\foreach \x in {1, 2, 3, 4}{\foreach \y in {1, 2, 3}{\path[fill = gray!25] (\x - 1, -\y + 1) -- (\x, -\y + 1) -- (\x, -\y) -- (\x - 1, -\y) -- cycle;}}
\foreach \x in {4}{\foreach \y in {4}{\path[fill = gray!80] (\x - 1, -\y + 1) -- (\x, -\y + 1) -- (\x, -\y) -- (\x - 1, -\y) -- cycle;}}
\foreach \x in {4}{\foreach \y in {1, 2, 3}{\path[pattern = north east lines, pattern color = gray] (\x - 1, -\y + 1) -- (\x, -\y + 1) -- (\x, -\y) -- (\x - 1, -\y) -- cycle;}}
\draw (0, 0) grid (4, -4);
\end{tikzpicture}
\qquad\text{and}\qquad
\begin{tikzpicture}[scale = 0.45, baseline = 0.45*-2.2cm]
\node at (2, 2) {$({\color{gray} \{ a, b, d \}}, {\color{gray!150}\{c\}} )$};
\draw (1 - 0.5, 0.65) node {$\scriptstyle a$};
\draw (2 - 0.5, 0.65) node {$\scriptstyle b$};
\draw (3 - 0.5, 0.65) node {$\scriptstyle c$};
\draw (4 - 0.5, 0.65) node {$\scriptstyle d$};
\draw (-0.65, -1 + 0.5) node {$\scriptstyle a$};
\draw (-0.65, -2 + 0.5) node {$\scriptstyle b$};
\draw (-0.65, -3 + 0.5) node {$\scriptstyle c$};
\draw (-0.65, -4 + 0.5) node {$\scriptstyle d$};
\foreach \x in {1, 2, 3, 4}{\foreach \y in {1, 2, 4}{\path[fill = gray!25] (\x - 1, -\y + 1) -- (\x, -\y + 1) -- (\x, -\y) -- (\x - 1, -\y) -- cycle;}}
\foreach \x in {3}{\foreach \y in {3}{\path[fill = gray!80] (\x - 1, -\y + 1) -- (\x, -\y + 1) -- (\x, -\y) -- (\x - 1, -\y) -- cycle;}}
\foreach \x in {3}{\foreach \y in {1, 2, 4}{\path[pattern = north east lines, pattern color = gray] (\x - 1, -\y + 1) -- (\x, -\y + 1) -- (\x, -\y) -- (\x - 1, -\y) -- cycle;}}
\draw (0, 0) grid (4, -4);
\end{tikzpicture}
\qquad\text{with}\qquad
\begin{tikzpicture}[scale = 0.45, baseline = -0.3cm]
\begin{scope}[yshift = 1.75cm]
\draw[fill = gray!25] (0, 0) -- (1, 0) -- (1, -1) -- (0, -1) -- cycle;
\draw[pattern = north east lines, pattern color = gray] (0, 0) -- (1, 0) -- (1, -1) -- (0, -1) -- cycle;
\draw (2.55, -0.6) node {$\in \ASC$,};
\end{scope}
\draw[fill = gray!80] (0, 0) -- (1, 0) -- (1, -1) -- (0, -1) -- cycle;
\draw[fill = gray!25] (-1.2, 0) -- (-0.2, 0) -- (-0.2, -1) -- (-1.2, -1) -- cycle;
\draw (3.2, -0.5) node {$\in \EQ$, and};
\begin{scope}[yshift = -1.75cm]
\draw[fill = white] (0, 0) -- (1, 0) -- (1, -1) -- (0, -1) -- cycle;
\draw (2.5, -0.5) node {$\in \INV$.};
\end{scope}
\end{tikzpicture}
\]
\end{ex}

Recall from Section~\ref{sec:compositions} that the compositions of a set $I$ are ordered by refinement.  The construction above gives a bijection
\[
\begin{array}{rcl}
\{\text{maximally refined set compositions}\} & \longleftrightarrow & \TTT[I] \\[0.5em]
A & \mapsto & \EQ(A) \cup \ASC(A).
\end{array}
\]
For a set composition which is not maximally refined, the construction above produces a preorder, which does not satisfy the antisymmetry condition required to be a partial order.

\begin{prop}
\label{prop:refinementsets}
Let $A = (A_{1}, \ldots, A_{\ell})$ and $B$ be set compositions of a finite set $I$.  Then
\[
\EQ(A \titsprod B) = \EQ(A) \cap \EQ(B)\qquad\text{and}\qquad \ASC(A \titsprod B) = \ASC(A) \sqcup \bigsqcup_{i = 1}^{\ell(A)} \ASC(B|_{A_{i}}) 
\]
\end{prop} 
\begin{proof}[Proof of Proposition~\ref{prop:refinementsets}]
Recall that $A \titsprod B = B|_{A_{1}}  \cdots B|_{A_{\ell(A)}}$.  The first statement follows from the definition of $\EQ$.  For the second, take $(i, j) \in \ASC(A \titsprod B)$; either $(i, j) \in A_{r} \times A_{s}$ for some $r < s$, or $i, j \in A_{r}$ for some $r$.  If the latter holds, $(i, j) \in B_{s} \times B_{t}$ for some $s < t$.
\end{proof}

This result can also be seen in the array-based construction described above.  

\begin{ex}
Taking $A \vDash \{a, b, c, d\}$ to be one of the set compositions in Example~\ref{ex:asceqinv} and $B = (\{b, c, d\}, \{a\})$, $A \titsprod B$ and the corresponding ascent, equality, and inversion sets are
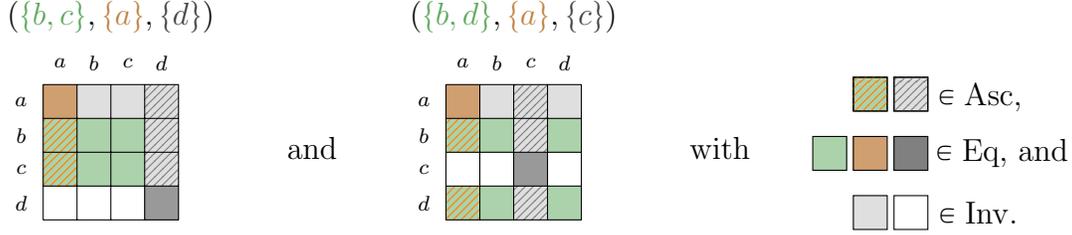
\[
\begin{tikzpicture}[scale = 0.45, baseline = 0.45*-2.2cm]
\node at (2, 2) {$({\color{green!30!gray} \{ b, c\}},{\color{orange!50!gray} \{a \}}, {\color{gray!150}\{d\}} )$};
\draw (1 - 0.5, 0.65) node {$\scriptstyle a$};
\draw (2 - 0.5, 0.65) node {$\scriptstyle b$};
\draw (3 - 0.5, 0.65) node {$\scriptstyle c$};
\draw (4 - 0.5, 0.65) node {$\scriptstyle d$};
\draw (-0.65, -1 + 0.5) node {$\scriptstyle a$};
\draw (-0.65, -2 + 0.5) node {$\scriptstyle b$};
\draw (-0.65, -3 + 0.5) node {$\scriptstyle c$};
\draw (-0.65, -4 + 0.5) node {$\scriptstyle d$};
\foreach \x in {1, 2, 3, 4}{\foreach \y in {1, 2, 3}{\path[fill = gray!25] (\x - 1, -\y + 1) -- (\x, -\y + 1) -- (\x, -\y) -- (\x - 1, -\y) -- cycle;}}
\foreach \x in {1, 2, 3}{\foreach \y in {2, 3}{\path[fill = green!30!gray!50] (\x - 1, -\y + 1) -- (\x, -\y + 1) -- (\x, -\y) -- (\x - 1, -\y) -- cycle;}}
\foreach \x in {1}{\foreach \y in {1}{\path[fill = orange!50!gray!75] (\x - 1, -\y + 1) -- (\x, -\y + 1) -- (\x, -\y) -- (\x - 1, -\y) -- cycle;}}
\foreach \x in {1}{\foreach \y in {2, 3}{\path[pattern = north east lines, pattern color = orange!100!gray] (\x - 1, -\y + 1) -- (\x, -\y + 1) -- (\x, -\y) -- (\x - 1, -\y) -- cycle;}}
\foreach \x in {4}{\foreach \y in {4}{\path[fill = gray!80] (\x - 1, -\y + 1) -- (\x, -\y + 1) -- (\x, -\y) -- (\x - 1, -\y) -- cycle;}}
\foreach \x in {4}{\foreach \y in {1, 2, 3}{\path[pattern = north east lines, pattern color = gray] (\x - 1, -\y + 1) -- (\x, -\y + 1) -- (\x, -\y) -- (\x - 1, -\y) -- cycle;}}
\draw (0, 0) grid (4, -4);
\end{tikzpicture}
\qquad\text{and}\qquad
\begin{tikzpicture}[scale = 0.45, baseline = 0.45*-2.2cm]
\node at (2, 2) {$({\color{green!30!gray} \{ b, d\}},{\color{orange!50!gray} \{a\}}, {\color{gray!150}\{c\}} )$};
\draw (1 - 0.5, 0.65) node {$\scriptstyle a$};
\draw (2 - 0.5, 0.65) node {$\scriptstyle b$};
\draw (3 - 0.5, 0.65) node {$\scriptstyle c$};
\draw (4 - 0.5, 0.65) node {$\scriptstyle d$};
\draw (-0.65, -1 + 0.5) node {$\scriptstyle a$};
\draw (-0.65, -2 + 0.5) node {$\scriptstyle b$};
\draw (-0.65, -3 + 0.5) node {$\scriptstyle c$};
\draw (-0.65, -4 + 0.5) node {$\scriptstyle d$};
\foreach \x in {1, 2, 3, 4}{\foreach \y in {1, 2, 4}{\path[fill = gray!25] (\x - 1, -\y + 1) -- (\x, -\y + 1) -- (\x, -\y) -- (\x - 1, -\y) -- cycle;}}
\foreach \x in {1, 2, 4}{\foreach \y in {2, 4}{\path[fill = green!30!gray!50] (\x - 1, -\y + 1) -- (\x, -\y + 1) -- (\x, -\y) -- (\x - 1, -\y) -- cycle;}}
\foreach \x in {1}{\foreach \y in {1}{\path[fill = orange!50!gray!75] (\x - 1, -\y + 1) -- (\x, -\y + 1) -- (\x, -\y) -- (\x - 1, -\y) -- cycle;}}
\foreach \x in {1}{\foreach \y in {2, 4}{\path[pattern = north east lines, pattern color = orange!100!gray] (\x - 1, -\y + 1) -- (\x, -\y + 1) -- (\x, -\y) -- (\x - 1, -\y) -- cycle;}}
\foreach \x in {3}{\foreach \y in {3}{\path[fill = gray!80] (\x - 1, -\y + 1) -- (\x, -\y + 1) -- (\x, -\y) -- (\x - 1, -\y) -- cycle;}}
\foreach \x in {3}{\foreach \y in {1, 2, 4}{\path[pattern = north east lines, pattern color = gray] (\x - 1, -\y + 1) -- (\x, -\y + 1) -- (\x, -\y) -- (\x - 1, -\y) -- cycle;}}
\draw (0, 0) grid (4, -4);
\end{tikzpicture}
\qquad\text{with}\qquad
\begin{tikzpicture}[scale = 0.45, baseline = -0.3cm]
\begin{scope}[yshift = 1.75cm]
\begin{scope}[xshift = -1.2cm]
\draw[fill = green!30!gray!50] (0, 0) -- (1, 0) -- (1, -1) -- (0, -1) -- cycle;
\draw[pattern = north east lines, pattern color = orange!100!gray] (0, 0) -- (1, 0) -- (1, -1) -- (0, -1) -- cycle;
\end{scope}
\begin{scope}[xshift = 0cm]
\draw[fill = gray!25] (0, 0) -- (1, 0) -- (1, -1) -- (0, -1) -- cycle;
\draw[pattern = north east lines, pattern color = gray] (0, 0) -- (1, 0) -- (1, -1) -- (0, -1) -- cycle;
\end{scope}
\draw (2.55, -0.6) node {$\in \ASC$,};
\end{scope}
\begin{scope}[xshift = -2.4cm]
\draw[fill = green!30!gray!50] (0, 0) -- (1, 0) -- (1, -1) -- (0, -1) -- cycle;
\end{scope}
\begin{scope}[xshift = -1.2cm]
\draw[fill = orange!50!gray!75] (0, 0) -- (1, 0) -- (1, -1) -- (0, -1) -- cycle;
\end{scope}
\begin{scope}[xshift = 0cm]
\draw[fill = gray] (0, 0) -- (1, 0) -- (1, -1) -- (0, -1) -- cycle;
\end{scope}
\draw (3.2, -0.5) node {$\in \EQ$, and};
\begin{scope}[yshift = -1.75cm]
\begin{scope}[xshift = -1.2cm]
\draw[fill = gray!25] (0, 0) -- (1, 0) -- (1, -1) -- (0, -1) -- cycle;
\end{scope}
\begin{scope}[xshift = 0cm]
\draw[fill = white] (0, 0) -- (1, 0) -- (1, -1) -- (0, -1) -- cycle;
\end{scope}
\draw (2.5, -0.5) node {$\in \INV$.};
\end{scope}.
\end{tikzpicture}
\]
\end{ex}

Generally the partition of $I \times I$ from $A \wedge B$ can be obtained from that of $A$ by taking each sub-array of cells with indices in $A_{r} \times A_{r}$ and partitioning it according to $B|_{A_{r}}$.  

\subsection{Functorial Subgroups}
\label{sec:CFalgprelims}

The functors used to define the product and coproduct pass between certain pattern subgroups of $\UT(\tau)$.  For a set composition $A \vDash I$, these are respectively denoted by
\begin{align*}
\textbf{Levi:} && \UL(\tau, A) =& \UT\big(\tau \cap \EQ(A)\big)
, \\[0.5em]
\textbf{radical:} && \UR(\tau, A) =& \UT\big(\big\{(i, j) \;\big|\: \text{$i = j$ or $(i, j) \in \tau \cap \ASC(A)$} \big\} \big)
,\,\text{and} \\[0.5em]
\textbf{parabolic:} && \UP(\tau, A) =& \UT\big(\tau \cap \big(\EQ(A) \cup \ASC(A)\big)\big).
\end{align*}

\begin{rem}
The nomenclature and notation for these groups is a deliberate reference to similar structures in the general linear groups; see Section~\ref{sec:GLalg} and Lemma~\ref{lem:subgroupintersection}.
\end{rem}

The visualization of $\EQ(A)$ and $\ASC(A)$ in Section~\ref{sec:CFcombprelims} also describes the groups defined above: each nonzero nondiagonal entry in an element of $\UL(\tau, A)$, $\UR(\tau, A)$, and $\UP(\tau, A)$ corresponds to an above-diagonal cell in $\EQ(A)$, $\ASC(A)$, and $\EQ(A) \cup \ASC(A)$.

\begin{ex}
\label{ex:functorialsubgroups1}
With $I = \{a, b, c, d\}$, $\tau = a < b < c < d$, and $A = (\{a, b, c\}, \{d\})$,
\[
\UL(\tau, A) = 
\begin{tikzpicture}[scale = 0.45, baseline = 0.45*-2.2cm]
\draw (1 - 0.5, 0.65) node {$\scriptstyle a$};
\draw (2 - 0.5, 0.65) node {$\scriptstyle b$};
\draw (3 - 0.5, 0.65) node {$\scriptstyle c$};
\draw (4 - 0.5, 0.65) node {$\scriptstyle d$};
\draw (-0.65, -1 + 0.5) node {$\scriptstyle a$};
\draw (-0.65, -2 + 0.5) node {$\scriptstyle b$};
\draw (-0.65, -3 + 0.5) node {$\scriptstyle c$};
\draw (-0.65, -4 + 0.5) node {$\scriptstyle d$};
\draw (0.2, 0.085) -- (-0.15, 0.085);
\draw[very thick] (-0.1, 0.1) -- (-0.1, -4.1);
\draw (-0.15, -4.085) -- (0.2, -4.085);
\draw (3.8, 0.085) -- (4.15, 0.085);
\draw[very thick] (4.1, 0.1) -- (4.1, -4.1);
\draw (4.15, -4.085) -- (3.8, -4.085);
\foreach \x in {1, 2, 3, 4}{\foreach \y in {1, 2, 3}{\path[fill = gray!25] (\x - 1, -\y + 1) -- (\x, -\y + 1) -- (\x, -\y) -- (\x - 1, -\y) -- cycle;}}
\foreach \x in {4}{\foreach \y in {4}{\path[fill = gray!80] (\x - 1, -\y + 1) -- (\x, -\y + 1) -- (\x, -\y) -- (\x - 1, -\y) -- cycle;}}
\foreach \x in {4}{\foreach \y in {1, 2, 3}{\path[pattern = north east lines, pattern color = gray] (\x - 1, -\y + 1) -- (\x, -\y + 1) -- (\x, -\y) -- (\x - 1, -\y) -- cycle;}}
\foreach \x in {1, 2, 3, 4}{\draw (\x - 0.5, -\x + 0.5) node {$1$};}
\foreach \x in {1}{\foreach \y in {2, 3, 4}{\draw (\x - 0.5, -\y + 0.5) node {$0$};}}
\foreach \x in {2}{\foreach \y in {3, 4}{\draw (\x - 0.5, -\y + 0.5) node {$0$};}}
\foreach \x in {3}{\foreach \y in {4}{\draw (\x - 0.5, -\y + 0.5) node {$0$};}}
\foreach \x in {2, 3}{\foreach \y in {1}{\draw (\x - 0.5, -\y + 0.5) node {$\ast$};}}
\foreach \x in {3}{\foreach \y in {2}{\draw (\x - 0.5, -\y + 0.5) node {$\ast$};}}
\foreach \x in {4}{\foreach \y in {1}{\draw (\x - 0.5, -\y + 0.5) node {$0$};}}
\foreach \x in {4}{\foreach \y in {2}{\draw (\x - 0.5, -\y + 0.5) node {$0$};}}
\foreach \x in {4}{\foreach \y in {3}{\draw (\x - 0.5, -\y + 0.5) node {$0$};}}
\end{tikzpicture},
\quad \UR(\tau, A) = 
\begin{tikzpicture}[scale = 0.45, baseline = 0.45*-2.2cm]
\draw (1 - 0.5, 0.65) node {$\scriptstyle a$};
\draw (2 - 0.5, 0.65) node {$\scriptstyle b$};
\draw (3 - 0.5, 0.65) node {$\scriptstyle c$};
\draw (4 - 0.5, 0.65) node {$\scriptstyle d$};
\draw (-0.65, -1 + 0.5) node {$\scriptstyle a$};
\draw (-0.65, -2 + 0.5) node {$\scriptstyle b$};
\draw (-0.65, -3 + 0.5) node {$\scriptstyle c$};
\draw (-0.65, -4 + 0.5) node {$\scriptstyle d$};
\draw (0.2, 0.085) -- (-0.15, 0.085);
\draw[very thick] (-0.1, 0.1) -- (-0.1, -4.1);
\draw (-0.15, -4.085) -- (0.2, -4.085);
\draw (3.8, 0.085) -- (4.15, 0.085);
\draw[very thick] (4.1, 0.1) -- (4.1, -4.1);
\draw (4.15, -4.085) -- (3.8, -4.085);
\foreach \x in {1, 2, 3, 4}{\foreach \y in {1, 2, 3}{\path[fill = gray!25] (\x - 1, -\y + 1) -- (\x, -\y + 1) -- (\x, -\y) -- (\x - 1, -\y) -- cycle;}}
\foreach \x in {4}{\foreach \y in {4}{\path[fill = gray!80] (\x - 1, -\y + 1) -- (\x, -\y + 1) -- (\x, -\y) -- (\x - 1, -\y) -- cycle;}}
\foreach \x in {4}{\foreach \y in {1, 2, 3}{\path[pattern = north east lines, pattern color = gray] (\x - 1, -\y + 1) -- (\x, -\y + 1) -- (\x, -\y) -- (\x - 1, -\y) -- cycle;}}
\foreach \x in {1, 2, 3, 4}{\draw (\x - 0.5, -\x + 0.5) node {$1$};}
\foreach \x in {1}{\foreach \y in {2, 3, 4}{\draw (\x - 0.5, -\y + 0.5) node {$0$};}}
\foreach \x in {2}{\foreach \y in {3, 4}{\draw (\x - 0.5, -\y + 0.5) node {$0$};}}
\foreach \x in {3}{\foreach \y in {4}{\draw (\x - 0.5, -\y + 0.5) node {$0$};}}
\foreach \x in {2, 3}{\foreach \y in {1}{\draw (\x - 0.5, -\y + 0.5) node {$0$};}}
\foreach \x in {3}{\foreach \y in {2}{\draw (\x - 0.5, -\y + 0.5) node {$0$};}}
\foreach \x in {4}{\foreach \y in {1}{\draw (\x - 0.5, -\y + 0.5) node {$\ast$};}}
\foreach \x in {4}{\foreach \y in {2}{\draw (\x - 0.5, -\y + 0.5) node {$\ast$};}}
\foreach \x in {4}{\foreach \y in {3}{\draw (\x - 0.5, -\y + 0.5) node {$\ast$};}}
\end{tikzpicture},
\quad \text{and} \quad \UP(\tau, A) = 
\begin{tikzpicture}[scale = 0.45, baseline = 0.45*-2.2cm]
\draw (1 - 0.5, 0.65) node {$\scriptstyle a$};
\draw (2 - 0.5, 0.65) node {$\scriptstyle b$};
\draw (3 - 0.5, 0.65) node {$\scriptstyle c$};
\draw (4 - 0.5, 0.65) node {$\scriptstyle d$};
\draw (-0.65, -1 + 0.5) node {$\scriptstyle a$};
\draw (-0.65, -2 + 0.5) node {$\scriptstyle b$};
\draw (-0.65, -3 + 0.5) node {$\scriptstyle c$};
\draw (-0.65, -4 + 0.5) node {$\scriptstyle d$};
\draw (0.2, 0.085) -- (-0.15, 0.085);
\draw[very thick] (-0.1, 0.1) -- (-0.1, -4.1);
\draw (-0.15, -4.085) -- (0.2, -4.085);
\draw (3.8, 0.085) -- (4.15, 0.085);
\draw[very thick] (4.1, 0.1) -- (4.1, -4.1);
\draw (4.15, -4.085) -- (3.8, -4.085);
\foreach \x in {1, 2, 3, 4}{\foreach \y in {1, 2, 3}{\path[fill = gray!25] (\x - 1, -\y + 1) -- (\x, -\y + 1) -- (\x, -\y) -- (\x - 1, -\y) -- cycle;}}
\foreach \x in {4}{\foreach \y in {4}{\path[fill = gray!80] (\x - 1, -\y + 1) -- (\x, -\y + 1) -- (\x, -\y) -- (\x - 1, -\y) -- cycle;}}
\foreach \x in {4}{\foreach \y in {1, 2, 3}{\path[pattern = north east lines, pattern color = gray] (\x - 1, -\y + 1) -- (\x, -\y + 1) -- (\x, -\y) -- (\x - 1, -\y) -- cycle;}}
\foreach \x in {1, 2, 3, 4}{\draw (\x - 0.5, -\x + 0.5) node {$1$};}
\foreach \x in {1}{\foreach \y in {2, 3, 4}{\draw (\x - 0.5, -\y + 0.5) node {$0$};}}
\foreach \x in {2}{\foreach \y in {3, 4}{\draw (\x - 0.5, -\y + 0.5) node {$0$};}}
\foreach \x in {3}{\foreach \y in {4}{\draw (\x - 0.5, -\y + 0.5) node {$0$};}}
\foreach \x in {2, 3, 4}{\foreach \y in {1}{\draw (\x - 0.5, -\y + 0.5) node {$\ast$};}}
\foreach \x in {3, 4}{\foreach \y in {2}{\draw (\x - 0.5, -\y + 0.5) node {$\ast$};}}
\foreach \x in {4}{\foreach \y in {3}{\draw (\x - 0.5, -\y + 0.5) node {$\ast$};}}
\end{tikzpicture}.
\]
With $I$ and $\tau$ as above and  and $B = (\{a, b, d\}, \{c\})$,
\[
\UL(\tau, B) = 
\begin{tikzpicture}[scale = 0.45, baseline = 0.45*-2.2cm]
\draw (1 - 0.5, 0.65) node {$\scriptstyle a$};
\draw (2 - 0.5, 0.65) node {$\scriptstyle b$};
\draw (3 - 0.5, 0.65) node {$\scriptstyle c$};
\draw (4 - 0.5, 0.65) node {$\scriptstyle d$};
\draw (-0.65, -1 + 0.5) node {$\scriptstyle a$};
\draw (-0.65, -2 + 0.5) node {$\scriptstyle b$};
\draw (-0.65, -3 + 0.5) node {$\scriptstyle c$};
\draw (-0.65, -4 + 0.5) node {$\scriptstyle d$};
\draw (0.2, 0.085) -- (-0.15, 0.085);
\draw[very thick] (-0.1, 0.1) -- (-0.1, -4.1);
\draw (-0.15, -4.085) -- (0.2, -4.085);
\draw (3.8, 0.085) -- (4.15, 0.085);
\draw[very thick] (4.1, 0.1) -- (4.1, -4.1);
\draw (4.15, -4.085) -- (3.8, -4.085);
\foreach \x in {1, 2, 3, 4}{\foreach \y in {1, 2, 4}{\path[fill = gray!25] (\x - 1, -\y + 1) -- (\x, -\y + 1) -- (\x, -\y) -- (\x - 1, -\y) -- cycle;}}
\foreach \x in {3}{\foreach \y in {3}{\path[fill = gray!80] (\x - 1, -\y + 1) -- (\x, -\y + 1) -- (\x, -\y) -- (\x - 1, -\y) -- cycle;}}
\foreach \x in {3}{\foreach \y in {1, 2, 4}{\path[pattern = north east lines, pattern color = gray] (\x - 1, -\y + 1) -- (\x, -\y + 1) -- (\x, -\y) -- (\x - 1, -\y) -- cycle;}}
\foreach \x in {1, 2, 3, 4}{\draw (\x - 0.5, -\x + 0.5) node {$1$};}
\foreach \x in {1}{\foreach \y in {2, 3, 4}{\draw (\x - 0.5, -\y + 0.5) node {$0$};}}
\foreach \x in {2}{\foreach \y in {3, 4}{\draw (\x - 0.5, -\y + 0.5) node {$0$};}}
\foreach \x in {3}{\foreach \y in {4}{\draw (\x - 0.5, -\y + 0.5) node {$0$};}}
\foreach \x in {2, 4}{\foreach \y in {1}{\draw (\x - 0.5, -\y + 0.5) node {$\ast$};}}
\foreach \x in {4}{\foreach \y in {2}{\draw (\x - 0.5, -\y + 0.5) node {$\ast$};}}
\foreach \x in {3}{\foreach \y in {1}{\draw (\x - 0.5, -\y + 0.5) node {$0$};}}
\foreach \x in {3}{\foreach \y in {2}{\draw (\x - 0.5, -\y + 0.5) node {$0$};}}
\foreach \x in {4}{\foreach \y in {3}{\draw (\x - 0.5, -\y + 0.5) node {$0$};}}
\end{tikzpicture},
\quad \UR(\tau, B) = 
\begin{tikzpicture}[scale = 0.45, baseline = 0.45*-2.2cm]
\draw (1 - 0.5, 0.65) node {$\scriptstyle a$};
\draw (2 - 0.5, 0.65) node {$\scriptstyle b$};
\draw (3 - 0.5, 0.65) node {$\scriptstyle c$};
\draw (4 - 0.5, 0.65) node {$\scriptstyle d$};
\draw (-0.65, -1 + 0.5) node {$\scriptstyle a$};
\draw (-0.65, -2 + 0.5) node {$\scriptstyle b$};
\draw (-0.65, -3 + 0.5) node {$\scriptstyle c$};
\draw (-0.65, -4 + 0.5) node {$\scriptstyle d$};
\draw (0.2, 0.085) -- (-0.15, 0.085);
\draw[very thick] (-0.1, 0.1) -- (-0.1, -4.1);
\draw (-0.15, -4.085) -- (0.2, -4.085);
\draw (3.8, 0.085) -- (4.15, 0.085);
\draw[very thick] (4.1, 0.1) -- (4.1, -4.1);
\draw (4.15, -4.085) -- (3.8, -4.085);
\foreach \x in {1, 2, 3, 4}{\foreach \y in {1, 2, 4}{\path[fill = gray!25] (\x - 1, -\y + 1) -- (\x, -\y + 1) -- (\x, -\y) -- (\x - 1, -\y) -- cycle;}}
\foreach \x in {3}{\foreach \y in {3}{\path[fill = gray!80] (\x - 1, -\y + 1) -- (\x, -\y + 1) -- (\x, -\y) -- (\x - 1, -\y) -- cycle;}}
\foreach \x in {3}{\foreach \y in {1, 2, 4}{\path[pattern = north east lines, pattern color = gray] (\x - 1, -\y + 1) -- (\x, -\y + 1) -- (\x, -\y) -- (\x - 1, -\y) -- cycle;}}
\foreach \x in {1, 2, 3, 4}{\draw (\x - 0.5, -\x + 0.5) node {$1$};}
\foreach \x in {1}{\foreach \y in {2, 3, 4}{\draw (\x - 0.5, -\y + 0.5) node {$0$};}}
\foreach \x in {2}{\foreach \y in {3, 4}{\draw (\x - 0.5, -\y + 0.5) node {$0$};}}
\foreach \x in {3}{\foreach \y in {4}{\draw (\x - 0.5, -\y + 0.5) node {$0$};}}
\foreach \x in {2, 4}{\foreach \y in {1}{\draw (\x - 0.5, -\y + 0.5) node {$0$};}}
\foreach \x in {4}{\foreach \y in {2}{\draw (\x - 0.5, -\y + 0.5) node {$0$};}}
\foreach \x in {4}{\foreach \y in {3}{\draw (\x - 0.5, -\y + 0.5) node {$0$};}}
\foreach \x in {3}{\foreach \y in {1}{\draw (\x - 0.5, -\y + 0.5) node {$\ast$};}}
\foreach \x in {3}{\foreach \y in {2}{\draw (\x - 0.5, -\y + 0.5) node {$\ast$};}}
\end{tikzpicture},
\quad \text{and} \quad \UP(\tau, B) = 
\begin{tikzpicture}[scale = 0.45, baseline = 0.45*-2.2cm]
\draw (1 - 0.5, 0.65) node {$\scriptstyle a$};
\draw (2 - 0.5, 0.65) node {$\scriptstyle b$};
\draw (3 - 0.5, 0.65) node {$\scriptstyle c$};
\draw (4 - 0.5, 0.65) node {$\scriptstyle d$};
\draw (-0.65, -1 + 0.5) node {$\scriptstyle a$};
\draw (-0.65, -2 + 0.5) node {$\scriptstyle b$};
\draw (-0.65, -3 + 0.5) node {$\scriptstyle c$};
\draw (-0.65, -4 + 0.5) node {$\scriptstyle d$};
\draw (0.2, 0.085) -- (-0.15, 0.085);
\draw[very thick] (-0.1, 0.1) -- (-0.1, -4.1);
\draw (-0.15, -4.085) -- (0.2, -4.085);
\draw (3.8, 0.085) -- (4.15, 0.085);
\draw[very thick] (4.1, 0.1) -- (4.1, -4.1);
\draw (4.15, -4.085) -- (3.8, -4.085);
\foreach \x in {1, 2, 3, 4}{\foreach \y in {1, 2, 4}{\path[fill = gray!25] (\x - 1, -\y + 1) -- (\x, -\y + 1) -- (\x, -\y) -- (\x - 1, -\y) -- cycle;}}
\foreach \x in {3}{\foreach \y in {3}{\path[fill = gray!80] (\x - 1, -\y + 1) -- (\x, -\y + 1) -- (\x, -\y) -- (\x - 1, -\y) -- cycle;}}
\foreach \x in {3}{\foreach \y in {1, 2, 4}{\path[pattern = north east lines, pattern color = gray] (\x - 1, -\y + 1) -- (\x, -\y + 1) -- (\x, -\y) -- (\x - 1, -\y) -- cycle;}}
\foreach \x in {1, 2, 3, 4}{\draw (\x - 0.5, -\x + 0.5) node {$1$};}
\foreach \x in {1}{\foreach \y in {2, 3, 4}{\draw (\x - 0.5, -\y + 0.5) node {$0$};}}
\foreach \x in {2}{\foreach \y in {3, 4}{\draw (\x - 0.5, -\y + 0.5) node {$0$};}}
\foreach \x in {3}{\foreach \y in {4}{\draw (\x - 0.5, -\y + 0.5) node {$0$};}}
\foreach \x in {4}{\foreach \y in {3}{\draw (\x - 0.5, -\y + 0.5) node {$0$};}}
\foreach \x in {2, 3, 4}{\foreach \y in {1}{\draw (\x - 0.5, -\y + 0.5) node {$\ast$};}}
\foreach \x in {3, 4}{\foreach \y in {2}{\draw (\x - 0.5, -\y + 0.5) node {$\ast$};}}
\end{tikzpicture}.
\]
\end{ex}

Note that
\begin{equation}
\label{eq:leviiso}
\UL(\tau, A) = \UT(\tau|_{A_{1}}) \oplus \cdots \oplus \UT(\tau|_{A_{\ell}}) 
\end{equation}
so that $\cf(\UL(\tau, A))$ is a summand of $\cfUT(A)$, as in Equation~\eqref{eq:cfA}.

\begin{prop}
\label{lem:functorialsubgroupinteractions}
\label{lem:compatibleparabolics}
Let $I$ be a finite set $\tau \in \TTT[I]$ a total order and $A \vDash I$ a set composition.
\begin{enumerate}
\item If $\ASC(A) \subseteq \tau$, then $\UP(\tau, A) = \UT(\tau)$ and $\UR(\tau, A) = \UT\big(\ASC(A) \cup \{(i, i) \;|\; i \in I\} \big)$; 

\item $\UR(\tau, A)$ is a normal complement of $\UL(\tau, A)$ in $\UP(\tau, A)$, so that 
\[
\UP(\tau, A) = \UL(\tau, A) \ltimes \UR(\tau, A).
\]

\end{enumerate}
\end{prop}
\begin{proof}
For 1., assume that $\ASC(A) \subseteq \tau$ and take $(i, j) \in \tau$.  If  $i$ and $j$ are in the same part of $A$, then $(i, j) \in \EQ(A)$, and otherwise $(i, j) \in \ASC(A)$ by assumption $(i, j) \in \ASC(A)$.  

For 2., note that $\UL(\tau, A)$ and $\UR(\tau, A)$ have trivial intersection, so by order consideration $\UP(\tau, A) = \UL(\tau, A)\UR(\tau, A)$.  What remains is to show that $\UR(\tau, A) \trianglelefteq \UP(\tau, A)$.

First assume that $\ASC(A) \subseteq \tau$ as above.  Applying Proposition~\ref{lem:patterngroupnormalcy}, take an element $(r, s) \in \ASC(A)$, and let $s$ and $t$ be the elements $1 \le s < t \le \ell(A)$ for which $j \in A_{s}$ and $k \in A_{t}$.  For each pair $(i, j), (k, l) \in \tau$, either $i \in A_{s}$ or $i \in A_{r}$ for $r \le s$, and likewise either $l \in A_{t}$ or $l \in A_{u}$ for $u \ge t$, so in any case $(i, l) \in \ASC(A)$, so that $\UR(\tau, A) \trianglelefteq \UT(\tau) = \UP(\tau, A)$.

Returning to the general case, the final step is to reduce this case to that of the previous paragraph.  By definition,
\[
\tau|_{A_{1}} \ordinalsum \cdots \ordinalsum \tau|_{A_{\ell}} = \tau|_{A_{1}} \sqcup \cdots \sqcup \tau|_{A_{\ell}} \sqcup \ASC(A), 
\]
so $\ASC(A) \subseteq \tau|_{A_{1}} \ordinalsum \cdots \ordinalsum \tau|_{A_{\ell}}$.  Thus, the relevant definitions imply that
\[
\UR(\tau, A) = \UT(\tau) \cap \UR(\tau|_{A_{1}} \ordinalsum \cdots \ordinalsum \tau|_{A_{\ell}}, A)
\]
and
\[
\UP(\tau, A) = \UT(\tau) \cap \UP(\tau|_{A_{1}} \ordinalsum \cdots \ordinalsum \tau|_{A_{\ell}}, A).
\]
Normalcy is preserved under intersection with a third group, so this completes the proof.
\end{proof}

\subsection{The downward functor $\resf$; coassociativity}
\label{sec:downmaps}

Let $A \vDash I$.  The map used in the coproduct of $\cfUT$ is the composition 
\[
\resf^{\UT(\tau)}_{\UL(\tau, A)} \colon\;\cf(\UT(\tau)) \xrightarrow{\:\res^{\UT(\tau)}_{\UP(\tau, A)}\;} \cf(\UP(\tau, A)) \xrightarrow{\;\defl^{\UP(\tau, A)}_{\UL(\tau, A)}\;} \cf(\UL(\tau, A))
\]
where $\res$ and $\defl$ are as defined in Section~\ref{sec:classfunctions}.

\begin{prop}
\label{prop:downtransitivity}
Let $I$ be a finite set, $\tau \in \TTT[I]$ a total order, and $A, B \vDash I$ set compositions with $B$ refining $A$. 
The diagram
\[
\begin{tikzpicture}[baseline = -1.1cm]
\draw (10, -2) node[] (B) {$\cf(\UL(\tau, B))$};
\draw (0, -2) node[] (A)  {$\cf(\UL(\tau, A))$};
\draw (0, 0) node[] (I) {$\cf(\UT(\tau))$};
\draw[thick, <-] (B) -- node[above] {$\resf^{\UT(\tau)}_{\UL(\tau, B)}$} (I);
\draw[thick, <-] (B) -- node[below] {$\resf^{\UT(\tau|_{A_{1}})}_{\UL(\tau|_{A_{1}}, B|_{A_{1}})} \otimes \cdots \otimes \resf^{\UT(\tau|_{A_{\ell}})}_{\UL(\tau|_{A_{\ell}}, B|_{A_{\ell}})}$} (A);
\draw[thick, <-] (A) -- node[left] {$\resf^{\UT(\tau)}_{\UL(\tau, A)}$} (I);
\end{tikzpicture}
\]
commutes.
\end{prop}
\begin{proof}
By Propositions~\ref{prop:refinementsets} and~\ref{lem:functorialsubgroupinteractions}, $\UR(\tau, B) \subseteq \UP(\tau, A)$, and moreover 
\[
\UR(\tau, B) = \big(\UR(\tau, B) \cap \UL(\tau, A)\big) \rtimes \big(\UR(\tau, B) \cap \UR(\tau, A)\big).
\]
Using Proposition~\ref{prop:refinementsets} again,
\[
\UR(\tau, B) \cap \UL(\tau, A) = \UR(\tau|_{A_{1}}, B|_{A_{1}}) \oplus \cdots \oplus \UR(\tau|_{A_{\ell}}, B|_{A_{\ell}}) 
\]
and
\[
\UR(\tau, B) \cap \UR(\tau, A) = \UR(\tau, A),
\]
so each element of $\UR(\tau, B)$ has a unique expression as the product of elements from the two subgroups displayed above.  
As a consequence, for $g \in \UL(\tau, B)$ and $\psi \in \cf(\UL(\tau, A))$, 
\begin{align*}
\left( \bigotimes_{i = 1}^{\ell} \resf^{\UT(\tau|_{A_{i}})}_{\UL(\tau|_{A_{i}}, B|_{A_{i}})} \right) \circ \resf^{\UT(\tau)}_{\UL(\tau, A)}(\psi)(g) &= \frac{1}{|\UR(\tau, B)|}  \sum_{x \in \UR(\tau, B)} \hspace{-0.5em} \psi(gx) \\
&= \resf^{\UT(\tau)}_{\UL(\tau, B)}(\psi)(g). \qedhere
\end{align*}
\end{proof}

\subsection{The upward functor $\infl$; associativity}
\label{sec:upmaps}

This section will define the representation theoretic map used in the product of $\cfUT$, and establish its associativity.  A consequence of Proposition~\ref{lem:functorialsubgroupinteractions}, $\UT(\tau) = \UL(\tau, A) \rtimes \UR(\tau, A)$ for any composition $A \vDash I$ with $\ASC(A) \subseteq \tau$; in this case inflation gives a map
\[
\infl^{\UT(\tau)}_{\UL(\tau, A)} \colon \cf(\UL(\tau, A)) \to \cf(\UT(\tau))
\]
as in Section~\ref{sec:classfunctions}.  
This coincides with the product map used in~\cite{AgBerTh}, so its associativity is known; the following results and proofs are included for the sake of completeness.

\begin{lem}
Let $I$ be a finite set, $\tau \in \TTT[I]$ a total order, and $A \vDash I$ a set composition with $\ASC(A) \subseteq \tau$.  The maps $\infl^{\UT(\tau)}_{\UL(\tau, A)}$ and $\resf^{\UT(\tau)}_{\UL(\tau, A)}$ are adjoint.
\end{lem}
\begin{proof}
By assumption, $\UP(\tau, A) = \UT(\tau)$, so $\resf^{\UT(\tau)}_{\UL(\tau, A)} = \defl^{\UT(\tau)}_{\UL(\tau, A)}$; the claim now follows from the fact that inflation and deflation are adjoint.
\end{proof}

By the uniqueness of adjoints, the transitive property described in Proposition~\ref{prop:downtransitivity} for $\resf^{\UT(\tau)}_{\UL(\tau, A)}$ extends to the inflation map $\infl^{\UT(\tau)}_{\UL(\tau, A)}$, giving the following result.

\begin{prop}
\label{prop:upwardtransitivity}
Let $I$ be a finite set, $\tau \in \TTT[I]$ a total order, and $A, B \vDash I$ set compositions with $B$ refining $A$ and with $\ASC(A), \ASC(B) \subseteq \tau$.  The diagram
\[
\begin{tikzpicture}[baseline = -1.1cm]
\draw (0, 0) node[] (B) {$\cf(\UL(\tau, B))$};
\draw (10, 0) node[] (A)  {$\cf(\UL(\tau, A))$};
\draw (10, -2) node[] (I) {$\cf(\UT(\tau))$};
\draw[thick, ->] (B) -- node[below left] {$\infl^{\UT(\tau)}_{\UL(\tau, B)}$} (I);
\draw[thick, ->] (B) -- node[above] {$\infl^{\UT(\tau|_{A_{1}})}_{\UL(\tau|_{A_{1}}, B|_{A_{1}})} \otimes \cdots \otimes \infl^{\UT(\tau|_{A_{\ell}})}_{\UL(\tau|_{A_{\ell}}, B|_{A_{\ell}})}$} (A);
\draw[thick, ->] (A) -- node[right] {$\infl^{\UT(\tau)}_{\UL(\tau, A)}$} (I);
\end{tikzpicture}
\]
commutes.
\end{prop}

\subsection{Compatibility and Naturality}
\label{sec:compatibility}
\label{sec:naturality}

This section establishes the remaining results used in the proof of Theorem~\ref{thm:CFhopfmonoid}.  
I will start with the compatibility axiom.  
In the following result, note that for any $A, B \vDash I$, Proposition~\ref{prop:refinementsets} gives that 
\[
\UL(\tau, A \titsprod B) = \UL(\tau, A) \cap \UL(\tau, B) = \UL(\tau, B \titsprod A).
\]

\begin{prop}
\label{prop:compatibility}
Let $I$ be a finite set, $\tau \in \TTT[I]$ a total order, and $A, B \vDash I$ set compositions with $\ASC(A) \subseteq \tau$.  The diagram
\begin{center}
\begin{tikzpicture}
\draw (0, -3) node[] (ULAB) {$\cf(\UL(\tau; A \titsprod B))$};
\draw (9, -3) node[] (ULBA) {$\cf(\UL(\tau; B \titsprod A))$};
\draw (0, 0) node[] (ULA) {$\cf(\UL(\tau; A))$};
\draw (4.5, 0) node[] (UT) {$\cf(\UT(\tau))$};
\draw (9, 0) node[] (ULB) {$\cf(\UL(\tau; B))$};
\draw[thick, ->] (ULA) -- node[above] {$\infl^{\UT(\tau)}_{\UL(\tau, A)}$} (UT);
\draw[thick, ->] (UT) -- node[above] {$\resf^{\UT(\tau)}_{\UL(\tau, B)}$} (ULB);
\draw[thick, ->] (ULA) -- node[left] {$\displaystyle \bigotimes_{i = 1}^{\ell(A)} \resf^{\UT(\tau|_{A_{i}})}_{\UL(\tau|_{A_{i}}, B|_{A_{i}})}$} (ULAB);
\draw[thick, ->] (ULAB) -- node[above] {$=$} (ULBA);
\draw[thick, ->] (ULBA) -- node[right] {$\displaystyle \bigotimes_{j = 1}^{\ell(B)} \infl^{\UT(\tau|_{B_{i}})}_{\UL(\tau|_{B_{i}}, A|_{B_{i}})}$} (ULB);
\end{tikzpicture}
\end{center}
commutes.
\end{prop}
\begin{proof}
By Lemma~\ref{lem:compatibleparabolics} $\UT(\tau) = \UL(\tau, A) \rtimes \UR(\tau, A)$, so any element $g \in \UT(\tau)$ can be written $g = g_{L}g_{R}$ with $g_{L} \in \UL(\tau, A)$ and $g_{R} \in \UR(\tau, A)$.  Thus
\[
 \resf^{\UT(\tau)}_{\UL(\tau, B)} \circ \infl^{\UT(\tau)}_{\UL(\tau, A)} (\psi) (g) = \frac{1}{|\UR(\tau, B)|}  \sum_{x \in \UR(\tau, B)}  \psi\big( (gx)_{L} \big)
\]
for $g \in \UL(\tau, B)$ and $\psi \in \cf(\UL(\tau, A))$.  On the other hand, 
\[
 \left( \bigotimes_{j = 1}^{\ell(B)} \infl^{\UT(\tau|_{B_{i}})}_{\UL(\tau|_{B_{i}}, A|_{B_{i}})}\right) \circ \left( \bigotimes_{i = 1}^{\ell(A)} \resf^{\UT(\tau|_{A_{i}})}_{\UL(\tau|_{A_{i}}, B|_{A_{i}})} \right)  (\psi)(g) = {\textstyle \frac{1}{|\UR(\tau, B) \cap \UL(\tau, A)|}}  \hspace{-2em}  \sum_{y \in \UR(\tau, A) \cap \UL(\tau, A)} \hspace{-2em} \psi(g_{L} y)
\]
Finally, 
\[
\sum_{x \in \UR(\tau, B)}  \psi\big( (gx)_{L} \big) = |\UR(\tau, B) \cap \UR(\tau, A)| \hspace{-1.5em} \sum_{y \in \UR(\tau, B) \cap \UL(\tau, A)} \hspace{-1.5em} \psi(g_{L} y)
\]
and $|\UR(\tau, B)|/|\UR(\tau, B) \cap \UR(\tau, A)| = |\UR(\tau, B) \cap \UL(\tau, A)|$.
\end{proof}


The naturality axiom of $\cfUT$ follows from an equivariance under bijections in the construction of the resflation and inflation maps.

\begin{prop}
\label{prop:naturality}
Let $I$ and $J$ be finite sets, $\tau \in \TTT[I]$ a total order, and $\sigma:I \shortto J$ a bijection.  Fix a composition $A = (A_{1}, \ldots, A_{\ell}) \vDash I$ and let ${}^{\sigma}A$ and ${}^{\sigma} \tau$ denote the composition and total order of $J$ obtained by applying $\sigma$ to each part of $A$ and $\tau$, respectively.  
The diagram 
\[
\begin{tikzpicture}[baseline = -1.1cm]
\draw (0, -2) node[] (A) {$\cf(\UL(\tau, A))$};
\draw (0, 0) node[] (I) {$\cf(\UT(\tau))$};
\draw (7, -2) node[] (sA) {$\cf(\UL({}^{\sigma}\tau, {}^{\sigma}A))$};
\draw (7, 0) node[] (sI) {$\cf(\UT({}^{\sigma}\tau))$};
\draw[thick, <-] (A) -- node[left] {$\resf^{\UT(\tau)}_{\UL(\tau, A)}$} (I);
\draw[thick, ->] (A) -- node[above] {$(\sigma|_{A_{1}})_{\ast} \otimes \cdots \otimes (\sigma|_{A_{\ell}})_{\ast}$} (sA);
\draw[thick, ->] (I) -- node[above] {$(\sigma)_{\ast}$} (sI);
\draw[thick, <-] (sA) -- node[right] {$\resf^{\UT({}^{\sigma}\tau)}_{\UL({}^{\sigma}\tau, {}^{\sigma}A)}$} (sI);
\end{tikzpicture}
\]
commutes, and if $\ASC(A) \subseteq \tau$, the diagram
\[
\begin{tikzpicture}[baseline = -1.1cm]
\draw (0, 0) node[] (A) {$\cf(\UL(\tau, A))$};
\draw (0, -2) node[] (I) {$\cf(\UT(\tau))$};
\draw (7, 0) node[] (sA) {$\cf(\UL({}^{\sigma}\tau, {}^{\sigma}A))$};
\draw (7, -2) node[] (sI) {$\cf(\UT({}^{\sigma}\tau))$};
\draw[thick, ->] (A) -- node[left] {$\infl^{\UT(\tau)}_{\UL(\tau, A)}$} (I);
\draw[thick, ->] (A) -- node[above] {$(\sigma|_{A_{1}})_{\ast} \otimes \cdots \otimes (\sigma|_{A_{\ell}})_{\ast}$} (sA);
\draw[thick, ->] (I) -- node[above] {$(\sigma)_{\ast}$} (sI);
\draw[thick, ->] (sA) -- node[right] {$\infl^{\UT({}^{\sigma}\tau)}_{\UL({}^{\sigma}\tau, {}^{\sigma}A)}$} (sI);
\end{tikzpicture}
\]
also commutes.
\end{prop}
\begin{proof}
Direct computation gives ${}^{\sigma}\UL(\tau, A) = \UL({}^{\sigma}\tau, {}^{\sigma}A)$ and ${}^{\sigma}\UR(\tau, A) = \UR({}^{\sigma}\tau, {}^{\sigma}A)$, so the proposed equality of maps holds.
\end{proof}


\section{From $\cfUT$ to $\cf(\UT_{\bullet})$}
\label{sec:Hopfalgebra}

This section establishes the titular Hopf algebra structure of class functions.  
This Hopf structure is the shadow of the monoid $\cfUT$ from Section~\ref{sec:cfmonoid}, and Section~\ref{sec:Fockfunctors} gives some technical background on this procedure.  
A mostly self-contained description of the Hopf algebra $\cf(\UT_{\bullet})$ and its structure maps is then given in Section~\ref{sec:cfalgebra}, and Theorem~\ref{thm:HopfAlgebra} verifies that the given maps satisfy the necessary axioms.

\subsection{The Fock functor $\overline{\calK}$}
\label{sec:Fockfunctors}

In~\cite[Chapter 15]{AgMahlong}, a functor $\overline{\calK}$ from $\CC$-vector species to graded $\CC$-vector spaces is defined.  This functor sends each connected Hopf monoid to a graded connected Hopf algebra in a manner described below.

Let $\mathbf{h}$ be a $\CC$-vector species.  For each $I \in \mathrm{set}^{\times}$, the symmetric group on $I$ is $\S{I} = \hom_{\mathrm{set}^{\times}}(I, I)$.  The group $\S{I}$ has a $\CC$-linear action on $\mathbf{h}[I]$ for each Hopf monoid $\mathbf{h}$, giving a space of \emph{$\S{I}$-covariants}
\[
(\mathbf{h}[I])_{\S{}} = \mathbf{h}[I] \big/ \{v - \mathbf{h}[\sigma](v) \;|\; \text{$\sigma \in \S{I}$, $v \in \mathbf{h}[I]$}\}.
\]
The graded vector space associated to $\mathbf{h}$ by the functor $\overline{\calK}$ is 
\[
\overline{\calK}(\mathbf{h}) = \bigoplus_{n \ge 0} (\mathbf{h}[n])_{\S{}},
\]
where $[n]$ is the set containing the first $n$ positive integers, with $[0] = \emptyset$.  

Now suppose that $\mathbf{h}$ is a connected Hopf monoid, let $n \ge 0$, and fix a subset $I \subseteq [n]$.  The \emph{canonicalization map} is the order-preserving bijection
\begin{equation}
\label{eq:canomapdef}
\begin{array}{rccc}
\cano_{I}: & I & \to & [|I|] \\
& r & \mapsto & \#\{ s \in I \;|\; s \le r\}.
\end{array}
\end{equation}
By the functoriality of $\mathbf{h}$, each $\mathbf{h}[\cano_{I}]$ induces an isomorphism between covariant spaces.  The \emph{straightening map} is the isomorphism
\[
\st_{(I, I^{c})} = \mathbf{h}[\cano_{I}] \otimes \mathbf{h}[\cano_{I^{c}}] : (\mathbf{h}[I])_{\S{}} \otimes (\mathbf{h}[I^{c}])_{\S{}} \to (\mathbf{h}[i])_{\S{}} \otimes (\mathbf{h}[n - i])_{\S{}} 
\]
where $i = |I|$ and $I^{c} = [n] \setminus I$.  The naturality of $\mathbf{h}$ also ensures that the product and coproduct of $\mathbf{h}$ also induce maps between $\S{}$-covariants
\[
(\mathbf{h}[I])_{\S{}} \otimes (\mathbf{h}[I^{c}])_{\S{}} \xrightarrow{\;\mu_{(I, I^{c})}\;} (\mathbf{h}[n])_{\S{}}
\qquad \text{and} \qquad
(\mathbf{h}[n])_{\S{}} \xrightarrow{\;\Delta_{(I, I^{c})}\;} (\mathbf{h}[I])_{\S{}} \otimes (\mathbf{h}[I^{c}])_{\S{}}
\]
where $\mu_{\emptyset, [n]}$, $\Delta_{\emptyset, [n]}$, $\mu_{[n], \emptyset}$, and $\Delta_{[n], \emptyset}$ denote the canonical identifications 
\[
\CC \otimes  \mathbf{h}[n]_{\S{}} \cong  \mathbf{h}[n]_{\S{}} \cong  \mathbf{h}[n]_{\S{}} \otimes \CC.
\]

\begin{thm}[{\cite[Theorem 15.12]{AgMahlong}}]
\label{thm:FockFunctor}
Let $\mathbf{h}$ be a connected $\CC$-Hopf monoid over.  Then $\overline{\calK}(\mathbf{h})$ is a graded connected $\CC$-Hopf algebra, with product and coproduct given by
\[
\mu = \bigoplus_{n \ge i \ge 0} \; \mu_{([i], [i]^{c})} \circ \st_{([i], [i]^{c})}^{-1}
\qquad \text{and} \qquad
\Delta = \bigoplus_{n \ge 0} \; \sum_{I \subseteq [n]} \; \st_{(I, I^{c})} \circ \Delta_{(I, I^{c})},
\]
and unit, counit, and antipode determined implicitly by connectedness.
\end{thm}

\begin{rem}
For the Hopf monoids in this paper, applying $\overline{\calK}$ is akin to ``unlabeling.''  Up to isomorphism, $\overline{\calK}(\mathbf{po})$ is the Hopf algebra of unlabelled posets and---as there is only one unlabelled total order of each size---$\overline{\calK}(\mathbf{t})$ is the standard Hopf algebra structure on $\CC[x]$.
\end{rem}

\subsection{The Hopf algebra of class functions}
\label{sec:cfalgebra}

Let $\UT_{n}$ denote the usual group of $n \times n$ unipotent upper triangular matrices over $\FF_{q}$, so that $\UT_{n} = \UT(1 < \cdots < n)$ in the notation of Section~\ref{sec:repprelims}.  
This section describes a Hopf algebra structure on the space
\[
\cf(\UT_{\bullet}) = \bigoplus_{n \ge 0} \cf(\UT_{n}).
\]
Logically, the results below rely heavily on those in Section~\ref{sec:cfmonoid}.  However, this section is written to communicate the construction with minimal reliance on earlier sections.  

For $n \ge 0$ and each subset $I \subseteq [n]$, the group $\UT_{n}$ has subgroups
\begin{align*}
\textbf{Levi:} && \UL_{I} &= \{g \in \UT_{n} \;|\; \text{$(g - 1_{n})_{i, j} \neq 0$ only if $(i, j) \in I \times I$ or $(i, j) \in I^{c} \times I^{c}$} \} \\[1em]
\textbf{radical:} && \UR_{I} &= \{g \in \UT_{n} \;|\; \text{$(g - 1_{n})_{i, j} \neq 0$ only if $(i, j) \in I \times I^{c}$} \},\,\text{and} \\[1em]
\textbf{parabolic:} && \UP_{I} &= \{g \in \UT_{n} \;|\; \text{$(g - 1_{n})_{i, j} = 0$ whenever $(i, j) \in I^{c} \times I$} \}.
\end{align*}
The above subgroups are the same as those defined in Section~\ref{sec:CFalgprelims} using the more cumbersome notation $\UL(1 < \cdots < n, (I, I^{c})^{\sharp})$, $\UR(1 < \cdots < n, (I, I^{c})^{\sharp})$, and $\UP(1 < \cdots < n, (I, I^{c})^{\sharp})$.

\begin{ex}
\label{ex:functorialsubgroups2}
The subgroups $\UL_{I}$, $\UR_{I}$, and $\UP_{I}$ can be constructed in a combinatorial manner laid out in Section~\ref{sec:CFalgprelims}.  For example, with $n = 4$:
\[
\UL_{[2]} = 
\begin{tikzpicture}[scale = 0.45, baseline = 0.45*-2.2cm]
\draw (0.2, 0.085) -- (-0.15, 0.085);
\draw[very thick] (-0.1, 0.1) -- (-0.1, -4.1);
\draw (-0.15, -4.085) -- (0.2, -4.085);
\draw (3.8, 0.085) -- (4.15, 0.085);
\draw[very thick] (4.1, 0.1) -- (4.1, -4.1);
\draw (4.15, -4.085) -- (3.8, -4.085);
\foreach \x in {1, 2, 3, 4}{\foreach \y in {1, 2}{\path[fill = gray!25] (\x - 1, -\y + 1) -- (\x, -\y + 1) -- (\x, -\y) -- (\x - 1, -\y) -- cycle;}}
\foreach \x in {3, 4}{\foreach \y in {3, 4}{\path[fill = gray!80] (\x - 1, -\y + 1) -- (\x, -\y + 1) -- (\x, -\y) -- (\x - 1, -\y) -- cycle;}}
\foreach \x in {3, 4}{\foreach \y in {1, 2}{\path[pattern = north east lines, pattern color = gray] (\x - 1, -\y + 1) -- (\x, -\y + 1) -- (\x, -\y) -- (\x - 1, -\y) -- cycle;}}
\foreach \x in {1, 2, 3, 4}{\draw (\x - 0.5, -\x + 0.5) node {$1$};}
\foreach \x in {1}{\foreach \y in {2, 3, 4}{\draw (\x - 0.5, -\y + 0.5) node {$0$};}}
\foreach \x in {2}{\foreach \y in {3, 4}{\draw (\x - 0.5, -\y + 0.5) node {$0$};}}
\foreach \x in {3}{\foreach \y in {4}{\draw (\x - 0.5, -\y + 0.5) node {$0$};}}
\foreach \x in {2}{\foreach \y in {1}{\draw (\x - 0.5, -\y + 0.5) node {$\ast$};}}
\foreach \x in {4}{\foreach \y in {3}{\draw (\x - 0.5, -\y + 0.5) node {$\ast$};}}
\foreach \x in {3, 4}{\foreach \y in {1}{\draw (\x - 0.5, -\y + 0.5) node {$0$};}}
\foreach \x in {3, 4}{\foreach \y in {2}{\draw (\x - 0.5, -\y + 0.5) node {$0$};}}
\end{tikzpicture},
\quad \UR_{[2]} = 
\begin{tikzpicture}[scale = 0.45, baseline = 0.45*-2.2cm]
\draw (0.2, 0.085) -- (-0.15, 0.085);
\draw[very thick] (-0.1, 0.1) -- (-0.1, -4.1);
\draw (-0.15, -4.085) -- (0.2, -4.085);
\draw (3.8, 0.085) -- (4.15, 0.085);
\draw[very thick] (4.1, 0.1) -- (4.1, -4.1);
\draw (4.15, -4.085) -- (3.8, -4.085);
\foreach \x in {1, 2, 3, 4}{\foreach \y in {1, 2}{\path[fill = gray!25] (\x - 1, -\y + 1) -- (\x, -\y + 1) -- (\x, -\y) -- (\x - 1, -\y) -- cycle;}}
\foreach \x in {3, 4}{\foreach \y in {3, 4}{\path[fill = gray!80] (\x - 1, -\y + 1) -- (\x, -\y + 1) -- (\x, -\y) -- (\x - 1, -\y) -- cycle;}}
\foreach \x in {3, 4}{\foreach \y in {1, 2}{\path[pattern = north east lines, pattern color = gray] (\x - 1, -\y + 1) -- (\x, -\y + 1) -- (\x, -\y) -- (\x - 1, -\y) -- cycle;}}
\foreach \x in {1, 2, 3, 4}{\draw (\x - 0.5, -\x + 0.5) node {$1$};}
\foreach \x in {1}{\foreach \y in {2, 3, 4}{\draw (\x - 0.5, -\y + 0.5) node {$0$};}}
\foreach \x in {2}{\foreach \y in {3, 4}{\draw (\x - 0.5, -\y + 0.5) node {$0$};}}
\foreach \x in {3}{\foreach \y in {4}{\draw (\x - 0.5, -\y + 0.5) node {$0$};}}
\foreach \x in {2}{\foreach \y in {1}{\draw (\x - 0.5, -\y + 0.5) node {$0$};}}
\foreach \x in {4}{\foreach \y in {3}{\draw (\x - 0.5, -\y + 0.5) node {$0$};}}
\foreach \x in {3, 4}{\foreach \y in {1}{\draw (\x - 0.5, -\y + 0.5) node {$\ast$};}}
\foreach \x in {3, 4}{\foreach \y in {2}{\draw (\x - 0.5, -\y + 0.5) node {$\ast$};}}
\end{tikzpicture},
\quad \text{and} \quad \UP_{[2]} = 
\begin{tikzpicture}[scale = 0.45, baseline = 0.45*-2.2cm]
\draw (0.2, 0.085) -- (-0.15, 0.085);
\draw[very thick] (-0.1, 0.1) -- (-0.1, -4.1);
\draw (-0.15, -4.085) -- (0.2, -4.085);
\draw (3.8, 0.085) -- (4.15, 0.085);
\draw[very thick] (4.1, 0.1) -- (4.1, -4.1);
\draw (4.15, -4.085) -- (3.8, -4.085);
\foreach \x in {1, 2, 3, 4}{\foreach \y in {1, 2}{\path[fill = gray!25] (\x - 1, -\y + 1) -- (\x, -\y + 1) -- (\x, -\y) -- (\x - 1, -\y) -- cycle;}}
\foreach \x in {3, 4}{\foreach \y in {3, 4}{\path[fill = gray!80] (\x - 1, -\y + 1) -- (\x, -\y + 1) -- (\x, -\y) -- (\x - 1, -\y) -- cycle;}}
\foreach \x in {3, 4}{\foreach \y in {1, 2}{\path[pattern = north east lines, pattern color = gray] (\x - 1, -\y + 1) -- (\x, -\y + 1) -- (\x, -\y) -- (\x - 1, -\y) -- cycle;}}
\foreach \x in {1, 2, 3, 4}{\draw (\x - 0.5, -\x + 0.5) node {$1$};}
\foreach \x in {1}{\foreach \y in {2, 3, 4}{\draw (\x - 0.5, -\y + 0.5) node {$0$};}}
\foreach \x in {2}{\foreach \y in {3, 4}{\draw (\x - 0.5, -\y + 0.5) node {$0$};}}
\foreach \x in {3}{\foreach \y in {4}{\draw (\x - 0.5, -\y + 0.5) node {$0$};}}
\foreach \x in {2}{\foreach \y in {1}{\draw (\x - 0.5, -\y + 0.5) node {$\ast$};}}
\foreach \x in {4}{\foreach \y in {3}{\draw (\x - 0.5, -\y + 0.5) node {$\ast$};}}
\foreach \x in {3, 4}{\foreach \y in {1}{\draw (\x - 0.5, -\y + 0.5) node {$\ast$};}}
\foreach \x in {3, 4}{\foreach \y in {2}{\draw (\x - 0.5, -\y + 0.5) node {$\ast$};}}
\end{tikzpicture}
\]
and
\[
\UL_{\{2, 4\}} = 
\begin{tikzpicture}[scale = 0.45, baseline = 0.45*-2.2cm]
\draw (0.2, 0.085) -- (-0.15, 0.085);
\draw[very thick] (-0.1, 0.1) -- (-0.1, -4.1);
\draw (-0.15, -4.085) -- (0.2, -4.085);
\draw (3.8, 0.085) -- (4.15, 0.085);
\draw[very thick] (4.1, 0.1) -- (4.1, -4.1);
\draw (4.15, -4.085) -- (3.8, -4.085);
\foreach \x in {1, 2, 3, 4}{\foreach \y in {2, 4}{\path[fill = gray!25] (\x - 1, -\y + 1) -- (\x, -\y + 1) -- (\x, -\y) -- (\x - 1, -\y) -- cycle;}}
\foreach \x in {1, 3}{\foreach \y in {1, 3}{\path[fill = gray!80] (\x - 1, -\y + 1) -- (\x, -\y + 1) -- (\x, -\y) -- (\x - 1, -\y) -- cycle;}}
\foreach \x in {1, 3}{\foreach \y in {2, 4}{\path[pattern = north east lines, pattern color = gray] (\x - 1, -\y + 1) -- (\x, -\y + 1) -- (\x, -\y) -- (\x - 1, -\y) -- cycle;}}
\foreach \x in {1, 2, 3, 4}{\draw (\x - 0.5, -\x + 0.5) node {$1$};}
\foreach \x in {1}{\foreach \y in {2, 3, 4}{\draw (\x - 0.5, -\y + 0.5) node {$0$};}}
\foreach \x in {2}{\foreach \y in {3, 4}{\draw (\x - 0.5, -\y + 0.5) node {$0$};}}
\foreach \x in {3}{\foreach \y in {4}{\draw (\x - 0.5, -\y + 0.5) node {$0$};}}
\foreach \x in {3}{\foreach \y in {1}{\draw (\x - 0.5, -\y + 0.5) node {$\ast$};}}
\foreach \x in {4}{\foreach \y in {2}{\draw (\x - 0.5, -\y + 0.5) node {$\ast$};}}
\foreach \x in {3}{\foreach \y in {2}{\draw (\x - 0.5, -\y + 0.5) node {$0$};}}
\foreach \x in {2}{\foreach \y in {1}{\draw (\x - 0.5, -\y + 0.5) node {$0$};}}
\foreach \x in {4}{\foreach \y in {1, 3}{\draw (\x - 0.5, -\y + 0.5) node {$0$};}}
\end{tikzpicture},
\quad \UR_{\{2, 4\}} = 
\begin{tikzpicture}[scale = 0.45, baseline = 0.45*-2.2cm]
\draw (0.2, 0.085) -- (-0.15, 0.085);
\draw[very thick] (-0.1, 0.1) -- (-0.1, -4.1);
\draw (-0.15, -4.085) -- (0.2, -4.085);
\draw (3.8, 0.085) -- (4.15, 0.085);
\draw[very thick] (4.1, 0.1) -- (4.1, -4.1);
\draw (4.15, -4.085) -- (3.8, -4.085);
\foreach \x in {1, 2, 3, 4}{\foreach \y in {2, 4}{\path[fill = gray!25] (\x - 1, -\y + 1) -- (\x, -\y + 1) -- (\x, -\y) -- (\x - 1, -\y) -- cycle;}}
\foreach \x in {1, 3}{\foreach \y in {1, 3}{\path[fill = gray!80] (\x - 1, -\y + 1) -- (\x, -\y + 1) -- (\x, -\y) -- (\x - 1, -\y) -- cycle;}}
\foreach \x in {1, 3}{\foreach \y in {2, 4}{\path[pattern = north east lines, pattern color = gray] (\x - 1, -\y + 1) -- (\x, -\y + 1) -- (\x, -\y) -- (\x - 1, -\y) -- cycle;}}
\foreach \x in {1, 2, 3, 4}{\draw (\x - 0.5, -\x + 0.5) node {$1$};}
\foreach \x in {1}{\foreach \y in {2, 3, 4}{\draw (\x - 0.5, -\y + 0.5) node {$0$};}}
\foreach \x in {2}{\foreach \y in {3, 4}{\draw (\x - 0.5, -\y + 0.5) node {$0$};}}
\foreach \x in {3}{\foreach \y in {4}{\draw (\x - 0.5, -\y + 0.5) node {$0$};}}
\foreach \x in {3}{\foreach \y in {1}{\draw (\x - 0.5, -\y + 0.5) node {$0$};}}
\foreach \x in {4}{\foreach \y in {2}{\draw (\x - 0.5, -\y + 0.5) node {$0$};}}
\foreach \x in {3}{\foreach \y in {2}{\draw (\x - 0.5, -\y + 0.5) node {$\ast$};}}
\foreach \x in {2}{\foreach \y in {1}{\draw (\x - 0.5, -\y + 0.5) node {$0$};}}
\foreach \x in {4}{\foreach \y in {1, 3}{\draw (\x - 0.5, -\y + 0.5) node {$0$};}}
\end{tikzpicture},
\quad \text{and} \quad \UP_{\{2, 4\}} = 
\begin{tikzpicture}[scale = 0.45, baseline = 0.45*-2.2cm]
\draw (0.2, 0.085) -- (-0.15, 0.085);
\draw[very thick] (-0.1, 0.1) -- (-0.1, -4.1);
\draw (-0.15, -4.085) -- (0.2, -4.085);
\draw (3.8, 0.085) -- (4.15, 0.085);
\draw[very thick] (4.1, 0.1) -- (4.1, -4.1);
\draw (4.15, -4.085) -- (3.8, -4.085);
\foreach \x in {1, 2, 3, 4}{\foreach \y in {2, 4}{\path[fill = gray!25] (\x - 1, -\y + 1) -- (\x, -\y + 1) -- (\x, -\y) -- (\x - 1, -\y) -- cycle;}}
\foreach \x in {1, 3}{\foreach \y in {1, 3}{\path[fill = gray!80] (\x - 1, -\y + 1) -- (\x, -\y + 1) -- (\x, -\y) -- (\x - 1, -\y) -- cycle;}}
\foreach \x in {1, 3}{\foreach \y in {2, 4}{\path[pattern = north east lines, pattern color = gray] (\x - 1, -\y + 1) -- (\x, -\y + 1) -- (\x, -\y) -- (\x - 1, -\y) -- cycle;}}
\foreach \x in {1, 2, 3, 4}{\draw (\x - 0.5, -\x + 0.5) node {$1$};}
\foreach \x in {1}{\foreach \y in {2, 3, 4}{\draw (\x - 0.5, -\y + 0.5) node {$0$};}}
\foreach \x in {2}{\foreach \y in {3, 4}{\draw (\x - 0.5, -\y + 0.5) node {$0$};}}
\foreach \x in {3}{\foreach \y in {4}{\draw (\x - 0.5, -\y + 0.5) node {$0$};}}
\foreach \x in {4}{\foreach \y in {3}{\draw (\x - 0.5, -\y + 0.5) node {$0$};}}
\foreach \x in {3}{\foreach \y in {1}{\draw (\x - 0.5, -\y + 0.5) node {$\ast$};}}
\foreach \x in {4}{\foreach \y in {2}{\draw (\x - 0.5, -\y + 0.5) node {$\ast$};}}
\foreach \x in {3}{\foreach \y in {2}{\draw (\x - 0.5, -\y + 0.5) node {$\ast$};}}
\foreach \x in {2}{\foreach \y in {1}{\draw (\x - 0.5, -\y + 0.5) node {$0$};}}
\foreach \x in {4}{\foreach \y in {1, 3}{\draw (\x - 0.5, -\y + 0.5) node {$0$};}}
%
\end{tikzpicture}.
\]
\end{ex}

\begin{rem}
There is some ambiguity in the notation above, as $I \subseteq [n]$ is also a subset of $[n+1]$, $[n+2]$, and so on.  Thus, the reader should take the phrase ``$I \subseteq [n]$'' to imply that all subgroups defined by the subset $I$ are of $\UT_{n}$.
\end{rem}

In the updated notation of this section, Proposition~\ref{lem:functorialsubgroupinteractions} states that for $n \ge 0$ and $I \subseteq [n]$, 
\[
\UP_{I} = \UL_{I} \rtimes \UR_{I}.
\]
Further, Equation~\eqref{eq:leviiso} gives an isomorphism of groups $\UL_{I} \cong \UT_{|I|} \times \UT_{|I^{c}|}$ which descends to a straightening isomorphism 
\begin{equation}
\label{eq:stmapnice}
\st_{(I, I^{c})}: \cf(\UL_{I}) \to \cf( \UT_{|I|} ) \otimes \cf(\UT_{|I^{c}|}), 
\end{equation}
as defined in Section~\ref{sec:Fockfunctors}.

\begin{thm}
\label{thm:HopfAlgebra}
The space $\cf(\UT_{\bullet})$ is a graded connected Hopf algebra with the product and coproduct maps
\[
\mu = \bigoplus_{n \ge i \ge 0} \infl^{\UT_{n}}_{\UL_{[i]}} \circ  \st_{([i], [i]^{c})}^{-1}
\qquad\text{and}\qquad
\Delta = \bigoplus_{n \ge 0} \; \sum_{I \subseteq [n]} \; \st_{(I, I^{c})} \circ \defl^{\UP_{I}}_{\UL_{I}} \circ \res^{\UT_{n}}_{\UP_{I}}.
\]
Moreover, there is a graded isomorphism $\cf(\UT_{\bullet}) \cong \overline{\calK}(\cfUT)$ of Hopf algebras.
\end{thm}
\begin{proof}
The proof will exhbit an isomorphism $\cf(\UT_{\bullet}) \cong \overline{\calK}(\cfUT)$.  For $n \ge 0$, the symmetric group $\S{[n]}$ acts freely and transitively on the set $\TTT[n]$ of total orders.  Thus 
\[
\cfUT[n] = \bigoplus_{\sigma \in \S{[n]}} \cfUT[\sigma]\left( \cf(\UT_{n}) \right),
\]
and consequently, inclusion $\cf(\UT_{n}) \hookrightarrow \cfUT[n]$ gives rise to a $\CC$-linear bijection
\[
\begin{array}{rcl}
q_{n}\colon \cf(\UT_{n}) & \to & \cfUT[n]_{\S{}} \\
\psi & \mapsto & \psi +  \{\chi - \cfUT[\sigma](\chi) \;|\; \text{$\sigma \in \S{[n]}$, $\chi \in \cfUT[n]$}\}.
\end{array}
\]
Thus, $\cf(\UT_{\bullet})$ and $\overline{\calK}(\cfUT)$ are isomorphic as graded vector spaces.  Furthermore, direct computation shows that for each $n \ge i \ge 0$,
\[
\mu_{([i], [i]^{c})} \circ \st_{([i], [i]^{c})}^{-1} \circ (q_{i} \otimes q_{n-i}) 
=
q_{n} \circ \infl^{\UT_{n}}_{\UL_{[i]}} \circ  \st_{([i], [i]^{c})}^{-1},
\]
and for each subset $I \subseteq [n]$, 
\[
\st_{(I, I^{c})} \circ \Delta_{(I, I^{c})} \circ q_{n} 
=
(q_{|I|} \otimes q_{|I^{c}|}) \circ \st_{(I, I^{c})} \circ \defl^{\UP_{I}}_{\UL_{I}} \circ \res^{\UT_{n}}_{\UP_{I}} ,
\]
making the map $\bigoplus_{n \ge 0} q_{n}$ a Hopf algebra isomorphism.
\end{proof}

\section{Induction to the general linear group}
\label{sec:GL}

For $n \ge 0$, let $\GL_{n} = \GL([n])$, the usual $n \times n$ general linear group over $\FF_{q}$.  Zelevinsky~\cite{Zel} defined a graded connected Hopf algebra structure on
\[
\cf(\GL_{\bullet}) = \bigoplus_{n \ge 0} \cf(\GL_{n}),
\]
the details of which are reviewed in Section~\ref{sec:GLalg}.  Recalling the content of Section~\ref{sec:Hopfalgebra}, $\UT_{n}$ is a subgroup of $\GL_{n}$ so that induction gives a graded linear map
\[
\ind^{\GL}_{\UT} = \bigoplus_{n \ge 0} \ind^{\GL_{n}}_{\UT_{n}} : \cf(\UT_{\bullet}) \to \cf(\GL_{\bullet}).
\]
The purpose of this section is to investigate the properties of this map.  First and foremost, Section~\ref{sec:inductionhom} proves the following result.

\begin{thm}
\label{thm:inductionhomomorphism}
The map $\ind^{\GL}_{\UT}$ is a graded Hopf algebra homomorphism.
\end{thm}

A second result provides useful information about the kernel of $\ind^{\GL}_{\UT}$, and may also prove to be useful in adapting this approach to the orthogonal and symplectic types following~\cite{vanLeeuwen}; refer to the introduction for a discussion of this prospect. 

Each general linear group has a ``diagram automorphism'' (actually an involutionary group antiautomorphism) which I will write as $\dagger$.  The unipotent upper triangular group $\UT_{n} \subseteq \GL_{n}$ is stable under this involution, so $\dagger$ induces a graded linear map
\begin{equation}
\label{eq:daggerautdef}
\daggeraut = \bigoplus_{n \ge 0} \daggeraut_{n} : \cf(\UT_{\bullet}) \to \cf(\UT_{\bullet})
\qquad \text{with} \qquad
\begin{array}{rccc}
\daggeraut_{n}: & \cf(\UT_{n}) & \to & \cf(\UT_{n}) \\
& \psi & \mapsto & \psi \circ \dagger.
\end{array}
\end{equation}
See Section~\ref{sec:antiautomorphism} for more details, and a proof of the following result.

\begin{thm}
\label{thm:daggeraut}
The map $\daggeraut$ is a graded antiautomorphism of $\cf(\UT_{\bullet})$, and moreover $\ind^{\GL}_{\UT}$ is $\daggeraut$-invariant, in the sense that $\ind^{\GL}_{\UT} \circ \daggeraut = \ind^{\GL}_{\UT}$.
\end{thm}

\subsection{Review of $\cf(\GL)$ as a Hopf algebra}
\label{sec:GLalg}

This section will recall relevant details about the Hopf algebra $\cf(\GL_{\bullet})$, following~\cite{Zel}.  
The structure maps will make use of the following subgroups of $\GL_{n}$, for each $n \ge i \ge 0$:
\begin{align*}
\textbf{(Levi)} && \L{i} = \begin{bmatrix} \GL_{i} & 0 \\ 0 & \GL_{n-i} \end{bmatrix} &= \left\{g \in \GL_{n} \;\middle|\; \hspace{-0.35em} \begin{array}{c} \text{$g_{r, s} \neq 0$ only if $(r, s) \in [i] \times [i]$} \\ \text{or $(r, s) \in [i]^{c} \times [i]^{c}$}  \end{array} \hspace{-0.35em}  \right\}, \\[1em]
\textbf{(radical)} && \R{i} = \begin{bmatrix} 1_{i} & \ast \\ 0 & 1_{n-i} \end{bmatrix} &= \left\{g \in \GL_{n} \;\middle|\; \hspace{-0.35em}  \begin{array}{c} \text{$(g-1_{n})_{r, s} \neq 0$ only if $(r, s) \in [i] \times [i]^{c}$}\end{array}  \hspace{-0.35em} \right\}, \\[1em]
\textbf{(parabolic)} && \P{i} = \begin{bmatrix} \GL_{i} & \ast \\ 0 & \GL_{n-i} \end{bmatrix} &= \left\{g \in \GL_{n} \;\middle|\; \hspace{-0.35em}  \begin{array}{c} \text{$g_{r, s} = 0$ whenever $(r, s) \in [i]^{c} \times [i]$} \end{array}  \hspace{-0.35em} \right\}.
\end{align*}
These definitions should be compared to the subgroups $\UL_{I}$, $\UR_{I}$, and $\UP_{I}$ of $\UT_{n}$ defined in Section~\ref{sec:cfalgebra}; in particular the definition of $\R{i}$ is exactly the same as the subgroup $\UL_{[i]} \subseteq \UT_{n}$.  Moreover, there is an analogous semidirect product structure
\[
\P{i} = \L{i} \ltimes \R{i},
\]
and $\L[n]{i} \cong \GL_{i} \times \GL_{n-i}$.  In an abuse of notation, I will write
\[
\st_{([i], [i]^{c})}: \cf(\L[n]{i}) \to \cf(\GL_{i}) \otimes \cf(\GL_{n-i})
\]
for the invertible linear map induced by this isomorphism, in addition to the one defined in Equation~\eqref{eq:stmapnice} for unipotent groups.

\begin{thm}[{\cite[Section 9.1]{Zel}}]
The space $\cf(\GL_{\bullet})$ is a graded connected Hopf algebra, with product and coproduct
\[
\mu = \bigoplus_{n \ge i \ge 0}  \ind^{\GL_{n}}_{\P{i}} \circ \infl^{\P{i}}_{\L{i}} \circ \st_{([i], [i]^{c})}^{-1}
\qquad\text{and}\qquad
\Delta = \bigoplus_{n \ge 0} \; \sum_{i = 0}^{n}  \st_{([i], [i]^{c})} \circ \defl^{\P{i}}_{\L{i}} \circ \res^{\GL_{n}}_{\P{i}}.
\]
\end{thm}

For a more detailed---and combinatorial---account of $\cf(\GL_{\bullet})$, the reader should refer to~\cite[\nopp IV]{Mac} or~\cite{Zel}. However, no further details are needed in the scope of this paper.

\subsection{Proof of Theorem~\ref{thm:inductionhomomorphism}}
\label{sec:inductionhom}

This section will give a proof Theorem~\ref{thm:inductionhomomorphism}, which follows two lemmas and some intervening definitions.  The first lemma shows that $\ind^{\GL}_{\UT}$ is an algebra homomorphism; since the algebra structure on $\cf(\UT_{\bullet})$ is the same as that in~\cite{AgEtAl}, this result is not original.  However, it has not been published elsewhere, so I have included a proof for the sake of completeness.

\begin{lem}
\label{lem:algebrahom}
The map $\ind^{\GL}_{\UT}: \cf(\UT_{\bullet}) \shortto \cf(\GL_{\bullet})$ is a graded algebra homomorphism.
\end{lem}
\begin{proof}
By definition, $\ind^{\GL}_{\UT}$ is a graded linear map, so what remains is to show that it is multiplicative, in the sense that 
\[
\ind^{\GL}_{\UT}  \circ \mu = \mu \circ (\ind^{\GL}_{\UT} \otimes \ind^{\GL}_{\UT}).
\]
Fixing $n \ge i \ge 0$ and comparing the multidegree $(i, n-i)$ components, the equation above is equivalent to
\[
\ind^{\GL_{n}}_{\UT_{n}} \circ \infl^{\UT_{n}}_{\UL_{[i]}} \circ \st_{([i], [i]^{c})}^{-1} 
=
 \ind^{\GL_{n}}_{\P{i}} \circ \infl^{\P{i}}_{\L{i}} \circ \st_{([i], [i]^{c})}^{-1} \circ  \ind^{\GL_{i} \times \GL_{n - i}}_{\UT_{i} \times \UT_{n - i}}
\]
This statement is implied by a simpler condition; after using the identities
\[
\ind^{\GL_{n}}_{\UT_{n}} = \ind^{\GL_{n}}_{\P[n]{i}} \circ \ind^{\P[n]{i}}_{\UT_{n}}
\qquad\text{and}\qquad
\st_{([i], [i]^{c})}^{-1} \circ \ind^{\GL_{i} \times \GL_{n - i}}_{\UT_{i} \times \UT_{n - i}} =  \ind^{\L{i}}_{\UL_{([i], [i]^{n})}} \circ  \st_{([i], [i]^{c})}^{-1},
\]
and cancelling the terms $\ind^{\GL_{n}}_{\P[n]{i}}$ and $\st_{([i], [i]^{c})}^{-1}$ from both expressions, one obtains
\[
\ind^{\P[n]{i}}_{\UT_{n}} \circ \infl^{\UT_{n}}_{\UL_{[i]}} = \infl^{\P[n]{i}}_{\L{i}} \circ  \ind^{\L{i}}_{\UL_{[i]}};
\]
the remainder of the proof establishes this identity.

For $p \in \P{i}$ write $p_{L} \in \L{i}$ and $p_{R} \in \R{i}$ for the terms in the factorization $p = p_{L}p_{R}$ guaranteed by the semidirect product $\P{i} = \L{i} \rtimes \R{i}$.  For $\psi \in \cf(\UL_{[i]})$ and $g \in \P{i}$, 
\[
\ind^{\P[n]{i}}_{\UT_{n}} \circ \infl^{\UT_{n}}_{\UL_{[i]}} (\psi)(g) = \frac{1}{|\UT_{n}|} \sum_{\substack{x \in \P{i} \\ x gx^{-1} \in \UT_{n}  }} \psi\big((x gx^{-1})_{L} \big)
\]
and 
\[
\infl^{\P[n]{i}}_{\L{i}} \circ  \ind^{\L{i}}_{\UL_{[i]}} (\psi)(g) = \frac{1}{|\UL_{[i]} |} \hspace{-0.5em} \sum_{\substack{y \in \L{i} \\ y g_{L} y^{-1} \in \UL_{[i]}  }} \hspace{-1em} \psi(y g_{L} y^{-1}).
\]
The above expressions can be seen to be equal as follows: for $x \in \P{i}$, $(xgx^{-1})_{L} = x_{L} g_{L} (x_{L})^{-1}$, and $\UL_{[i]} = \L{i} \cap \UT_{n}$, so
\[
\frac{1}{|\UT_{n}|} \sum_{\substack{x \in \P{i} \\ x gx^{-1} \in \UT_{n}  }}  \hspace{-0.5em} \psi\big((x gx^{-1})_{L} \big) 
= \frac{1}{|\UT_{n}|} \hspace{-1.5em} \sum_{\substack{x_{L} \in \L{i} \\ x_{R} \in \R{i} \\ x_{L} g_{L}(x_{L}) ^{-1}  \in \UL_{[i]}   }} \hspace{-1.5em}  \psi\big(x_{L} g_{L} (x_{L})^{-1} \big)
= \frac{|\R{i}|}{|\UT_{n} |} \hspace{-1em} \sum_{\substack{y \in \L{i} \\ y g_{L} y^{-1} \in \UL_{[i]}  }} \hspace{-1em} \psi(y g_{L} y^{-1}),
\]
and moreover $\R{i} = \UR_{[i]}$, so that $|\UT_{n}:\R{i}| = |\UL_{[i]}|$.
\end{proof}

The remainder of the section builds the necessary framework to establish that $\ind^{\GL}_{\UT}$ is a coalgebra homomorphism, which is done in the proof of Theorem~\ref{thm:inductionhomomorphism} following Lemma~\ref{lem:subgroupintersection}.  
A focal point for this proof is the double coset decomposition
\[
\UT_{n} \backslash \GL_{n} / \P{i},
\]
which I will now review.  
The Bruhat decomposition (see e.g.~\cite[173]{Zel}) gives a bijection for each $n \ge i \ge 0$:
\[
\begin{array}{rcl}
\left\{\begin{array}{c} \text{Double cosets} \\ \text{$\UT_{n} \backslash \GL_{n} / \P{i} $} \end{array}\right\} & \longleftrightarrow & \left\{\begin{array}{c} \text{Left cosets} \\ \text{$\S{[n]}/(\S{[i]} \times \S{[i]^{c}})$} \end{array}\right\} \\[1em]
\UT_{n} \; x \; \P{i}  & \mapsto & \S{[n]} \cap (\UT_{n} \; x \; \P{i})  \\[0.25em]
\UT_{n} \;\sigma\; \P{i}  & \mapsfrom & \sigma(\S{[i]} \times \S{[i]^{c}})
\end{array}
\]
where $\S{[n]}$ is identified with the subset of permutation matrices in $\GL_{n}$.  
A distinguished set of left coset representatives for $\S{[n]}/(\S{[i]} \times \S{[i]^{c}})$, $n \ge i \ge 0$, can be constructed as follows.  For $I \subseteq [n]$,  let $\w{I} \in \S{[n]}$ be the unique permutation for which 
\[
{}^{\w{I}}(1 < \cdots < n) = i_{1} < \cdots < i_{|I|} < j_{1} < \cdots < j_{n-|I|},
\]
where $i_{1} < \cdots < i_{|I|} \in \TTT[I]$ and $j_{1} < \cdots < j_{n-|I|} \in \TTT[I^{c}]$ are the usual (numerical) orders on $I$ and $I^{c}$.  The symmetric group $\S{[n]}$ acts by permutation of the $i$-element subsets of $[n]$, and the stabilizer of $[i]$ is $\S{[i]} \times \S{[i]^{c}}$, giving the following bijection:
\[
\begin{array}{ccc}
\left\{ \begin{array}{c} \text{Subsets $I \subseteq [n]$} \\ \text{ with $|I| = i$} \end{array} \right\}
& \to &
 \left\{\begin{array}{c} \text{Left cosets of} \\ \text{$\S{[n]}/(\S{[i]} \times \S{[i]^{c}})$} \end{array}\right\} \\[1em]
I & \mapsto & \w{I} (\S{[i]} \times \S{[i]^{c}})) \\ 
\sigma([i]) & \mapsfrom & \sigma (\S{[i]} \times \S{[i]^{c}})) ,
\end{array}
\]  
so the set $\{\w{I} \;|\; \text{$I \subseteq [n]$ and $|I| = i$}\}$ is also a complete set of double coset representatives in $\UT_{n} \backslash \GL_{n} / \P{i}$.  In addition to being canonical, this choice of (double) coset representatives is justified by the following result; recall the groups $\UL_{I}$, $\UR_{I}$, and $\UP_{I}$ defined in Section~\ref{sec:cfalgebra}.  

\begin{lem}
\label{lem:subgroupintersection}
Let $n \ge i \ge 0$ and fix $I \subseteq [n]$ with $|I| = i$.  Then
\[
\UL_{I} = {}^{\w{I}}\L{i} \cap \UT_{n}, \qquad \UR_{I} = {}^{\w{I}}\R{i} \cap \UT_{n},
\]
and 
\[
\UP_{I} = {}^{\w{I}}\P{i} \cap \UT_{n}.
\]
\end{lem}
\begin{proof}
Each assertion follows from the relevant definitions; for example the fact that 
\[
\{(\w{I}(r), \w{I}(s)) \;|\; (r, s) \in [i]^{c} \times [i] \} = I^{c} \times I
\]
for each $I \subseteq [n]$ implies that $\UP_{I} = {}^{\w{I}}\P{i} \cap \UT_{n}$.
\end{proof}

\begin{proof}[Proof of Theorem~\ref{thm:inductionhomomorphism}]
Lemma~\ref{lem:algebrahom} shows that $\ind^{\GL}_{\UT}$ is a graded algebra homomorphism, so what remains is to show that it is a coalgebra homomorphism, meaning that
\[
\Delta \circ \ind^{\GL}_{\UT} = \big(\ind^{\GL}_{\UT} \otimes \ind^{\GL}_{\UT} \big)\circ \Delta.
\]
It is sufficient to compare the components of these maps between the (multi-) degree $n$ and $(i, n-i)$ components of $\cf(\UT_{\bullet})$ and $\cf(\GL_{\bullet}) \otimes \cf(\GL_{\bullet})$ for each $n \ge i \ge 0$; these are respectively
\[
\st_{([i], [i]^{c})} \circ \defl^{\P{i}}_{\L{i}} \circ \res^{\GL_{n}}_{\P{i}} \circ \ind^{\GL_{n}}_{\UT_{n}}
\]
and
\[
\sum_{\substack{I \subseteq [n] \\ |I| = i }} \ind^{\GL_{i} \times \GL_{n - i}}_{\UT_{i} \times \UT_{n-i}} \circ  \st_{(I, I^{c})} \circ \defl^{\UP_{I}}_{\UL_{I}} \circ \res^{\UT_{n}}_{\UP_{I}}.
\]
In order to show that the expressions above are equal, some simplification is required; first, with Lemma~\ref{lem:subgroupintersection} and the preceding discussion, the Mackey formula states that
\[
 \res^{\GL_{n}}_{\P{i}} \circ \ind^{\GL_{n}}_{\UT_{n}} = \sum_{\substack{I \subseteq [n] \\ |I| = i }} \ind^{\P{i}}_{ {}^{\w{i}^{-1}} \UP_{I} } \circ (\w{I})^{\ast} \circ \res^{\UT_{n}}_{\UP_{I} },
\]
while direct computations yields the equations
\[
\st_{(I, I^{c})} =  \st_{([i], [i]^{c})} \circ (\w{I})^{\ast}
\qquad\text{and}\qquad
\ind^{\GL_{i} \times \GL_{n - i}}_{\UT_{i} \times \UT_{n-i}}  \circ \st_{([i], [i]^{c})} = \st_{([i], [i]^{c})} \circ \ind^{\L{i}}_{\UL_{[i]}}.
\]
With the identities above, it is sufficient to show that for each $i$-element subset $I \subseteq [n]$,
\begin{equation}
\label{eq:GLsol}
\defl^{\P{i}}_{\L{i}} \circ  \ind^{\P{i}}_{ {}^{\w{I}^{-1}} \UP_{I} } \circ (\w{I})^{\ast} 
=
\ind^{\L{i}}_{\UL_{[i]}} \circ (\w{I})^{\ast} \circ \defl^{ \UP_{I}}_{\UL_{I}},
\end{equation}
which is the content of the remainder of the proof.

To show Equation~\eqref{eq:GLsol}, fix $\psi \in \cf(\UP_{I})$ and $g \in \L{i}$.  Then
\[
\defl^{\P{i}}_{\L{i}} \circ  \ind^{\P{i}}_{ {}^{\w{I}^{-1}} \UP_{I} } \circ (\w{I})^{\ast} (\psi)(g) 
= \frac{1}{|\R{i}| |\UP_{I}|} \hspace{-1.75em} \sum_{\substack{x \in \P{i} \\ y \in \R{i} \\ x g y x^{-1} \in {}^{\w{I}^{-1} }\UP_{I} }} \hspace{-2em} \psi (\w{I} x g y x^{-1} \w{I}^{-1})
\]
and
\[
\ind^{\L{i}}_{\UL_{[i]}} \circ (\w{I})^{\ast} \circ \defl^{ \UP_{I}}_{\UL_{I}}(\psi)(g) 
= \frac{1}{|\UL_{[i]}| |\UR_{I}|} \hspace{-1em} \sum_{\substack{ x \in \L{i} \\ y \in \UR_{I} \\ xgx^{-1} \in \UL_{[i]} }} \hspace{-1.5em}  \psi( \w{I} xgx^{-1} \w{I}^{-1} y ).
\]
Now, consider the map
\[
\begin{array}{rcl}
\pi: \P{i} \times \R{i} & \to & \L{i} \times {}^{\w{I}} \R{i} \\[0.5em]
(x, y) & \mapsto & (x_{L},\, \w{I} x_{L} g^{-1} x_{R} g y x^{-1}\w{I}^{-1} ),
\end{array}
\]
which is well defined due to the fact that $\R{i} \trianglelefteq \P{i}$, so that successively $g^{-1} x_{R} g \in \R{i}$, $g^{-1} x_{R} g y (x_{R})^{-1} \in \R{i}$, and thus $x_{L} g^{-1} x_{R} g y x^{-1} \in \R{i}$.  The map $\pi$ is also surjective: for each $(\tilde{x}, \tilde{y}) \in \L{i} \times \UR_{[i]}$,
\[
\pi^{-1}(\tilde{x}, \tilde{y}) = \{ (x,\, g^{-1} (x_{R})^{-1} g (x_{L})^{-1} \w{I}^{-1} \tilde{y} \w{I} x) \;|\; \text{$x \in \P{i}$ with $x_{L} = \tilde{x}$} \},
\]
which contains $|\R{i}|$ elements.  The key property of the map $\pi$ is that, for $(\tilde{x}, \tilde{y}) \in \L{i} \times \UR_{[i]}$ and $(x, y) \in \pi^{-1}(\tilde{x}, \tilde{y})$,
\[
\w{I} xgyx^{-1} \w{I}^{-1} = \w{I}\tilde{x} g \tilde{x}^{-1} \w{I}^{-1} y,
\]
which can be verified by direct computation; from this it follows that $x g y x^{-1} \in {}^{\w{I}^{-1} }\UP_{I}$ if and only if $\tilde{x} g \tilde{x}^{-1} \in \UL_{[i]}$ (so that $\w{I}\tilde{x} g \tilde{x}^{-1} \w{I}^{-1} \in \UL_{I}$) and also $y \in \UR_{I}$.  Thus up to scaling $\defl^{\P{i}}_{\L{i}} \circ  \ind^{\P{i}}_{ {}^{\w{I}^{-1}} \UP_{I} } \circ (\w{I})^{\ast} (\psi)(g)$ is equal to
\begin{align*}
\sum_{\substack{x \in \P{i} \\ y \in \R{i} \\ x g y x^{-1} \in {}^{\w{I}^{-1} }\UP_{I} }} \hspace{-2em} \psi (\w{I} x g y x^{-1} \w{I}^{-1})
&= \hspace{-0.5em} \sum_{\substack{ \tilde{x} \in \L{i} \\ \tilde{y} \in \UR_{I} \\ \tilde{x}g\tilde{x}^{-1} \in \UL_{[i]} }} \hspace{-0.5em} \sum_{(x, y) \in \pi^{-1}(\tilde{x}, \tilde{y})} \hspace{-0.5em}  \psi (\w{I} x g y x^{-1} \w{I}^{-1}) \\
&= |\R{i}| \hspace{-1.5em} \sum_{\substack{ \tilde{x} \in \L{i} \\ \tilde{y} \in \UR_{I} \\ \tilde{x}g\tilde{x}^{-1} \in \UL_{[i]} }}  \hspace{-1.5em}  \psi( \w{I}\tilde{x} g \tilde{x}^{-1} \w{I}^{-1} y ).
\end{align*}
Finally, $|\UL_{[i]}| = |\UL_{I}|$, so that $|\UL_{[i]}| |\UR_{I}| = |\UP_{I}|$, completing the proof.
\end{proof}

\subsection{The diagram automorphism}
\label{sec:antiautomorphism}

For $n \ge 0$, let $\tilde{w} \in \S{[n]}$ be the permutation defined by
\[
\longw(i) = n + 1 - i,
\]
so that $\longw$ sends the usual total order $1 < \cdots < n$ to its reverse, $n < \cdots < 1$.  Typically, $\tilde{w}$ is  denoted by $w_{0}$; my choice of notation is to avoid confusion with the permutations $\w{I}$ from Section~\ref{sec:inductionhom}.  Let
\[
\begin{array}{rcl}
\dagger\colon \Mat[n] & \to & \Mat[n] \\
x & \mapsto & x^{\dagger}
\end{array}
\qquad\text{with}\qquad
x^{\dagger} = {}^{\longw} (x^{T}).
\]
Thus, $x^{\dagger}$ is the matrix with $r, s$-entry $x_{\longw(s), \longw(r)}$.

\begin{ex}
For $n = 4$, $q = 5$, and
\[
x = \begin{bmatrix} 1 & 4 & 0 & 0 \\ 2 & 0 & 1 & 0 \\ 0 & 0 & 1 & 0 \\ 0 & 0 & 0 & 3 \end{bmatrix}
\qquad\text{we have}\qquad
x^{\dagger} = \begin{bmatrix} 3 & 0 & 0 & 0 \\ 0 & 1 & 1 & 0 \\ 0 & 0 & 0 & 4 \\ 0 & 0 & 2 & 1 \end{bmatrix}
\]
\end{ex}

Several properties of this map are immediate from the definition: $\dagger$ is an involution, $\dagger$ preserves rank, and $\dagger$ reverses multiplication, in that 
\[
(xy)^{\dagger} = y^{\dagger} x^{\dagger}
\qquad \text{for $x, y \in \Mat([n])$}.
\]
Thus, $\dagger$ is an antiautomorphism of the group $\GL_{n}$.  Further, $\UT_{n}^{\dagger} = \UT_{n}$, so $\dagger$ restricts to an antiautomorphism of $\UT_{n}$.  

The map $\daggeraut$ in defined in Equation~\eqref{eq:daggerautdef} is given by by pre-composition with $\dagger$ in each graded degree.  The remainder of the section uses this definition to prove Theorem~\ref{thm:daggeraut}

\begin{proof}[Proof of Theorem~\ref{thm:daggeraut}]
First consider the claim that $\ind^{\GL}_{\UT} \circ \daggeraut = \ind^{\GL}_{\UT}$.  From the definition of the induction map, it suffices to show that for each $n \ge 0$ and $g \in \UT_{n}$, the elements $g$ and $g^{\dagger}$ are $\GL_{n}$-conjugate.  The conjugacy class of a unipotent element $G$ of $\GL_{n}$ is determined by the sequence
\[
\operatorname{rank}\big((g - 1_{n})^{k}\big) 
\qquad\text{for $k = 1, 2, \ldots, n$.}
\]
Since $(g^{\dagger} - 1_{n})^{k} = ((g - 1_{n})^{k})^{\dagger}$ has the same rank as $(g - 1_{n})^{k}$, $g$ and $g^{\dagger}$ are $\GL_{n}$-conjugate.

The rest of the proof will establish that $\daggeraut$ is an involutionary Hopf algebra antiautomorphism of $\cf(\UT_{\bullet})$.  By definition, $\daggeraut$ is a graded linear transformation of order two, so what remains is to show that
\[
\mu \circ (\daggeraut \otimes \daggeraut) = \daggeraut \circ \mu \circ \beta
\qquad\text{and}\qquad
\Delta \circ \daggeraut = \beta \circ (\daggeraut \otimes \daggeraut) \circ  \Delta,
\] 
where $\beta: \cf(\UT_{\bullet}) \otimes \cf(\UT_{\bullet}) \shortto \cf(\UT_{\bullet}) \otimes \cf(\UT_{\bullet})$ is the graded linear map which interchanges the first and second tensor factors.

The following facts---which can be verified directly---will be used.  For $n \ge 0$ and $I \subseteq [n]$,
\[
\UL_{I}^{\dagger} = \UL_{\longw(I^{c})}, \qquad
\UR_{I}^{\dagger} = \UR_{\longw(I^{c})},
\qquad\text{and}\qquad
\UP_{I}^{\dagger} = \UP_{\longw(I^{c})}.
\]
For each element $p \in \UP_{I}$, write $p_{1} \in \UT_{|I|}$, $p_{2} \in \UT_{|I^{c}|}$, and $p_{R} \in \UR_{I}$ for the elements satisfying $p = \big( ({}^{\cano_{I}^{-1}} p_{1}) \oplus ({}^{\cano_{I^{c}}^{-1}} p_{2}) \big)p_{R}$, where $\cano$ is as defined in Equation~\eqref{eq:canomapdef} so that $({}^{\cano_{I}^{-1}} p_{1}) \oplus ({}^{\cano_{I^{c}}^{-1}} p_{2}) \in \UL_{I}$.  Then
\begin{equation}
\label{eq:daggerautproof}
p^{\dagger} = \big( ({}^{\cano_{\longw(I^{c})}^{-1}} (p_{2}^{\dagger})) \oplus ({}^{\cano_{\longw(I)}^{-1}} (p_{1}^{\dagger}))\big)  ({}^{({}^{\cano_{I}^{-1}} p_{1}) \oplus ({}^{\cano_{I^{c}}^{-1}} p_{2}) } p_{R})^{\dagger}
\end{equation}
gives a decomposition of $p^{\dagger} \in \UP_{\longw(I^{c})}$ in accordance with the semidirect product structure $\UP_{\longw(I^{c})} = \UL_{\longw(I^{c})} \rtimes \UR_{\longw(I^{c})}$, where ${}^{({}^{\cano_{I}^{-1}} p_{1}) \oplus ({}^{\cano_{I^{c}}^{-1}} p_{2}) } p_{R}$ is the conjugate of $p_{R}$ by $({}^{\cano_{I}^{-1}} p_{1}) \oplus ({}^{\cano_{I^{c}}^{-1}} p_{2})$.

To show the necessary identity for the product map, take $I = [i] \subseteq [n]$ with $n \ge i \ge 0$, so that $\UP_{[i]} = \UT_{n}$.  For $g \in \UT_{n}$, and for $\psi_{1} \in \cf(\UT_{i})$, and $\psi_{2} \in \cf(\UT_{n-i})$, 
\begin{align*}
\infl^{\UT_{n}}_{\UL_{[i]}} \circ \st_{([i], [i]^{c})}^{-1} \circ (\daggeraut \otimes \daggeraut) (\psi_{1} \times \psi_{2})(g)
&= (\psi_{1} \otimes \psi_{2})(g_{1}^{\dagger} \otimes g_{2}^{\dagger}) \\
&= \daggeraut \circ \infl^{\UT_{n}}_{\UL_{[n-i]}} \circ \st_{([n-i], [n-i]^{c})}^{-1} (\psi_{2} \otimes \psi_{1})(g),
\end{align*}
so that summing over $i$ gives $\sum_{i = 1}^{n} \infl^{\UT_{n}}_{\UL_{[i]}} \circ \st_{([i], [i]^{c})}^{-1} \circ (\daggeraut \otimes \daggeraut) = \sum_{i = 1}^{n} \daggeraut \circ \infl^{\UT_{n}}_{\UL_{[i]}} \circ \st_{([i], [i]^{c})}^{-1} \circ \beta$ as desired.

For the coproduct identity, let $I \subseteq [n]$ be arbitrary.  For $\psi \in \cf(\UT_{n})$, $g_{1} \in \UT_{|I|}$, and $g_{2} \in \UT_{|I^{c}|}$, 
\begin{align*}
\st_{(I, I^{c})} \circ  \resf^{\UT_{n}}_{\UL_{I}} \circ \; \daggeraut_{n} (\psi) (g_{1} \times g_{2}) &= \frac{1}{|\UR_{I}|} \sum_{x \in \UR_{I}} \psi\big( (g x)^{\dagger} \big).
\end{align*}
Applying Equation~\eqref{eq:daggerautproof} to $(gx)^{\dagger} \in \UP_{\longw(I^{c})}$ and writing $x' = ({}^{({}^{\cano_{I}^{-1}} g_{1}) \oplus ({}^{\cano_{I^{c}}^{-1}} g_{2}) } x)^{\dagger}$, the above expression becomes
\begin{align*}
\frac{1}{|\UR_{\longw(I^{c})}|} \sum_{x' \in \UR_{\longw(I^{c})}} \psi( g^{\dagger} x' ) &= \st_{(\longw(I^{c}), \longw(I))} \circ  \resf^{\UT_{n}}_{\UL_{\longw(I^{c})}} (\psi) (g_{2}^{\dagger} \times g_{1}^{\dagger}) \\
&= (\daggeraut_{|I^{c}|} \otimes \daggeraut_{|I|} ) \circ \st_{(\longw(I^{c}), \longw(I))} \circ  \resf^{\UT_{n}}_{\UL_{\longw(I^{c})}} \circ (\psi) (g_{2} \times g_{1}).
\end{align*}
Summing over all $I$ gives $\sum_{I \subseteq [n]} \st_{(I, I^{c})} \circ  \resf^{\UT_{n}}_{\UL_{I}} \circ \; \daggeraut = \sum_{I \subseteq [n]} \beta \circ (\daggeraut \otimes \daggeraut) \circ \st_{(I, I^{c})} \circ  \resf^{\UT_{n}}_{\UL_{I}}$, completing the proof.
\end{proof}

\section{The sub-Hopf algebra $\scf(\UT_{\bullet})$}
\label{sec:subHopf}

This section introduces a sub-Hopf algebra $\scf(\UT_{\bullet})$ of $\cf(\UT_{\bullet})$.  Section~\ref{sec:scf} describes the underlying space and some of its representation theoretic origins, Section~\ref{sec:GPHopf} describes and isomorphic Hopf algebra on indifference graphs, and Section~\ref{sec:scfiso} connects these threads, giving a formal presentation of $\scf(\UT_{\bullet})$.  Finally, Section~\ref{sec:noncocommutativity} establishes that $\scf(\UT_{\bullet})$---and thus $\cf(\UT_{\bullet})$---is noncommutative and noncocommutative.

\subsection{Natural unit interval orders and $\scf(\UT_{n})$}
\label{sec:scf}

In~\cite{AlTh20}, Aliniaeifard and Thiem define a particularly well-behaved subspace $\scf(\UT_{n})$ of $\cf(\UT_{n})$, for each $n \ge 0$.  
This subspace originates in the study of supercharacter theory~\cite{DiaconisIsaacs}, but for the sake of brevity this framework will not be introduced.  
Instead, this section will describe two key bases $\scf(\UT_{n})$, with particular attention to the underlying combinatorics which will be used in the following sections.

For $n \ge 0$, a \emph{natural unit interval order} of $[n]$ is a partial order $\pi$ of $[n]$ which extends to the usual order $1 < \cdots < n$ and moreover satisfies
\[
\text{for each $(j, k) \in \pi$ with $j \neq k$}: \qquad \{(i, l) \;|\;\text{$i \le j$ and $k \le l$}\} \subseteq \pi.
\]
Thus, these partial orders satisfy the condition of Lemma~\ref{lem:patterngroupnormalcy} for the order $1 < \ldots < n$ of $[n]$, so that $\pi$ is a natural unit interval order if and only if $\UT(\pi) \trianglelefteq \UT_{n}$.  
Let
\[
\natorder_{n} = \{\text{natural unit interval orders of $[n]$}\}.
\]
The size $|\natorder_{n}|$ is the $n$th Catalan number~\cite[Exercise 6.19.ddd]{StanEnum2}, which can be seen  by drawing a Dyck path above the matrix entries which are uniformly $1$ or $0$ in each $g \in \UT(\pi)$.

\begin{ex}
Let
\[
\pi = \begin{tikzpicture}[scale = 0.65, baseline = 0.65*0.3cm]
\draw (0, 0) node[inner sep = 0.05cm] (a) {$\scriptstyle 1$};
\draw (1, 0) node[inner sep = 0.05cm] (b) {$\scriptstyle 2$};
\draw (1.5, 0.5) node[inner sep = 0.05cm] (c) {$\scriptstyle 3$};
\draw (0.5, 1) node[inner sep = 0.05cm] (d) {$\scriptstyle 4$};
\draw (a) -- (d);
\draw (b) -- (d);
\end{tikzpicture} 
\qquad \text{and} \qquad
\rho = \begin{tikzpicture}[scale = 0.65, baseline = 0.65*0.3cm]
\draw (0, 0) node[inner sep = 0.05cm] (a) {$\scriptstyle 1$};
\draw (1, 0) node[inner sep = 0.05cm] (b) {$\scriptstyle 2$};
\draw (1.5, 0.5) node[inner sep = 0.05cm] (c) {$\scriptstyle 4$};
\draw (0.5, 1) node[inner sep = 0.05cm] (d) {$\scriptstyle 3$};
\draw (a) -- (d);
\draw (b) -- (d);
\end{tikzpicture}. 
\]
Then $\pi \in \natorder_{4}$ is a natural unit interval order, but $\rho \notin \natorder_{4}$ is not, as it contains neither $(1, 4)$ nor $(2, 4)$.  The Dyck path corresponding to $\pi$ is $(EEESSESS)$, as shown below:
\[
\UT \left( \begin{tikzpicture}[scale = 0.35, baseline = 0.35*0.2cm]
\draw (0, 0) node[inner sep = 0.025cm] (a) {$\scriptscriptstyle 1$};
\draw (1, 0) node[inner sep = 0.025cm] (b) {$\scriptscriptstyle 2$};
\draw (1.5, 0.5) node[inner sep = 0.025cm] (c) {$\scriptscriptstyle 3$};
\draw (0.5, 1) node[inner sep = 0.025cm] (d) {$\scriptscriptstyle 4$};
\draw (a) -- (d);
\draw (b) -- (d);
\end{tikzpicture} \right) = 
\begin{tikzpicture}[scale = 0.45, baseline = 0.45*-2.2cm]
\draw[color = gray!50] (-0.1, 0.1) grid (4.1, -4.1);
\draw[very thick] (0, 0) -- (3, 0) -- (3, -2) -- (4, -2) -- (4, -4);
\foreach \x in {1, 2, 3, 4}{\draw[color = gray] (\x - 0.5, -\x + 0.5) node {$1$};}
\foreach \x in {1}{\foreach \y in {2, 3, 4}{\draw[color = gray] (\x - 0.5, -\y + 0.5) node {$0$};}}
\foreach \x in {2}{\foreach \y in {3, 4}{\draw[color = gray] (\x - 0.5, -\y + 0.5) node {$0$};}}
\foreach \x in {3}{\foreach \y in {4}{\draw[color = gray] (\x - 0.5, -\y + 0.5) node {$0$};}}
\foreach \x in {2, 3}{\foreach \y in {1}{\draw[color = gray] (\x - 0.5, -\y + 0.5) node {$0$};}}
\foreach \x in {3}{\foreach \y in {2}{\draw[color = gray] (\x - 0.5, -\y + 0.5) node {$0$};}}
\foreach \x in {4}{\foreach \y in {3}{\draw[color = gray] (\x - 0.5, -\y + 0.5) node {$0$};}}
\foreach \x in {4}{\foreach \y in {1}{\draw[color = gray] (\x - 0.5, -\y + 0.5) node {$\ast$};}}
\foreach \x in {4}{\foreach \y in {2}{\draw[color = gray] (\x - 0.5, -\y + 0.5) node {$\ast$};}}
\end{tikzpicture}.
\]
\end{ex}

For each $\pi \in \natorder_{n}$, the space of cosets $\CC[\UT_{n}/\UT(\pi)]$ has a canonical $\UT_{n}$-module structure which affords the permutation character
\[
\permchar^{\pi} = \ind^{\UT_{n}}_{\UT(\pi)}(\mathbbm{1}) \in \cf(\UT_{n}),
\]
where $\mathbbm{1} \in \cf(\UT(\pi))$ denotes the trivial character.  The subspace of superclass functions defined in~\cite{AlTh20} is
\[
\scf(\UT_{n}) = \CC\spanning\{ \permchar^{\pi} \;|\; \pi \in \natorder_{n} \}.
\]

The permutation characters $\{\permchar^{\pi} \;|\; \pi \in \NNN\OOO_{n}\}$ are a basis for the space $\scf(\UT_{n})$.  Moreover, $\scf(\UT_{n})$ has a second basis of characters (``supercharacters'') which have a triangular relationship with the permutation characters, and several bases consisting of class functions that are not characters.  The following results will use one of the latter bases, which consists of the identifier functions $\{\permind_{\pi} \;|\; \pi \in \natorder_{n}\}$ with
\[
\permind_{\pi}(g) = \begin{cases} 1 & \text{if $g \in \UT(\pi)$} \\ 0 & \text{otherwise.} \end{cases}
\]
As $\UT(\pi)$ is a normal subgroup for each $\pi \in \NNN\OOO_{n}$, the identifier functions abover are closely related to the permutation characters by the equation $\permind_{\pi} = q^{-|\UT_{n}:\UT(\pi)|} \permchar^{\pi}$.

\subsection{Guay-Paquet's Hopf algebra}
\label{sec:GPHopf}

This section will give a presentation of the ``Hopf algebra of Dyck paths'' defined in~\cite[Section 6]{GP16} as it relates to the content of this paper.

Given a partial order $\pi$ of $[n]$, the \emph{incomparability graph} of $\pi$ is the graph
\[
\operatorname{inc}(\pi) = ([n], E_{\pi}) \qquad\text{where}\qquad E_{\pi} = \{\{i, j\} \subseteq [n] \;|\; \text{neither $(i, j)$ nor $(j, i) \in \pi$}\}.
\]
For $\pi \in \NNN\OOO$, the graph $\operatorname{inc}(\pi)$ is known as an \emph{indifference graph}.  
The underlying free graded $\CC[t]$-module for Guay-Paquet's Hopf algebra is
\[
\mathsf{IG}(t) = \bigoplus_{n \ge 0} \mathsf{IG}_{n}(t) \qquad\text{with}\qquad \mathsf{IG}_{n}(t) = \CC[t]\spanning\{ \operatorname{inc}(\pi) \;|\; \pi \in \natorder_{n} \},
\]
where $ \operatorname{inc}(\pi)$ is treated as a formal basis element.

In~\cite{GP16}, the structure maps of $\mathsf{IG}(t)$ are defined in terms of operations on incomparability graphs.  To simplify the statement of future results, I will instead give an equivalent description in terms of natural unit interval orders.  Recall the ordinal sum and restriction operations defined in Section~\ref{sec:monoidexamples}, and the canonicalization map defined in Section~\ref{sec:Fockfunctors}.  

For $n, m \ge 0$, $\pi \in \natorder_{n}$ and $\rho \in \natorder_{m}$, the \emph{shifted ordinal sum} of $\pi$ and $\rho$ is
\[
\pi \ordinalsumsh{i}
\rho = \pi \ordinalsum \cano_{[n+m] \setminus [n]}^{-1}(\rho),
\]
which is a natural unit interval order of $[n+m]$.

\begin{ex}
Let $\pi = \tikz[scale = 0.5, baseline = 0.15cm]{
\draw[fill] (0, 0) circle (2pt) node[inner sep = 1pt] (1) {};
\draw[fill] (0.5, 0.5) circle (2pt) node[inner sep = 1pt] (2) {};
\draw[fill] (-0.5, 0.75) circle (2pt) node[inner sep = 1pt] (3) {};
\draw[fill] (0.5, 1.125) circle (2pt) node[inner sep = 1pt] (4) {};
\draw (1) node[below] {$\scriptstyle 1$};
\draw (2) node[right] {$\scriptstyle 2$};
\draw (3) node[left] {$\scriptstyle 3$};
\draw (4) node[right] {$\scriptstyle 4$};
\draw (1) -- (2);
\draw (1) -- (3);
\draw (2) -- (4);}$ and $\rho = \tikz[scale = 0.5, baseline = -0.1cm]{
\draw[fill] (0, 0) circle (2pt) node[below] {$\scriptstyle 1$};
\draw[fill] (1, 0) circle (2pt) node[below] {$\scriptstyle 2$};}$.  Then $\pi \ordinalsumsh{4} \rho = \tikz[scale = 0.5, baseline = 0.5*0.7cm]{
\draw[fill] (0, 0) circle (2pt) node[inner sep = 1pt] (1) {};
\draw[fill] (0.5, 0.5) circle (2pt) node[inner sep = 1pt] (2) {};
\draw[fill] (-0.5, 0.75) circle (2pt) node[inner sep = 1pt] (3) {};
\draw[fill] (0.5, 1.125) circle (2pt) node[inner sep = 1pt] (4) {};
\draw[fill] (-0.5, 1.75) circle (2pt) node[inner sep = 1pt] (5) {};
\draw[fill] (0.5, 1.75) circle (2pt) node[inner sep = 1pt] (6) {};
\draw (1) node[below] {$\scriptstyle 1$};
\draw (2) node[right] {$\scriptstyle 2$};
\draw (3) node[left] {$\scriptstyle 3$};
\draw (4) node[right] {$\scriptstyle 4$};
\draw (5) node[left] {$\scriptstyle 5$};
\draw (6) node[right] {$\scriptstyle 6$};
\draw (1) -- (2);
\draw (1) -- (3);
\draw (2) -- (4);
\draw (3) -- (5);
\draw (3) -- (6);
\draw (4) -- (5);
\draw (4) -- (6);}$.
\end{ex}

For $n \ge 0$, $\pi \in \natorder_{n}$, and $I \subseteq [n]$, and define the \emph{shifted restriction} of $\pi$ to be
\[
\pi|^{sh}_{I} = \cano_{I}(\pi|_{I}),
\]
which is also a natural unit interval order of $[|I|]$ that is isomorphic as a poset to $\pi|_{I}$.  Furthermore, define the \emph{$I$-ascents} of $\pi \in \natorder_{n}$ to be
\[
\asc_{I}(\pi) = |\{ (i, j) \in I \times I^{c} \;|\; \text{$1 \le i < j \le n$ and $(i, j) \notin \pi$} \}|.
\]

\begin{ex}
Let $\pi = \tikz[scale = 0.5, baseline = 0.5*0.7cm]{
\draw[fill] (0, 0) circle (2pt) node[inner sep = 1pt] (1) {};
\draw[fill] (0.5, 0.5) circle (2pt) node[inner sep = 1pt] (2) {};
\draw[fill] (-0.5, 0.75) circle (2pt) node[inner sep = 1pt] (3) {};
\draw[fill] (0.5, 1.125) circle (2pt) node[inner sep = 1pt] (4) {};
\draw[fill] (-0.5, 1.75) circle (2pt) node[inner sep = 1pt] (5) {};
\draw[fill] (0.5, 1.75) circle (2pt) node[inner sep = 1pt] (6) {};
\draw (1) node[below] {$\scriptstyle 1$};
\draw (2) node[right] {$\scriptstyle 2$};
\draw (3) node[left] {$\scriptstyle 3$};
\draw (4) node[right] {$\scriptstyle 4$};
\draw (5) node[left] {$\scriptstyle 5$};
\draw (6) node[right] {$\scriptstyle 6$};
\draw (1) -- (2);
\draw (1) -- (3);
\draw (2) -- (4);
\draw (3) -- (5);
\draw (3) -- (6);
\draw (4) -- (5);
\draw (4) -- (6);}$. Then $\pi|_{\{2, 3, 5, 6\}}^{sh} = \tikz[scale = 0.5, baseline = 0.5*1.05cm]{
\draw[fill] (0.5, 0.75) circle (2pt) node[inner sep = 1pt] (2) {};
\draw[fill] (-0.5, 0.75) circle (2pt) node[inner sep = 1pt] (3) {};
\draw[fill] (-0.5, 1.75) circle (2pt) node[inner sep = 1pt] (5) {};
\draw[fill] (0.5, 1.75) circle (2pt) node[inner sep = 1pt] (6) {};
\draw (2) node[right] {$\scriptstyle 1$};
\draw (3) node[left] {$\scriptstyle 2$};
\draw (5) node[left] {$\scriptstyle 3$};
\draw (6) node[right] {$\scriptstyle 4$};
\draw (3) -- (5);
\draw (3) -- (6);
\draw (2) -- (5);
\draw (2) -- (6);}$ and $\asc_{\{2, 3, 5, 6\}}(\pi) = |\{(3, 4)\}| = 1$.
\end{ex}

The product map for the Hopf algebra $\mathsf{IG}(t)$ is the graded $\CC[t]$-linear map defined by
\[
\mu(\operatorname{inc}(\pi) \otimes \operatorname{inc}(\rho)) = \operatorname{inc}(\pi \ordinalsumsh{n} \rho)
\]
for $\pi \in \natorder_{n}$ and $\rho \in \natorder_{m}$ with $n, m \ge 0$, and the coproduct is the graded $\CC[t]$-linear map defined by
\[
\Delta(\operatorname{inc}(\pi)) = \sum_{I \subseteq [n]} t^{\asc_{I}(\pi)} \operatorname{inc}(\pi|^{sh}_{I}) \otimes \operatorname{inc}(\pi|^{sh}_{I^{c}})
\]
for $\pi \in \natorder_{n}$.

\begin{thm}[{\cite[Corollary~77]{GP16}}]
With the product and coproduct defined above, $\mathsf{IG}(t)$ is a graded connected $\CC[t]$-Hopf algebra.
\end{thm}

As a consequence, for any complex number $\alpha \in \CC$ there is a graded connected $\CC$-Hopf algebra $\mathsf{IG}(\alpha)$ given by evaluating $t = \alpha$ in the formulas above.

\subsection{Realizing $\mathsf{IG}(q^{-1})$ as a sub-Hopf algebra of $\cf(\UT_{\bullet})$}
\label{sec:scfiso}

The main result of this section follows; recall $\permind_{\pi} \in \scf(\UT_{n}) \subseteq \cf(\UT_{n})$ as defined in Section~\ref{sec:scf} and the Hopf algebra $\mathsf{IG}(t)$ defined in Section~\ref{sec:GPHopf}.

\begin{thm}
\label{thm:Dhom}
Recall that $q$ is the order of $\FF_{q}$.  The graded linear map
\[
\begin{array}{rcl}
\mathsf{IG}(q^{-1}) & \to & \cf(\UT_{\bullet}) \\
\operatorname{inc}(\pi) & \mapsto & \permind_{\pi}
\end{array}
\]
is an injective Hopf algebra homomorphism.
\end{thm}

The proof of Theorem~\ref{thm:Dhom} follows the next corollary.

\begin{cor}
\label{cor:scfsubhopf}
The space
\[
\scf(\UT_{\bullet}) = \bigoplus_{n \ge 0} \scf(\UT_{n})
\]
is a sub-Hopf algebra of $\cf(\UT_{\bullet})$, and moreover isomorphic to $\mathsf{IG}(q^{-1})$.  In particular, 
\[
\mu(\permind_{\pi} \otimes \permind_{\rho}) = \permind_{\pi \ordinalsumsh{n} \rho}
\qquad\text{and}\qquad
\Delta(\permind_{\pi}) = \sum_{I \subseteq [n]} q^{-\asc_{I}(\pi)} \permind_{\pi|^{sh}_{I}} \otimes \permind_{\pi|^{sh}_{I^{c}}}
\]
for $\pi \in \natorder_{n}$ and $\rho \in \natorder_{m}$ with $n, m \ge 0$.
\end{cor}

\begin{proof}[Proof of Theorem~\ref{thm:Dhom}]
By definition, the given map is an injective graded linear map.  To show that it is an algebra homomorphism, it is sufficient to show that for $\pi \in \natorder_{n}$ and $\rho \in \natorder_{m}$ with $n, m \ge 0$, 
\[
\infl^{\UT_{n+m}}_{\UL_{[n]}} \circ \st_{([n], [n+m] \setminus [n])}^{-1}(\permind_{\pi} \otimes \permind_{\rho}) = \permind_{\pi \ordinalsumsh{n} \rho}.
\]
This follows from direct computation: for $g \in \UT_{n + m}$, 
\[
\infl^{\UT_{n+m}}_{\UL_{[n]}} \circ \st_{([n], [n+m] \setminus [n])}^{-1}(\permind_{\pi} \otimes \permind_{\rho})(g) = \begin{cases} 1 & \text{if $g \in \big(\UT(\pi) \oplus {}^{\cano_{[n+m] \setminus [n]}^{-1}} \UT(\rho)\big) \UR_{[n]}$} \\ 0 & \text{otherwise,}  \end{cases}
\]
and $(\UT(\pi) \oplus {}^{\cano_{[n+m] \setminus [n]}^{-1}} \UT(\rho))\UR_{[n]} = \UT(\pi \ordinalsumsh{n} \rho)$.

Now consider the claim that the given map is a coalgebra homomorphism.  To establish this, it is sufficient to show that for $\pi \in \natorder_{n}$ and $I \subseteq [n]$ with $n \ge 0$, 
\[
\st_{(I, I^{c})} \circ \resf^{\UT_{n}}_{\UL_{I}}(\permind_{\pi}) = q^{-\asc_{I}(\pi)} \permind_{\pi|^{sh}_{I}} \otimes \permind_{\pi|^{sh}_{I^{c}}},
\]
where $I^{c} = [n] \setminus I$ as usual.  For $g_{1} \in \UT_{|I|}$ and $g_{2} \in \UT_{|I^{c}|}$, 
\[
\st_{(I, I^{c})} \circ \resf^{\UT_{n}}_{\UL_{I}}(\permind_{\pi})(g_{1} \times g_{2}) = \frac{|\{  \UT(\pi) \cap ({}^{\cano_{I}^{-1}} g_{1} \oplus {}^{\cano_{I^{c}}^{-1}} g_{2})\UR_{I}  \}|}{|\UR_{I}|}.
\]
For each $x \in \UR_{[n]}$, $({}^{\cano_{I}^{-1}} g_{1} \oplus {}^{\cano_{I^{c}}^{-1}} g_{2})x \in \UT(\pi)$ if and only if both $x$ and ${}^{\cano_{I}^{-1}} g_{1} \oplus {}^{\cano_{I^{c}}^{-1}} g_{2}$ are contained in $\UT(\pi)$.  Moreover, the condition ${}^{\cano_{I}^{-1}} g_{1}  \oplus {}^{\cano_{I^{c}}^{-1}} g_{2} \in \UT(\pi)$ is equivalent to the simultaneous conditions 
\[
(g_{1}-1_{|I|})_{r, s} \neq 0 \qquad\text{only if}\qquad(\cano_{I}^{-1}(r), \cano_{I}^{-1}(s)) \in \pi
\]
and
\[
(g_{2}-1_{|I^{c}|})_{r, s} \neq 0 \qquad\text{only if}\qquad(\cano_{I^{c}}^{-1}(r), \cano_{I^{c}}^{-1}(s)) \in \pi,
\]
or in other terms, $g_{1} \times g_{2} \in \UT(\pi|^{sh}_{I}) \times \UT(\pi|^{sh}_{I^{c}})$, so 
\[
\st_{(I, I^{c})} \circ \resf^{\UT_{n}}_{\UL_{I}}(\permind_{\pi}) = \frac{|\UT(\pi) \cap \UR_{I}|}{|\UR_{I}|} \permind_{\pi|^{sh}_{I}} \otimes \permind_{\pi|^{sh}_{I^{c}}}.
\]
Finally, $|\UT(\pi) \cap \UR_{I}| = |\UR_{I}|\big/ q^{\asc_{I}(\pi)}$, so the claimed coproduct formula holds.
\end{proof}

\subsection{Nonommutativity and noncocommutativity}
\label{sec:noncocommutativity}

The final result makes use of the structure constants in Corollary~\ref{cor:scfsubhopf} to show that the Hopf algebras $\scf(\UT_{\bullet})$ and $\cf(\UT_{\bullet})$ are distinct from those in~\cite{AgBerTh} and~\cite{AlTh20}, respectively.  To the best of my knowledge the noncocommutativity part of this result is novel even when interpreted as a statement about $\mathsf{IG}(t)$, though it is likely known to the author of~\cite{GP16}.

\begin{thm}
\label{thm:noncocomm}
The Hopf algebras $\cf(\UT_{\bullet})$ and $\scf(\UT_{\bullet})$ are noncommutative and noncocommutative.
\end{thm}
\begin{proof}
It is enough to show that $\scf(\UT_{\bullet})$ is noncommutative and noncocommutative. This can be established by example, with
\[
\mu\left(\permind_{\begin{tikzpicture}[scale = 0.35, baseline = 0.35*-0.2cm]
\draw (0, 0) node[inner sep = 0.05cm] (1) {$\scriptscriptstyle 1$};
\end{tikzpicture}}
\otimes
\permind_{\begin{tikzpicture}[scale = 0.35, baseline = 0.35*-0.2cm]
\draw (0, 0) node[inner sep = 0.05cm] (1) {$\scriptscriptstyle 1$};
\draw (1, 0) node[inner sep = 0.05cm] (2) {$\scriptscriptstyle 2$};
\end{tikzpicture}} \right)
= 
\permind_{\begin{tikzpicture}[scale = 0.35, baseline = 0.35*0.3cm]
\draw (0.5, 0) node[inner sep = 0.05cm] (1) {$\scriptscriptstyle 1$};
\draw (0, 1) node[inner sep = 0.05cm] (2) {$\scriptscriptstyle 2$};
\draw (1, 1) node[inner sep = 0.05cm] (3) {$\scriptscriptstyle 3$};
\draw (1) -- (2);
\draw (1) -- (3);
\end{tikzpicture}}
\neq
\permind_{\begin{tikzpicture}[scale = 0.35, baseline = 0.35*0.3cm]
\draw (0, 0) node[inner sep = 0.05cm] (1) {$\scriptscriptstyle 1$};
\draw (1, 0) node[inner sep = 0.05cm] (2) {$\scriptscriptstyle 2$};
\draw (0.5, 1) node[inner sep = 0.05cm] (3) {$\scriptscriptstyle 3$};
\draw (1) -- (3);
\draw (2) -- (3);
\end{tikzpicture}}
= 
\mu \left( \permind_{\begin{tikzpicture}[scale = 0.35, baseline = 0.35*-0.2cm]
\draw (0, 0) node[inner sep = 0.05cm] (1) {$\scriptscriptstyle 1$};
\draw (1, 0) node[inner sep = 0.05cm] (2) {$\scriptscriptstyle 2$};
\end{tikzpicture}} 
\otimes
\permind_{\begin{tikzpicture}[scale = 0.35, baseline = 0.35*-0.2cm]
\draw (0, 0) node[inner sep = 0.05cm] (1) {$\scriptscriptstyle 1$};
\end{tikzpicture}} \right)
\]
establishing noncommutativity, and 
\begin{multline*}
\Delta\left(  \permind_{\begin{tikzpicture}[scale = 0.35, baseline = 0.35*0.3cm]
\draw (0, 0) node[inner sep = 0.05cm] (1) {$\scriptscriptstyle 1$};
\draw (1, 0) node[inner sep = 0.05cm] (2) {$\scriptscriptstyle 2$};
\draw (1.75, 0.5) node[inner sep = 0.05cm] (3) {$\scriptscriptstyle 3$};
\draw (0.5, 1) node[inner sep = 0.05cm] (4) {$\scriptscriptstyle 4$};
\draw (1) -- (4);
\draw (2) -- (4);
\end{tikzpicture}} \right) 
= \underbrace{ \permind_{\begin{tikzpicture}[scale = 0.35, baseline = 0.35*0.3cm]
\draw (0, 0) node[inner sep = 0.05cm] (1) {$\scriptscriptstyle 1$};
\draw (1, 0) node[inner sep = 0.05cm] (2) {$\scriptscriptstyle 2$};
\draw (1.75, 0.5) node[inner sep = 0.05cm] (3) {$\scriptscriptstyle 3$};
\draw (0.5, 1) node[inner sep = 0.05cm] (4) {$\scriptscriptstyle 4$};
\draw (1) -- (4);
\draw (2) -- (4);
\end{tikzpicture}}  \otimes  \permind_{\emptyset}}_{([4], \emptyset)} 
+  \underbrace{q^{-1}  \permind_{\begin{tikzpicture}[scale = 0.35, baseline = 0.35*-0.2cm]
\draw (0.25, 0) node[inner sep = 0.05cm] (1) {$\scriptscriptstyle 1$};
\draw (1, 0) node[inner sep = 0.05cm] (2) {$\scriptscriptstyle 2$};
\draw (1.75, 0) node[inner sep = 0.05cm] (3) {$\scriptscriptstyle 3$};
\end{tikzpicture}}  \otimes  \permind_{\begin{tikzpicture}[scale = 0.35, baseline = 0.35*-0.2cm]
\draw (0, 0) node[inner sep = 0.05cm] (1) {$\scriptscriptstyle 1$};
\end{tikzpicture}}}_{(\{1, 2, 3\}, \{4\})} 
+ \underbrace{q^{-2}  \permind_{\begin{tikzpicture}[scale = 0.35, baseline = 0.35*0.3cm]
\draw (0, 0) node[inner sep = 0.05cm] (1) {$\scriptscriptstyle 1$};
\draw (1, 0) node[inner sep = 0.05cm] (2) {$\scriptscriptstyle 2$};
\draw (0.5, 1) node[inner sep = 0.05cm] (4) {$\scriptscriptstyle 3$};
\draw (1) -- (4);
\draw (2) -- (4);
\end{tikzpicture}}  \otimes  \permind_{\begin{tikzpicture}[scale = 0.35, baseline = 0.35*-0.2cm]
\draw (0, 0) node[inner sep = 0.05cm] (1) {$\scriptscriptstyle 1$};
\end{tikzpicture}}}_{(\{1, 2, 4\}, \{3\})} 
+ \underbrace{q^{-1}   \permind_{\begin{tikzpicture}[scale = 0.35, baseline = 0.35*0.3cm]
\draw (0, 0) node[inner sep = 0.05cm] (1) {$\scriptscriptstyle 1$};
\draw (0.75, 0.5) node[inner sep = 0.05cm] (3) {$\scriptscriptstyle 2$};
\draw (0, 1) node[inner sep = 0.05cm] (4) {$\scriptscriptstyle 3$};
\draw (1) -- (4);
\end{tikzpicture}}  \otimes  \permind_{\begin{tikzpicture}[scale = 0.35, baseline = 0.35*-0.2cm]
\draw (0, 0) node[inner sep = 0.05cm] (1) {$\scriptscriptstyle 1$};
\end{tikzpicture}}}_{(\{1, 3, 4\}, \{2\})} 
+ \underbrace{  \permind_{\begin{tikzpicture}[scale = 0.35, baseline = 0.35*0.3cm]
\draw (1, 0) node[inner sep = 0.05cm] (2) {$\scriptscriptstyle 1$};
\draw (1.75, 0.5) node[inner sep = 0.05cm] (3) {$\scriptscriptstyle 2$};
\draw (1, 1) node[inner sep = 0.05cm] (4) {$\scriptscriptstyle 3$};
\draw (2) -- (4);
\end{tikzpicture}}  \otimes  \permind_{\begin{tikzpicture}[scale = 0.35, baseline = 0.35*-0.2cm]
\draw (0, 0) node[inner sep = 0.05cm] (1) {$\scriptscriptstyle 1$};
\end{tikzpicture}}}_{(\{2, 3, 4\}, \{1\})} 
\\[0.5em]
%
%
+ \underbrace{q^{-2}  \permind_{\begin{tikzpicture}[scale = 0.35, baseline = 0.35*-0.2cm]
\draw (0, 0) node[inner sep = 0.05cm] (1) {$\scriptscriptstyle 1$};
\draw (0.75, 0) node[inner sep = 0.05cm] (2) {$\scriptscriptstyle 2$};
\end{tikzpicture}} \otimes  \permind_{\begin{tikzpicture}[scale = 0.35, baseline = 0.35*-0.2cm]
\draw (0, 0) node[inner sep = 0.05cm] (3) {$\scriptscriptstyle 1$};
\draw (0.75, 0) node[inner sep = 0.05cm] (4) {$\scriptscriptstyle 2$};
\end{tikzpicture}}}_{(\{1, 2\}, \{3, 4\})}
+ \underbrace{q^{-2}   \permind_{\begin{tikzpicture}[scale = 0.35, baseline = 0.35*-0.2cm]
\draw (0, 0) node[inner sep = 0.05cm] (1) {$\scriptscriptstyle 1$};
\draw (0.75, 0) node[inner sep = 0.05cm] (3) {$\scriptscriptstyle 2$};
\end{tikzpicture}} \otimes  \permind_{\begin{tikzpicture}[scale = 0.35, baseline = 0.35*0.3cm]
\draw (0, 0) node[inner sep = 0.05cm] (2) {$\scriptscriptstyle 1$};
\draw (0, 1) node[inner sep = 0.05cm] (4) {$\scriptscriptstyle 2$};
\draw (2) -- (4);
\end{tikzpicture}}}_{(\{1, 3\}, \{2, 4\})}
+ \underbrace{q^{-2}  \permind_{\begin{tikzpicture}[scale = 0.35, baseline = 0.35*0.3cm]
\draw (0, 0) node[inner sep = 0.05cm] (1) {$\scriptscriptstyle 1$};
\draw (0, 1) node[inner sep = 0.05cm] (4) {$\scriptscriptstyle 2$};
\draw (1) -- (4);
\end{tikzpicture}} \otimes  \permind_{\begin{tikzpicture}[scale = 0.35, baseline = 0.35*-0.2cm]
\draw (0, 0) node[inner sep = 0.05cm] (2) {$\scriptscriptstyle 1$};
\draw (0.75, 0) node[inner sep = 0.05cm] (3) {$\scriptscriptstyle 2$};
\end{tikzpicture}}}_{(\{1, 4\}, \{2, 3\})}
+ \underbrace{q^{-1}   \permind_{\begin{tikzpicture}[scale = 0.35, baseline = 0.35*-0.2cm]
\draw (0, 0) node[inner sep = 0.05cm] (2) {$\scriptscriptstyle 1$};
\draw (0.75, 0) node[inner sep = 0.05cm] (3) {$\scriptscriptstyle 2$};
\end{tikzpicture}} \otimes  \permind_{\begin{tikzpicture}[scale = 0.35, baseline = 0.35*0.3cm]
\draw (0, 0) node[inner sep = 0.05cm] (1) {$\scriptscriptstyle 1$};
\draw (0, 1) node[inner sep = 0.05cm] (4) {$\scriptscriptstyle 2$};
\draw (1) -- (4);
\end{tikzpicture}}}_{(\{2, 3\}, \{1, 4\})}
+ \underbrace{q^{-1}   \permind_{\begin{tikzpicture}[scale = 0.35, baseline = 0.35*0.3cm]
\draw (1, 0) node[inner sep = 0.05cm] (2) {$\scriptscriptstyle 1$};
\draw (1, 1) node[inner sep = 0.05cm] (4) {$\scriptscriptstyle 2$};
\draw (2) -- (4);
\end{tikzpicture}} \otimes  \permind_{\begin{tikzpicture}[scale = 0.35, baseline = 0.35*-0.2cm]
\draw (0, 0) node[inner sep = 0.05cm] (1) {$\scriptscriptstyle 1$};
\draw (0.75, 0) node[inner sep = 0.05cm] (3) {$\scriptscriptstyle 2$};
\end{tikzpicture}}}_{(\{2, 4\}, \{1, 3\})} 
\\[0.5em]
%
%
+ \underbrace{ \permind_{\begin{tikzpicture}[scale = 0.35, baseline = 0.35*-0.2cm]
\draw (0, 0) node[inner sep = 0.05cm] (3) {$\scriptscriptstyle 1$};
\draw (0.75, 0) node[inner sep = 0.05cm] (4) {$\scriptscriptstyle 2$};
\end{tikzpicture}} \otimes  \permind_{\begin{tikzpicture}[scale = 0.35, baseline = 0.35*-0.2cm]
\draw (0, 0) node[inner sep = 0.05cm] (1) {$\scriptscriptstyle 1$};
\draw (0.75, 0) node[inner sep = 0.05cm] (2) {$\scriptscriptstyle 2$};
\end{tikzpicture}}}_{(\{3, 4\}, \{1, 2\})} 
+ \underbrace{q^{-2}    \permind_{\begin{tikzpicture}[scale = 0.35, baseline = 0.35*-0.2cm]
\draw (0, 0) node[inner sep = 0.05cm] (1) {$\scriptscriptstyle 1$};
\end{tikzpicture}} \otimes  \permind_{\begin{tikzpicture}[scale = 0.35, baseline = 0.35*0.3cm]
\draw (0, 0) node[inner sep = 0.05cm] (2) {$\scriptscriptstyle 1$};
\draw (0.75, 0.5) node[inner sep = 0.05cm] (3) {$\scriptscriptstyle 2$};
\draw (0, 1) node[inner sep = 0.05cm] (4) {$\scriptscriptstyle 3$};
\draw (2) -- (4);
\end{tikzpicture}}}_{(\{1\}, \{2, 3, 4\})}
+ \underbrace{q^{-1}  \permind_{\begin{tikzpicture}[scale = 0.35, baseline = 0.35*-0.2cm]
\draw (1, 0) node[inner sep = 0.05cm] (2) {$\scriptscriptstyle 1$};
\end{tikzpicture}} \otimes  \permind_{\begin{tikzpicture}[scale = 0.35, baseline = 0.35*0.3cm]
\draw (0, 0) node[inner sep = 0.05cm] (1) {$\scriptscriptstyle 1$};
\draw (0.75, 0.5) node[inner sep = 0.05cm] (3) {$\scriptscriptstyle 2$};
\draw (0, 1) node[inner sep = 0.05cm] (4) {$\scriptscriptstyle 3$};
\draw (1) -- (4);
\end{tikzpicture}}}_{(\{2\}, \{1, 3, 4\})}
+ \underbrace{q^{-1}  \permind_{\begin{tikzpicture}[scale = 0.35, baseline = 0.35*0.3cm]
\draw (1.75, 0.5) node[inner sep = 0.05cm] (3) {$\scriptscriptstyle 1$};
\end{tikzpicture}} \otimes  \permind_{\begin{tikzpicture}[scale = 0.35, baseline = 0.35*0.3cm]
\draw (0, 0) node[inner sep = 0.05cm] (1) {$\scriptscriptstyle 1$};
\draw (1, 0) node[inner sep = 0.05cm] (2) {$\scriptscriptstyle 2$};
\draw (0.5, 1) node[inner sep = 0.05cm] (4) {$\scriptscriptstyle 3$};
\draw (1) -- (4);
\draw (2) -- (4);
\end{tikzpicture}}}_{(\{3\}, \{1, 2, 4\})}
+ \underbrace{ \permind_{\begin{tikzpicture}[scale = 0.35, baseline = 0.35*-0.2cm]
\draw (0, 0) node[inner sep = 0.05cm] (4) {$\scriptscriptstyle 1$};
\end{tikzpicture}} \otimes  \permind_{\begin{tikzpicture}[scale = 0.35, baseline = 0.35*-0.2cm]
\draw (0.25, 0) node[inner sep = 0.05cm] (1) {$\scriptscriptstyle 1$};
\draw (1, 0) node[inner sep = 0.05cm] (2) {$\scriptscriptstyle 2$};
\draw (1.75, 0) node[inner sep = 0.05cm] (3) {$\scriptscriptstyle 3$};
\end{tikzpicture}}}_{(\{4\}, \{1, 2, 3\})} 
\\[0.5em]
%
%
+ \underbrace{ \permind_{\emptyset} \otimes  \permind_{\begin{tikzpicture}[scale = 0.35, baseline = 0.35*0.3cm]
\draw (0, 0) node[inner sep = 0.05cm] (1) {$\scriptscriptstyle 1$};
\draw (1, 0) node[inner sep = 0.05cm] (2) {$\scriptscriptstyle 2$};
\draw (1.75, 0.5) node[inner sep = 0.05cm] (3) {$\scriptscriptstyle 3$};
\draw (0.5, 1) node[inner sep = 0.05cm] (4) {$\scriptscriptstyle 4$};
\draw (1) -- (4);
\draw (2) -- (4);
\end{tikzpicture}}}_{(\emptyset, [4])},
\end{multline*}
demonstrating noncocommutativity: the $(3, 1)$-graded component
\[
\Delta_{(3, 1)}\left(  \permind_{\begin{tikzpicture}[scale = 0.35, baseline = 0.35*0.3cm]
\draw (0, 0) node[inner sep = 0.05cm] (1) {$\scriptscriptstyle 1$};
\draw (1, 0) node[inner sep = 0.05cm] (2) {$\scriptscriptstyle 2$};
\draw (1.75, 0.5) node[inner sep = 0.05cm] (3) {$\scriptscriptstyle 3$};
\draw (0.5, 1) node[inner sep = 0.05cm] (4) {$\scriptscriptstyle 4$};
\draw (1) -- (4);
\draw (2) -- (4);
\end{tikzpicture}} \right)
=
q^{-1}  \permind_{\begin{tikzpicture}[scale = 0.35, baseline = 0.35*-0.2cm]
\draw (0.25, 0) node[inner sep = 0.05cm] (1) {$\scriptscriptstyle 1$};
\draw (1, 0) node[inner sep = 0.05cm] (2) {$\scriptscriptstyle 2$};
\draw (1.75, 0) node[inner sep = 0.05cm] (3) {$\scriptscriptstyle 3$};
\end{tikzpicture}}  \otimes  \permind_{\begin{tikzpicture}[scale = 0.35, baseline = 0.35*-0.2cm]
\draw (0, 0) node[inner sep = 0.05cm] (1) {$\scriptscriptstyle 1$};
\end{tikzpicture}}
+ q^{-2}  \permind_{\begin{tikzpicture}[scale = 0.35, baseline = 0.35*0.3cm]
\draw (0, 0) node[inner sep = 0.05cm] (1) {$\scriptscriptstyle 1$};
\draw (1, 0) node[inner sep = 0.05cm] (2) {$\scriptscriptstyle 2$};
\draw (0.5, 1) node[inner sep = 0.05cm] (4) {$\scriptscriptstyle 3$};
\draw (1) -- (4);
\draw (2) -- (4);
\end{tikzpicture}}  \otimes  \permind_{\begin{tikzpicture}[scale = 0.35, baseline = 0.35*-0.2cm]
\draw (0, 0) node[inner sep = 0.05cm] (1) {$\scriptscriptstyle 1$};
\end{tikzpicture}}
+ (q^{-1} + 1) \permind_{\begin{tikzpicture}[scale = 0.35, baseline = 0.35*0.3cm]
\draw (0, 0) node[inner sep = 0.05cm] (1) {$\scriptscriptstyle 1$};
\draw (0.75, 0.5) node[inner sep = 0.05cm] (3) {$\scriptscriptstyle 2$};
\draw (0, 1) node[inner sep = 0.05cm] (4) {$\scriptscriptstyle 3$};
\draw (1) -- (4);
\end{tikzpicture}}  \otimes  \permind_{\begin{tikzpicture}[scale = 0.35, baseline = 0.35*-0.2cm]
\draw (0, 0) node[inner sep = 0.05cm] (1) {$\scriptscriptstyle 1$};
\end{tikzpicture}}
\]
and the $(1, 3)$-graded component
\[
\Delta_{(1, 3)}\left(  \permind_{\begin{tikzpicture}[scale = 0.35, baseline = 0.35*0.3cm]
\draw (0, 0) node[inner sep = 0.05cm] (1) {$\scriptscriptstyle 1$};
\draw (1, 0) node[inner sep = 0.05cm] (2) {$\scriptscriptstyle 2$};
\draw (1.75, 0.5) node[inner sep = 0.05cm] (3) {$\scriptscriptstyle 3$};
\draw (0.5, 1) node[inner sep = 0.05cm] (4) {$\scriptscriptstyle 4$};
\draw (1) -- (4);
\draw (2) -- (4);
\end{tikzpicture}} \right)
=
(q^{-2}+q^{-1}) \permind_{\begin{tikzpicture}[scale = 0.35, baseline = 0.35*-0.2cm]
\draw (0, 0) node[inner sep = 0.05cm] (1) {$\scriptscriptstyle 1$};
\end{tikzpicture}} \otimes  \permind_{\begin{tikzpicture}[scale = 0.35, baseline = 0.35*0.3cm]
\draw (0, 0) node[inner sep = 0.05cm] (2) {$\scriptscriptstyle 1$};
\draw (0.75, 0.5) node[inner sep = 0.05cm] (3) {$\scriptscriptstyle 2$};
\draw (0, 1) node[inner sep = 0.05cm] (4) {$\scriptscriptstyle 3$};
\draw (2) -- (4);
\end{tikzpicture}}
+ q^{-1}  \permind_{\begin{tikzpicture}[scale = 0.35, baseline = 0.35*0.3cm]
\draw (1.75, 0.5) node[inner sep = 0.05cm] (3) {$\scriptscriptstyle 1$};
\end{tikzpicture}} \otimes  \permind_{\begin{tikzpicture}[scale = 0.35, baseline = 0.35*0.3cm]
\draw (0, 0) node[inner sep = 0.05cm] (1) {$\scriptscriptstyle 1$};
\draw (1, 0) node[inner sep = 0.05cm] (2) {$\scriptscriptstyle 2$};
\draw (0.5, 1) node[inner sep = 0.05cm] (4) {$\scriptscriptstyle 3$};
\draw (1) -- (4);
\draw (2) -- (4);
\end{tikzpicture}}
+  \permind_{\begin{tikzpicture}[scale = 0.35, baseline = 0.35*-0.2cm]
\draw (0, 0) node[inner sep = 0.05cm] (4) {$\scriptscriptstyle 1$};
\end{tikzpicture}} \otimes  \permind_{\begin{tikzpicture}[scale = 0.35, baseline = 0.35*-0.2cm]
\draw (0.25, 0) node[inner sep = 0.05cm] (1) {$\scriptscriptstyle 1$};
\draw (1, 0) node[inner sep = 0.05cm] (2) {$\scriptscriptstyle 2$};
\draw (1.75, 0) node[inner sep = 0.05cm] (3) {$\scriptscriptstyle 3$};
\end{tikzpicture}}
\]
are not equivalent under the interchange of tensor factors.
\end{proof}

\begin{rem}
There are notions of commutativity and cocommutativity for Hopf monoids, see~\cite[Section 1.2.6]{AgMahlong}.  
Under $\overline{\calK}$, (co)commutative Hopf monoids become (co)commutative Hopf algebras, so Theorem~\ref{thm:noncocomm} also implies that $\cfUT$ is a noncommutative and noncocommutative Hopf monoid.
\end{rem}

\printbibliography

\end{document}